\documentclass[a4paper,11pt]{amsart}

\usepackage[a4paper]{geometry}
\usepackage[utf8]{inputenc}
\usepackage[T1]{fontenc}

\usepackage{amssymb}
\usepackage{caption}
\usepackage[dvipsnames]{xcolor}
\usepackage{enumitem}
\usepackage{graphicx}
\usepackage{ifthen}
\usepackage{lmodern}
\usepackage{mathtools}
\usepackage[]{microtype}
\UseMicrotypeSet[protrusion]{basicmath}
\usepackage{soul}
\usepackage{stmaryrd}
\SetSymbolFont{stmry}{bold}{U}{stmry}{m}{n}
\usepackage{svg}
\usepackage{verbatim}
\usepackage{xspace}

\usepackage{hyperref}
\makeatletter
\@ifpackageloaded{hyperref}{

\newcommand\myshade{85}
\colorlet{mylinkcolor}{violet}
\colorlet{mycitecolor}{Green}
\colorlet{myurlcolor}{Aquamarine}

\hypersetup{
citecolor = mylinkcolor!\myshade!black,
linkcolor  = mycitecolor!\myshade!black,
urlcolor   = myurlcolor!\myshade!black,
linkbordercolor = mylinkcolor!\myshade!black,
colorlinks = true,
}
}{
\newcommand{\texorpdfstring}[2]{#1} \newcommand{\phantomsection}{} }

\def\namedlabel#1#2{\begingroup
  #2\def\@currentlabel{#2}\phantomsection\label{#1}\endgroup
}
\makeatother

\usepackage[initials, alphabetic]{amsrefs}
\usepackage{cleveref}

\numberwithin{equation}{section}
\newtheorem{theorem}{Theorem}[section]
\newtheorem{corollary}[theorem]{Corollary}
\newtheorem{lemma}[theorem]{Lemma}
\newtheorem{claim}{Claim}[theorem]
\newtheorem{proposition}[theorem]{Proposition}

\newtheorem*{theorem*}{Theorem}

\theoremstyle{definition}
\newtheorem{definition}[theorem]{Definition}

\newtheorem{fact}[theorem]{Fact}

\newtheorem{question}[theorem]{Question}

\theoremstyle{remark}
\newtheorem{remark}[theorem]{Remark}
\newtheorem*{claim*}{Claim}

\DeclareMathOperator{\Homeo}{Homeo}

\DeclarePairedDelimiter{\Card}{\lvert}{\rvert}
\DeclarePairedDelimiter{\set}{\{}{\}}

\newcommand{\N}{\mathbb{N}}

\newcommand{\acts}{\curvearrowright}

\newcommand{\restr}[1]{\upharpoonright #1}

\newcommand{\cB}{\mathcal{B}}

\newcommand{\cC}{\mathcal{C}}

\newcommand{\cD}{\mathcal{D}}

\newcommand{\cE}{\mathcal{E}}

\newcommand{\cI}{\mathcal{I}}

\newcommand{\cL}{\mathcal{L}}

\newcommand{\cM}{\mathcal{M}}

\newcommand{\bO}{\mathbf{O}}

\newcommand{\cP}{\mathcal{P}}

\newcommand{\cQ}{\mathcal{Q}}

\newcommand{\bbR}{\mathbb{R}}

\newcommand{\cR}{\mathcal{R}}

\newcommand{\cS}{\mathcal{S}}

\newcommand{\bU}{\mathbf{U}}

\newcommand{\cU}{\mathcal{U}}

\newcommand{\bV}{\mathbf{V}}

\newcommand{\cV}{\mathcal{V}}

\newcommand{\bW}{\mathbf{W}}

\newcommand{\cW}{\mathcal{W}}

\newcommand{\cZ}{\mathcal{Z}}

\renewcommand{\phi}{\varphi}

\newcommand{\homeo}{\mathrm{Homeo}}
\newcommand{\dom}{\mathrm{dom}}
\newcommand{\codom}{\mathrm{codom}}

\newcommand{\sub}{\subseteq}

\newcommand{\Fraisse}{Fra{\"\i}ss{\'e}\xspace}

\renewcommand{\mod}{\text{ mod}\,}

\DeclarePairedDelimiter{\walk}{\langle}{\rangle}

\newcommand{\E}{\mathbin{E}}
\newcommand{\ball}[2]{[#2]_{#1}}

\newcommand{\chains}{\Phi}
\newcommand{\runion}{\bigvee}
\newcommand{\conc}{{}^{\frown}}
\newcommand{\refines}{\preceq}
\newcommand{\rmap}[2]{\rho^{#1}_{#2}}

\newcommand{\refex}{\sqsubseteq}
\newcommand{\opwalk}[2]{\bO_{#1, #2}}

\DeclareMathOperator{\Int}{Int}
\DeclareMathOperator{\wind}{wd}
\DeclareMathOperator{\weight}{weight}
\DeclareMathOperator{\Bd}{Bd}
\DeclareMathOperator{\Cl}{Cl}
\DeclareMathOperator{\diam}{diam}
\DeclareMathOperator{\Star}{Star}
\DeclareMathOperator{\Core}{Core}
\DeclareMathOperator{\len}{lh}
\DeclareMathOperator{\Imm}{Im} 
\title{Surfaces and other Peano Continua with no Generic Chains}
\date{}

\author{Gianluca Basso}
\address{Gianluca Basso}
\email{gianluca.basso@protonmail.com}
\urladdr{https://gianlucabasso.com}

\author{Alessandro Codenotti} 
\address{Alessandro Codenotti, Dipartimento di Matematica, Universit\'{a} di Bologna, Piazza di Porta S. Donato 5, 40126 Bologna, Italy}
\email{alessandro.codenotti@unibo.it}

\author{Andrea Vaccaro}
\address{Andrea Vaccaro, Mathematisches Institut, Fachbereich Mathematik und Informatik der
Universit\"at M\"unster, Einsteinstrasse 62, 48149 M\"unster, Germany}
\email{avaccaro@uni-muenster.de}
\urladdr{https://sites.google.com/view/avaccaro}


\thanks{
  G.B. was partially supported by Istituto Nazionale di Alta Matematica
  ``Francesco Severi''.
  A.C. and A.V. were supported by the Deutsche Forschungsgemeinschaft (DFG, German Research Foundation) under Germany’s Excellence Strategy EXC 2044–-390685587, Mathematics Münster: Dynamics–Geometry–Structure, through SFB 1442. 
  A.C. was moreover supported by the ERC Starting Grant 101077154 ``Definable Algebraic Topology''.
  A.V. was moreover supported by the ERC Advanced Grant 834267--AMAREC}

\begin{document}

\begin{abstract}
	The space of chains on a compact connected space encodes all the different ways of continuously growing out of a point until exhausting the space.
	A chain is \emph{generic} if its orbit under the action of the underlying homeomorphism group is comeager.
	In this paper we show that a large family of topological spaces do not have a generic chain: in addition to all manifolds of dimension at least 3, for which the result was already known, our theorem covers all compact surfaces except for the sphere and the real projective plane – for which the question remains open – as well as all other homogeneous Peano continua, circle excluded.
	If the spaces are moreover strongly locally homogeneous, which is the case for any closed manifold and the Menger curve, we prove that chains cannot be classified up to homeomorphism by countable structures, and that the underlying homeomorphism groups have non-metrizable universal minimal flows, with all orbits meager,
	in contrast to the case of 1-dimensional manifolds.
	The proof of the main result is of combinatorial nature, and it relies on the creation of a dictionary between open sets of chains on one side, and walks on finite connected graphs on the other.
\end{abstract}
\vspace*{-0.5cm}

\maketitle

\section{Introduction}
\label{sec:Introduction}

Given a metrizable compact connected space $X$, a \emph{maximal chain of compact connected sets of $X$} – just \emph{chain}, from now on –, is a collection of compact connected subsets of $X$ which is linearly ordered by inclusion and maximal with respect to this property.
Equivalently, it is a homeomorphic image of the interval $[0,1]$ in the hyperspace $C(X)$ of compact connected subsets of $X$ such that $0$ is mapped to a singleton, $1$ to the point $X \in C(X)$ and for each $s<t$, the image of $s$ is a subset of the image of $t$ \cite{illanes1999hyperspaces}*{Lemma 14.7}.
The collection $\chains(X)$ of all such chains therefore represents all the different ways one can start, say at time $0$, from a point of $X$, and then continuously grow out of it, until, at time $1$, they have exhausted $X$.

A classification of such chains up to homeomorphisms of $X$ is only feasible for very simple spaces, such as the interval or the circle (see \Cref{sec:appendix-cicle-arc}).
At the very least, we can ask for which spaces there is \emph{essentially} only one chain, that is, whether there is $\cC \in \chains(X)$ such that its orbit under the action $\Homeo(X) \acts \chains(X)$ is comeager in $\chains(X)$, endowed with its natural compact hyperspace topology.
In such case, we say that $\cC$ is a \emph{generic chain}.

Focusing on compact manifolds, the chain on the circle $S^1$ consisting of all intervals centered around some point $x_0$ is generic in $\chains(S^1)$, and a similar statement holds for the closed interval (see \Cref{prop:arc-has-generic-chains} and \Cref{thm:genericS1}). On the other hand, Gutman, Tsankov and Zucker proved in \cite{GTZ} that no closed manifold of dimension at least $3$ has generic chains, and nor does the Hilbert cube.
The case of surfaces was left open. A corollary of our main theorem gives the following.
\begin{theorem} \label{thm:surfaces}
	If $X$ is a compact surface other than the sphere or the real projective plane, then $X$ has no generic chain.
\end{theorem}

We obtain \Cref{thm:surfaces} as consequence of a stronger statement which not only recovers the aforementioned result of \cite{GTZ}, but which moreover covers a vast class of Peano continua going well
beyond the context of manifolds.

A \emph{Peano continuum} is a metrizable compact connected space which is locally connected.
Among notable examples which are not manifolds are the Sierpi\'{n}ski carpet and the Menger curve. These are universal one-dimensional continua obtained by iteratively carving out the square and the cube, respectively, similarly to how the Cantor space is obtained from the unit interval.

Before stating our main theorem, we introduce some basic terminology.
A subset $K \sub X$ is \emph{locally separating} if $U \setminus K$ is disconnected for some connected open $U \sub X$.
A space is \emph{planar} if it admits an embedding in $\bbR^2$, and it is \emph{locally non-planar} if none of its nonempty open sets are planar.
Recall that a \emph{simple closed curve} in $X$ is the image of an embedding of the circle $S^1$ into $X$.
Let us say that a Peano continuum $X$ \emph{has a circular covering} if it admits a covering $O_0, \dots, O_{\ell-1}$ in $\ell\ge 4$ \emph{connected} open subsets such that $\Cl(O_i) \cap \Cl(O_j) \neq \emptyset$ if and only if $\Card{i-j} \le 1 \mod \ell$, and $O_i \setminus \bigcup_{j\ne i} O_j$ is connected for each $i<\ell$.

Our main result is the following.

\begin{theorem}
	\label{thm:the-theorem-introduction}
	Let $X$ be a Peano continuum with no locally separating points.
	If $X$ either
	\begin{enumerate}
		\item \label{itm:non-planar} has a locally non-planar open subset,
		\item \label{itm:nls-curve} has a planar open set containing a simple closed curve which is not locally separating, or
		\item \label{itm:circular} has a circular covering,
	\end{enumerate}
	then $X$ has no generic chain.
\end{theorem}

Let us parse the above hypotheses for $X$ a compact manifold (we refer to \cite{MR2954043} for background and definitions).
A compact manifold has a locally separating point if and only if it is $1$-dimensional, and it has a locally non-planar open subset if and only if it is at least $3$-dimensional.
A boundary of a surface is a simple closed curve which is not locally separating and lies in a planar open set.
Finally, the sphere and the real projective plane are the only closed surfaces without a circular covering (see \Cref{prop:vaguely-circular-surfaces}).

We prove \Cref{thm:the-theorem-introduction} by first establishing a correspondence between Peano continua and open sets of chains on one side, and finite connected graphs and walks on these graphs on the other.
We then find a combinatorial necessary condition for the existence of a generic chain on a Peano continuum (\Cref{thm:from-Rosendal-to-walks}), and prove that it is does not hold under our assumptions.

The combinatorial condition is an \emph{off-by-one} weak amalgamation principle, roughly stating that each walk can be refined by another walk in such a way that any two different walks refining it can be reconciled by a small perturbation.
We extract the combinatorial idea at the core of the main proof of \cite{GTZ} – that chains can be made irreconcilable by winding enough times in opposite directions – and discretize it, producing irreconcilable walks which wind around circular subgraphs.

The discretization relies on a handful of classical continuum-theoretical black-boxes, all of which are stated in \Cref{sec:Brick}.
In the following two sections, these black-boxes are used to derive fundamental combinatorial statements about graphs and walks on graphs.
The proof of \Cref{thm:the-theorem-introduction} is purely combinatorial.

This approach to the study of topological and dynamical properties of compact metrizable spaces by combinatorial means traces its roots to the field of \emph{projective \Fraisse theory}, first introduced in \cite{Irwin2006}.
In rough terms, one constructs an infinite compact zero-dimensional graph, the \emph{combinatorial model} of $X$, and finite approximations of such graph, and studies them to derive properties about $X$, leaning on the dictionary between the combinatorics of countable structures and the dynamics of Polish non-archimedean groups put forth in the seminal paper \cite{KPT}.
The aforementioned Menger curve was studied under this lens in \cite{PS22}, and
it follows from the results in \cite{BK:Lelek}, again obtained in a projective \Fraisse theoretical framework, that the Lelek fan, a one-dimensional continuum which is not locally connected, has a generic \emph{downward closed} chain.

Nonetheless, our strategy is innovative in that it is the first instance in which non-trivial \emph{necessary} conditions for the existence of a comeager orbit for an action of a homeomorphism group of a compact connected space have been translated to a combinatorial statement.
Such statement is a weakening of the weak amalgamation principle which was isolated as a criterion for the existence of comeager orbits in certain actions of automorphism groups of countable structures by Ivanov \cite{MR1777786}, and independently by Kechris and Rosendal \cite{MR2308230}.
In this paper, we restrict ourselves to finite combinatorial objects, without explicit mention of the combinatorial models of the spaces involved, though our proofs could also be used to derive the analogous statements for such combinatorial models and their groups of automorphisms.

\Cref{thm:the-theorem-introduction} has important dynamical consequences for the groups of homeomorphisms involved.
If $G$ is a topological group, we call a continuous action $G \acts Y$ on a compact space a \emph{$G$-flow}.
A $G$-flow is \emph{minimal} if every orbit is dense.
Both classifying minimal flows of a group $G$ and, vice versa, classifying groups based on the properties of their minimal flows are central goals in \emph{topological dynamics}.
Some groups of homeomorphisms have very simple minimal dynamics: the only minimal flow of the group $\Homeo_+([0,1])$ of orientation preserving homeomorphisms of the interval is the trivial action and the only non-trivial minimal flow of $\Homeo_+(S^1)$ is the action on the circle itself (see \cite{Pestov}).

A number of dividing lines regarding the complexity of the minimal dynamics of a topological group have been studied, and are particularly well understood when $G$ is Polish, i.e., separable and completely metrizable.
A Polish group is \emph{extremely amenable} if all of its minimal flows are trivial, it is \emph{CAP} if they are metrizable (see \cite{CAP}), and it has the \emph{generic point property} if every minimal flow has a comeager orbit (see \cite{AKL}).
Every extremely amenable group is CAP, every Polish CAP group has the generic point property (see \cite{ben2017metrizable}), and no locally compact non-compact group has the generic point property.
Such properties can also be stated in terms of the \emph{universal minimal flow} of $G$: CAP groups are also known as groups \emph{with metrizable universal minimal flow.}

Until recently (see \cite{kaleido}), there was essentially only one known example of Polish group with the generic point property which is not CAP (see \cite{kwiat}).
In \cite{GTZ} it is proved that the subgroup $\Homeo_0(X)$ of homeomorphisms isotopic to the identity is not CAP, for any closed surface $X$, but it was left open whether such groups have the generic point property.
Since for a closed surface $X$ the flow $\Homeo_0(X) \acts \chains(X)$ is minimal (\cite{gutman_minimal}), \Cref{thm:surfaces} allows us to cover all but two of these cases.
\begin{theorem} \label{thm:no_gpp}
	If $X$ is a closed surface which is not the sphere or the real projective plane, then $\Homeo_0(X)$ does not have the generic point property.
\end{theorem}

The following fascinating questions remain open.

\begin{question}[\cite{GTZ}*{Question 1.3}]
	\label{que:sphere}
	Is there a generic chain on the sphere $S^2$? Does $\Homeo(S^2)$ have the generic point property? What about the real projective plane?
\end{question}

Let us end this introduction with a series of corollaries of \Cref{thm:the-theorem-introduction}.
A space $X$ is \emph{homogeneous} if the action $\Homeo(X) \acts X$ is transitive.
Homogeneous Peano continua are either the circle, or a closed surface, or are locally non-planar (see \cite{MR2674033}*{Remark 8}).
\begin{corollary}
	\label{cor:homo-chains}
	No homogeneous Peano continuum except for the circle and, possibly, the sphere and real projective plane, has a generic chain.
\end{corollary}
An example of a homogeneous Peano continuum is the Menger curve, which is the unique one-dimensional locally non-planar Peano continuum without locally separating points (see \cite{MR0096180}).
Another is the Hilbert cube, for which the non-existence of generic chain was already proved in \cite{GTZ}.

On the other hand, the Sierpi\'{n}ski carpet, which is the unique one-dimensional planar Peano continuum without locally separating points (see \cite{whyburn}), is not homogeneous.
Nonetheless, it has plenty of non-locally-separating simple closed curves, since it can be constructed as the complement, in the sphere, of countably many pairwise disjoint open disks whose union is dense in the sphere itself.
In particular, any open set is planar and contains a non-locally-separating simple closed curve, so  the Sierpi\'{n}ski carpet falls within the scope of \Cref{thm:the-theorem-introduction}.

A one-dimensional Peano continuum with no locally separating points is either the Menger curve or has a nonempty open connected subset homeomorphic to an open subset of the Sierpi\'{n}ski carpet (see \cite{MR1314981}*{Lemma 1.1}). Therefore
\begin{corollary}
	\label{cor:one-dimensional-no-generic-chain}
	One-dimensional Peano continua without locally separating points do not have a generic chain.
\end{corollary}

We prove in \Cref{thm:local-transitive} that a sufficient condition for minimality of the action $\Homeo(X) \acts \chains(X)$, when $X$ is a Peano continuum without locally separating points, is for $X$ to be \emph{strongly locally homogeneous}, see \Cref{sec:Minimal_spaces_of_chains_and_dynamical_consequences} for the definition. For example, all closed manifolds of dimension at least $2$ are strongly locally homogeneous.
Since these spaces are in particular homogeneous, they have no generic chain, by \Cref{cor:homo-chains}, except possibly if they are the sphere or the real projective plane. Therefore:
\begin{corollary} \label{corollary:no_gpp}
	If $X$ is a strongly locally homogeneous Peano continuum which is not the circle, the sphere, or the real projective plane, then $\Homeo(X)$ does not have the generic point property.
	In particular, its universal minimal flow is not metrizable and has no comeager orbit.
\end{corollary}

Among other examples of spaces which fit the hypotheses of \Cref{corollary:no_gpp} are the Menger curve, the universal $k$-dimensional compactum $\mu^k$, and any compact connected $\mu^k$-manifold, for $k\ge 1$ (see \cite{MR0920964}*{Theorem 3.2.2}).

For the spaces in \Cref{corollary:no_gpp}, we are also able to prove that the action $\Homeo(X) \acts \chains(X)$ is \emph{generically turbulent}, which implies a strong non-classifiability result for the orbit equivalence relation of the action.
We say that chains on $X$ are \emph{classifiable by countable structures} if there is a Borel function $\cM \colon \chains(X) \to \mathrm{Mod}(\cL)$ to the space of structures with domain $\N$ in some countable relational language $\cL$ such that
\begin{equation*}
	\cC' \in \Homeo(X) \cdot \cC \Longleftrightarrow \cM(\cC) \cong \cM(\cC'),
\end{equation*}
that is, if we can reduce the orbit equivalence relation among chains to the isomorphism relation between the associated structures.
By Hjorth's theorem (see \cite{hjorth2000classification}) we obtain:

\begin{corollary} \label{corollary:no_classification}
	If $X$ is a strongly locally homogeneous Peano continuum which is not the circle, the sphere, or the real projective plane, then chains on $X$ are not classifiable by countable structures.
\end{corollary}

This is in contrast with the case of $1$-dimensional manifolds, whose chains are classifiable by countable structures and have a generic element, as we prove in \Cref{sec:appendix-cicle-arc}.
The techniques in this appendix are distinct from those in the other sections.

We remark that the first part of \Cref{que:sphere} could be answered in the positive by proving that chains on the sphere are classifiable by countable structures, but we believe that this is not the case.

In \Cref{sec:appendix-sierpinski}, we prove that $\Homeo(S) \acts \chains(S)$ is minimal, for the Sierpi\'{n}ski carpet $S$, and thus that $\Homeo(S)$ does not have the generic point property, and its universal minimal flow is not metrizable and has no comeager orbit.
The proof is rather lengthy and technical; it could foreseeably be simplified by finer results on the degree of transitivity of the diagonal actions $\Homeo(S) \acts S^n$, which we were not able to find in the literature.
We leave the question of classifiability by countable structures of chains on the Sierpi\'{n}ski carpet open.

\subsection{Plan of the paper}

Sections \ref{S.pc} and \ref{sec:walks-on-graphs} are devoted to preliminaries on chains and walks on graph, respectively.
In section \ref{sec:Brick} we recall Bing's definition of brick partitions of Peano continua and state Bing's theorem – which serves as our main discretization tool – and other relevant facts about them.

In section \ref{sec:Walks on Brick Partitions and their Refinements} we construct the dictionary between open sets of chains of Peano continua and walks on graphs, and prove the required discretization results to establish the combinatorial necessary condition for the existence of a generic chain, which is \Cref{thm:from-Rosendal-to-walks}.
Section \ref{sec:The_proof} is the heart of the paper and contains the proof of \Cref{thm:the-theorem-introduction}, which we obtain as a sum of Theorems \ref{thm:the-theorem}, \ref{thm:local-hp-robust-cycle0}, \ref{thm:local-hp-robust-cycle1}, and \ref{thm:circular-covering-robust}, whose proofs are completely combinatorial.
Section \ref{sec:Minimal_spaces_of_chains_and_dynamical_consequences} is devoted to the question of minimality and generic turbulence of $\Homeo(X)\acts \chains(X)$, and their consequences for classifiability of chains and topological dynamics of $\Homeo(X)$.

Finally, appendices \ref{sec:appendix-cicle-arc} and \ref{sec:appendix-sierpinski} are devoted to the cases of $1$-dimensional manifolds, and of the Sierpi\'{n}ski carpet, respectively.

\subsection*{Acknowledgements}
We are grateful to Riccardo Camerlo, Aleksandra Kwiatkowska, Aristotelis Panagiotopoulos, and Todor Tsankov for helpful discussions at different stages of the project.
We wish to thank the anonymous referees for their detailed and insightful comments.

\section{Peano Continua and Chains} \label{S.pc}
In this preliminary section we briefly introduce Peano continua and their spaces of maximal connected chains. We refer to \cites{nadler2017continuum, illanes1999hyperspaces} for a more thorough presentation of these topics.

A \emph{continuum} is a compact connected metrizable space. A \emph{Peano continuum} is a locally connected continuum. An \emph{arc} is a homeomorphic image of the interval $[0,1]$; a \emph{continuum theoretical graph} is a space composed of finitely many arcs that intersect each other at most in their endpoints; a \emph{continuum theoretical tree} is a connected continuum theoretical graph that contains no simple closed curve.

We refer the reader to \cite{nadler2017continuum}*{Chapter VIII} for the proofs of the following basic properties of Peano continua which will be implicitly used throughout the rest of the paper.
\noindent\begin{enumerate}
	\item Every Peano continuum is arcwise connected. \item Every open connected subset of a Peano continuum is arcwise connected. \item Whenever $U$ is an open connected subset of a Peano continuum and $F\subseteq U$ is a finite set of points, there exists a continuum theoretical tree $T\subseteq U$ with $F\subseteq T$.
\end{enumerate}

Given a topological space $X$ and $Y\subseteq X$ we denote by $\Cl(Y)$, $\Int(Y)$ and $\Bd(Y)$ respectively the closure, interior and boundary of $Y$.

Given a family $\cW$ of open subsets of a space $X$ and $C \subseteq X$, we define
\begin{equation*}
	\Star(C, \cW) \coloneqq \{ W \in \cW : \Cl(W) \cap C \not = \emptyset \}.
\end{equation*}
It is generally more common to define $\Star(C,\cW)$ as the set of those $W\in\cW$ that meet $C$, without taking their closures first. Our definition is slightly different, but it will simplify the notation throughout the paper.

If $V$ is an open set, let
\begin{equation*}
	\cW(V) \coloneqq \{ W \in \cW : W\subseteq V\},
\end{equation*}
and
\begin{equation*}
	\Core(V, \cW) \coloneqq \{ W \in \cW : \Cl(W) \subseteq V \}.
\end{equation*}
Note that $\Core(V, \cW) = \cW(V) \setminus \Star(\Bd(V), \cW)$.

An open set $U$ is \emph{regular} if $U = \Int(\Cl(U))$.
If $U, V \sub X$ are regular open sets, let $U \vee V$ be the regular open set $\Int(\Cl(U) \cup \Cl(V))$ and, similarly, let $\runion \cU \coloneqq \Int(\Cl(\bigcup_{U \in \cU}U))$, for any family $\cU$ of regular open sets.

Let $\Homeo(X)$ denote the set of all homeomorphisms from $X$ onto itself. The \emph{compact-open topology} on $\Homeo(X)$ is the topology generated by sets of the following form, with $K$ being a compact subset of $X$ and $U$ an open one:
\begin{equation*}
	\{f\in\Homeo(X) : f[K]\subseteq U\}.
\end{equation*}
If $X$ is a compact metrizable space, the compact-open topology on $\Homeo(X)$ agrees with the topology of uniform convergence and turns $\Homeo(X)$ into a Polish topological group.

\subsection{The Space of Maximal Connected Chains} \label{ss.chains}

Given a topological space $X$, let $K(X)$ denote the set of compact subspaces of $X$. The \emph{Vietoris topology} on $K(X)$ is generated by sets of form $\set{K \in K(X) : K \sub U}$ and $\set{K \in K(X) : K \cap U \ne \emptyset}$, for $U$ open in $X$.
If $X$ is compact, the Vietoris topology on $K(X)$ is also compact and the subspace $C(X)\subseteq K(X)$ of \emph{subcontinua}, i.e. compact connected sets, is closed (hence compact itself).
If $X$ is moreover metrizable, say by a metric $d$, then the associated \emph{Hausdorff metric} on $K(X)$, defined by
\begin{equation*}
	d_H(K, K') \coloneqq \max \set*{\sup_{x \in K} d(x,K'), \sup_{x' \in K'} d(K,x')},
\end{equation*}
is compatible with the Vietoris topology.
Given $U_0, \dots, U_{n-1}$ \emph{regular} open sets, we define the open set:
\begin{equation*}
	O(U_0, \dots, U_{n-1}) \coloneqq \left\{ K \in K(X) : K \cap U_i \not = \emptyset \text{ for every } i < n \text{ and } K \subseteq \runion_{i<n} U_i \right\}.
\end{equation*}

A collection $\cC$ of compact subsets of $X$ is called a \emph{chain} (a \emph{nest} in \cite{illanes1999hyperspaces}) if it is linearly ordered by inclusion, that is, if for every $C_1,C_2\in\cC$, either $C_1\subseteq C_2$ or $C_2 \subseteq C_1$. A chain $\cC\subseteq K(X)$ is \emph{maximal} if it is not properly contained in any other chain, and it is \emph{connected} if it consists of connected sets.

If $X$ is compact and $\cC\subseteq K(X)$ is a maximal chain, then $\cC$ is compact (as a subspace of $K(X)$), hence it can be thought of as a point in $K(K(X))$.
If moreover $X$ is a continuum, every
maximal connected chain is the image of a topological embedding $[0,1]\to C(X)$, so it is in particular an element of $C(C(X))$ (see \cite{illanes1999hyperspaces}*{Lemma 14.7}).
In this case, we define the space $\chains(X)$ of \emph{maximal connected chains on $X$} as
\begin{equation*}
	\chains(X)\coloneqq\set{\cC\in C(C(X)) : \cC\text{ is a maximal connected chain}},
\end{equation*}
which is a compact subset of $C(C(X))$ (see \cite{uspenskij2000universal}*{\S2}).
The map $\chains(X) \to X$ that to each chain $\cC$ associates the unique element of the singleton $\bigcap \cC$ – the so-called \emph{root} of $\cC$ – is continuous.
Given a compatible metric $d$ on $X$, we denote by $d_\chains$ the Hausdorff metric on $\chains(X)$ relative to $d_H$.

\medskip

\emph{From now on, all chains we consider are maximal and connected.}

\medskip

The action $\Homeo(X) \acts X$ naturally induces a continuous action $\Homeo(X) \acts \chains(X)$ defined as $h\cdot\cC \coloneqq \{h[C] : C\in\cC\}$, for $h\in\Homeo(X)$ and $\cC\in\chains(X)$.
We say that a chain $\cC \in \chains(X)$ is \emph{generic} if $\Homeo(X)\cdot \cC$ is comeager in $\chains(X)$.

\section{Walks on Graphs}
\label{sec:walks-on-graphs}

\textit{Throughout this paper, the word graph refers to finite connected reflexive undirected graphs.}

\medskip
In this section we present definitions and terminology on graphs and walks, as well as some preliminary results on the notion of winding number, which is introduced in \Cref{def:wd-for-any-walk}.

Let $(G,E)$ be a graph.
We denote by $d_G$, or just $d$ if there is no risk of confusion, the \emph{path metric} on $G$.
If $x \in G$ and $A \sub G$, we write $x \E A$ if there is $y \in A$ with $x \E y$, and we denote by $\ball{n}{A} \coloneqq \set{x \in G : d_G(x, A) \le n}$ the closed ball of $d_G$-radius $n$ around $A$.
If $A = \set{x}$, we write $\ball{n}{x}$.

If $f, g \colon (G,E) \to (H,E')$ are functions between the same two graphs, we denote by $d_{\sup}(f, g) \coloneqq \max\set{d_H(f(x), g(x)) : x \in G }$ their distance according to the supremum metric.

A function $f \colon (G,E) \to (H,E')$ is a \emph{monotone epimorphism} (also called \emph{connected epimorphism} in \cite{PS22}) if it is a homomorphism of graphs (that is, $f(x)\E' f(y)$ whenever $x \E y$) such that the preimage of each nonempty connected subset of $H$ is nonempty and connected in $G$.
Equivalently, it is a homomorphism of graphs which is surjective on edges and such that the preimage of each vertex of $H$ is nonempty and connected (see \cite{PS22}*{Lemma 1.1}).
In particular, monotone epimorphisms are onto.

A \emph{walk} on $(G,E)$ is a sequence of vertices $w = \walk{ w(0), \dots, w(\ell-1)}$ such that $w(i) \E w(i+1)$ for every $i < \ell-1$.
We say that $\ell$ is the \emph{length} of $w$ and denote it by $\len(w)$.
We admit the empty walk, which has length $0$.
We denote by $\Imm(w) \coloneqq \set{w(i) : i < \len(w)}$.
Adopting a common programming language convention, let us write $w(-1)$ in lieu of $w(\len(w)-1)$ for the last element of $w$.
With a slight abuse of notation, we shall write $\ball{n}{w}$
for the set $[\Imm(w)]_n$.

If $f \colon (G,E) \to (H,E')$ is a graph homomorphism and $w$ is a walk on $(G, E)$, let $f \cdot w$ be the induced walk $\walk{f(w(0)), \dots, f(w(-1))}$ on $(H, E')$.

A walk $w$ is
\begin{enumerate}
	\item
	      \emph{reduced} if $w(i) \ne w(i+1)$ for every $i < \len(w)-1$;
	\item
	      \emph{uncrossed} if whenever $i, j < \len(w)-1$ are such that $w(i) = w(j)$, then $w(i+1) = w(j+1)$;
	\item
	      a \emph{path} if $w(i) = w(j)$ if and only if $i = j$;
	\item
	      a \emph{spaced path} if $w(i) \E w(j)$ if and only if $| i - j | \le 1$ for every $i,j < \len(w)$.
\end{enumerate}

If $A, B \sub C \sub G$ are nonempty subsets, a \emph{path from $A$ to $B$ in $C$} is a path $w$ on $G$ such that $w(0) \in A, w(-1) \in B$ and $w(i) \in C \setminus (A \cup B)$ for all $0<i<\len(w)-1$.
Such a path always exists if $C$ is connected.
If $w$ is a path, we refer to $w(0), w(-1)$ as its \emph{endpoints}.

A walk is a \emph{circular path} if it has length at least $3$, $w(i) \ne w(j)$ for each $i < j <\len(w)-1$, and $w(-1) = w(0)$.
We modify the definition of spaced for circular paths accordingly.
We also define the \emph{circular length} of a circular path $w$ as $\len(w)-1$.

Given two walks $w, v$ on $(G, E)$, we say that $v$ is \emph{confined} in $w$ if $\Imm(v) \sub \Imm(w)$,
and that $v$ is \emph{confined} in some $A \subseteq G$ if $\Imm(v) \sub A$.
We say that $w$ \emph{refines} $v$, in symbols $w \refex v$, if:
\begin{center}
	for every $s \le \len(v)$, there is $t \le \len(w)$ such that $\set{w(i) : i < t} = \set{v(j) : j <s}$.
\end{center}
In words: $w$ starts on the same vertex as $v$, it then visits each \emph{new} vertex that $v$ visits, in the same order, possibly backtracking from time to time, until it covers all that has been visited by $v$, after which it is allowed to roam free.

We say that $w$ \emph{monotonically} refines $v$ if there is a strictly increasing function $f \colon \len(v) \to \len(w)$ with $f(0)=0$ and such that
\begin{center}
	$w(i) = v(j)$ for all $f(j) \le i < f(j+1)$ when $j < \len(v)-1$,
\end{center}
and $w(i) = v(-1)$ for $i\ge f(\len(v)-1)$.
If $w$ monotonically refines $v$, then it refines $v$, and $w,v$ are confined to one another.

Given two walks $w, v$ on $(G, E)$ such that $w(-1) \E v(0)$, let $w \conc v$ be the walk obtained by the \emph{concatenation} of $w$ and $v$.
Note that each reduced uncrossed walk $w$ admits a \emph{unique} decomposition $v \conc c \conc \cdots \conc c \conc s$, where $v$ is a path, $c$ is a circular path, and $s$ is an initial segment of $c$ ($v, c$ or $s$ could be empty).

Finally, we will say that a reduced uncrossed walk is a \emph{lasso} if its decomposition is $v \conc c$, for a path $v$ and a nonempty circular path $c$.

Monotone maps allow lifting walks, as we see in the following lemma.

\begin{lemma}[Monotone Lifting]
	\label{lem:lemmetto-monotone}
	Let $f \colon (G,E) \to (H,E')$ be a monotone epimorphism between graphs, and let $w$ be a reduced walk on $(H, E')$.
	Then there is a reduced walk $v$ on $(G, E)$ such that $f \cdot v$ monotonically refines $w$.
	Moreover, $v$ can be chosen to be uncrossed or a (spaced/circular/spaced circular) path whenever $w$ is.
\end{lemma}
\begin{proof}
	Denote $A_i \coloneqq f^{-1}(w(i))$, for each $i < \len{w}$.
	By hypothesis, they are all connected subsets of $G$.
	Let $P_1$ be a path from $A_1 \cap \ball{1}{A_{0}}$ to $A_1 \cap \ball{1}{A_{2}}$ in $A_1$, of minimal length.
	Since $P_1$ is of minimal length, it is spaced, and any $P_1(j)$ which is not an endpoint of $P_1$ has distance at least $2$ from $A_{0}$ and $A_{2}$.

	Suppose we are given a path $P_i$ in $A_i$, for some $1 \le i < \len(w)-2$, such that $P_i(-1) \E A_{i+1}$ and such that any $P_i(j)$ which is not an endpoint of $P_i$ has distance at least $2$ from $A_{i+1}$.
	Let $P_{i+1}$ be a path from $A_{i+1} \cap \ball{1}{P_i(-1)}$ to $A_{i+1} \cap \ball{1}{A_{i+2}}$ in $A_{i+1}$, of minimal length.
	Then $P_{i+1}$ is spaced, it is nonempty because $w$ is reduced, and
	\begin{equation}
		\label{eq:far-away}
		d_G(P_{i+1}(j), P_i) \ge 2, \text{ for each } 0< j < \len(P_{i+1})-1.
	\end{equation}

	Let $x_0 \in A_0$ be adjacent to $P_1(0)$ and $x_{-1} \in A_{\len(w)-1}$ be adjacent to $P_{\len(w)-2}(-1)$.
	Then the walk $v \coloneqq \langle x_0 \rangle \conc P_1 \conc \cdots P_{\len(w)-2} \conc \langle x_{-1}\rangle$ is reduced and such that $f\cdot v$ monotonically refines $w$, and it is a path whenever $w$ is.

	We prove that $v$ is spaced whenever $w$ is.
	Let $i_0, j_0 < \len(v)$ with $| i_0 - j_0 | > 1$ be given, and let $i, j < \len(w)$ be such that $f(v(i_0)) = w(i)$ and $f(v(j_0)) = w(j)$.
	If $i = j$, then $v(i_0), v(j_0) \in P_i$, which is spaced, so $d_G(v(i_0), v(j_0))>1$.
	If $| i - j | > 1$, then $d_H(w(i), w(j))>1$, since $w$ is spaced, so also $d_G(v(i_0), v(j_0))>1$.
	Otherwise, suppose without loss of generality that $j = i +1$.
	We conclude by \eqref{eq:far-away}.

	\textbf{If $w$ is a circular path:} let $P_0$ be a path from $A_{0} \cap \ball{1}{P_{\len(w)-2}(-1)}$ to $A_0 \cap \ball{1}{P_1(0)}$ in $A_0$, of minimal length.
	As above, $P_0$ is spaced and any $P_0(j)$ which is not an endpoint has distance at least $2$ from $P_{\len(w)-2}$ and $P_1$.
	The walk $v \coloneqq P_0 \conc P_1 \conc \cdots P_{\len(w)-2} \conc \langle P_0(0)\rangle$ is a circular path such that $f\cdot v$ monotonically refines $w$, which is spaced whenever $w$ is.

	\textbf{If $w$ is uncrossed:} let $w = w_0 \conc c \conc \dots \conc c \conc s$, with $w_0$ a path, $c$ a circular path, and $s$ an initial segment of $c$, be its decomposition.
	Apply the previous constructions to $w_0$ to obtain a path $v_0$ in $G$, and to $c$ to obtain a circular path $c'$ in $G$.
	Let $s'$ be the shortest initial segment of $c'$ such that $f \cdot s'$ monotonically refines $s$.

	Let $A \coloneqq f^{-1}(w_0(-1))$ and $B \coloneqq f^{-1}(c(0))$, so that $A \E B$.
	Let $Q_0$ be a path from $A \cap \ball{1}{v_{0}(-1)}$ to $A \cap \ball{1}{B}$ in $A$, and $Q_1$ be a path from $B \cap \ball{1}{Q_{0}(-1)}$ to $B \cap \ball{1}{c'}$ in $B$ ($Q_0$ and/or $Q_1$ could be empty).
	Then $v \coloneqq v_0 \conc Q_0 \conc Q_1 \conc c' \conc \dots \conc c' \conc s'$ is an uncrossed walk such that $f\cdot v$ monotonically refines $w$.
\end{proof}

\subsection{The Winding Number of a Walk}

\begin{definition}
	\label{def:wd-for-any-walk}
	Let $(G, E)$ be a graph
	and $C$ be a circular path on $G$ of circular length $\ell$.
	For each ordered pair of vertices $(x,y)$ in $G$, we define its \emph{weight} relative to $C$ to be:

	\begin{equation*}
		\weight_C(x,y) \coloneqq \begin{cases}
			1  & \text{if there is } i<\ell \text{ such that } x = C(i),\, y = C(i+1 \mod \ell)  \\
			-1 & \text{if there is } i<\ell \text{ such that } y = C(i), \, x = C(i+1 \mod \ell) \\
			0  & \text{otherwise.}
		\end{cases}
	\end{equation*}
	Notice that $\weight_C$ is well-defined since $C$ is a path.
	For any walk $w$ on $G$, we define its \emph{winding number around $C$} as
	\begin{equation*}
		\wind_C(w) \coloneqq \sum_{i<\len(w)-1} \weight_C(w(i), w(i+1)).
	\end{equation*}
\end{definition}

We devote the end of the section to a few easy lemmas on walks and their winding numbers which will be applied later in the paper.
The impatient reader is advised to skip them and return to consult them during the proof of \Cref{thm:the-theorem}. The keen reader, on the other hand, is encouraged to compare Lemma \ref{lem:wd-refine-initial-segment} with \cite{GTZ}*{Lemma 2.12} and Lemma \ref{lem:close-walks} with \cite{GTZ}*{Lemma 2.10}.
For ease of reference, the graphs are called $\cS$ and $\cW$ rather than $(G, E)$ or $(H, E')$, as elsewhere in this section.

\begin{lemma}
	\label{lem:wd-monotonically-refine}
	Let $\cS$ be a graph
	and $C$ be a circular path on $\cS$.
	Let $v$ and $w$ be walks on $\cS$ such that $w$ monotonically refines $v$.
	Then $\wind_C(w) = \wind_C(v)$.
\end{lemma}
\begin{proof}
	By assumption there is a strictly increasing function $f \colon \len(v) \to \len(w)$ with
	\begin{equation*}
		f(0)=0 \text{ and such that } w(i) = v(j) \text{ for all } f(j) \le i < f(j+1).
	\end{equation*}
	In particular
	$\weight_C(w(i), w(i+1)) =0$
	whenever there is $j< \len(v)$ with $f(j) \le i < i+1 < f(j+1)$, since in this case $w(i)= w(i+1)$.
	Therefore:
	\begin{align*}
		\wind_C(w) & = \sum_{j < \len(v) -1} \weight_C(w(f(j)), w(f(j +1))) \\
		           & = \sum_{j < \len(v) -1} \weight_C(v(j), v(j+1))        \\
		           & = \wind_C(v).
	\end{align*}
\end{proof}

\begin{lemma}
	\label{lem:wd-refine-initial-segment}
	Let $\cS$ be a graph and $C$ be a circular path on $\cS$.
	Let $\cW$ be a graph and $\alpha \colon \cW \to \cS$ be a graph homomorphism.
	Let $w$ be a spaced path on $\cW$ and $z$ be a walk on $\cW$ which refines and is confined to an initial segment of $w$.
	Then
	\begin{equation*}
		\wind_C(\alpha \cdot z) = \wind_C(\alpha \cdot \walk{w(0), \dots, w(j)}),
	\end{equation*}
	where $j < \len(w)$ is the unique index such that $w(j) = z(-1)$.
\end{lemma}
\begin{proof}
	We proceed by induction on the length $\ell$ of $z$.
	If $\ell = 1$, then
	\begin{equation*}
		\wind_C(\alpha \cdot z) = 0 = \wind_C(\alpha \cdot \walk{w(0)}).
	\end{equation*}
	For $\ell>1$, suppose that the statement is true for $\walk{z(0), \dots, z(\ell-2)}$.
	Suppose $z(\ell-2) = w(i)$, for some $i<\len(w)$.
	Since $z$ refines and is confined to an initial segment of $w$, which is a spaced path, it must be that $z(\ell-1)$ is equal to either $w(i-1) , w(i)$, or $w(i+1)$.
	In each case, a straightforward computation gives the desired result.
\end{proof}

\begin{lemma}
	\label{lem:close-walks}
	Let $\cS$ be a graph and $C$ be a spaced circular path on $\cS$ of circular length $\ell \ge 4$.
	Let $z, z'$ be walks of the same length $k$ on $\cS$, which are confined to $C$ and such that, for each $i<k$, $d_{\cS}(z(i), z'(i)) \le 1$, and also, for $i<k-1$, $d_{\cS}(z'(i), z(i+1)) \le 1$.
	Then $\Card{\wind_C(z)- \wind_C(z')}\le 2$.
\end{lemma}
\begin{proof}
	We prove that under the hypotheses:
	\begin{equation*}
		\wind_C(z') = \weight_C(z'(0), z(0)) + \wind_C(z) + \weight_C(z(-1), z'(-1)).
	\end{equation*}
	{
	The conclusion then follows immediately, since the values $\weight_C(z'(0), z(0))$ and $\weight_C(z(-1), z'(-1))$ have absolute values at most $1$.}

	We show by induction that for each $i<k$:
	\begin{multline}
		\label{eq:induction}
		\wind_C(\walk{z'(0), \dots, z'(i)}) = \weight_C(z'(0), z(0)) + \\ + \wind_C(\walk{z(0), \dots, z(i)}) + \weight_C(z(i), z'(i)).
	\end{multline}
	If $i=0$, then $\weight_C(z'(0), z(0)) = -\weight_C(z(0), z'(0))$, so \eqref{eq:induction} holds.

	Suppose \eqref{eq:induction} is true for $i< k-1$ and let us prove it for $i+1$.
	By definition of $\wind_C$ it suffices to verify that
	\begin{multline*}
		\weight_C(z'(i), z'(i+1)) = \weight_C(z(i), z(i+1)) +\\ + \weight_C(z(i+1), z'(i+1)) - \weight_C(z(i), z'(i)).
	\end{multline*}
	But $- \weight_C(z(i), z'(i)) = \weight_C(z'(i), z(i))$ and, using that $z'(i), z(i+1)$ lie in $C$ at distance at most $1$ we get
	\begin{equation*}
		\weight_C(z'(i), z(i)) + \weight_C(z(i), z(i+1)) = \weight_C(z'(i), z(i+1)).
	\end{equation*}
	Also $z'(i), z'(i+1)$ are at distance at most $1$, so
	\begin{equation*}
		\weight_C(z'(i), z(i+1)) + \weight_C(z(i+1), z'(i+1)) = \weight_C(z'(i), z'(i+1)),
	\end{equation*}
	and we are done.
\end{proof}

\section{Brick Partitions}
\label{sec:Brick}

In this section we summarize some well-known definitions and facts about brick partitions of Peano continua which will be of use in the next sections.

A \emph{quasi-partition} of a space $X$ is a finite family $\cW$ of pairwise disjoint open subsets whose union is dense in $X$.
If $d$ is a metric on $X$, the \emph{mesh} of a quasi-partition $\cW$, denoted $\mathrm{mesh}_d(\cW)$, is the maximum of the diameters of its elements.
Each quasi-partition $\cW$ gives rise to a finite graph $(\cW, E)$ – its \emph{nerve graph} –, namely the graph whose vertex set is $\cW$
and such that there is an edge $W \E W'$ if and only if $\Cl(W) \cap \Cl(W') \not = \emptyset$.
To avoid confusion, we say that $A \sub \cW$ is \emph{$\E$-connected} when $A$ is connected as a subgraph of $(\cW, E)$. With a slight abuse of notation, we shall denote the edge relation on nerve graphs of different quasi-partitions always with the letter $E$ (it will always be clear from the context to which quasi-partition we are referring to).

A metric space $X$ is \emph{uniformly locally connected} if for every $\varepsilon>0$ there exists $\delta>0$ such that for all $x,y\in X$, $d(x,y)<\delta$ implies that there is a connected set $C\subseteq X$ with $x,y\in C$ and $\diam(C)<\varepsilon$. Peano continua are uniformly locally connected, and in a Peano continuum the connected set $C$ can be always assumed to be an arc.

A quasi-partition $\cW=\{W_0,\ldots,W_{n-1}\}$ of a Peano continuum $X$ is called a \emph{brick partition} if $W_i$ is a connected and uniformly locally connected regular open set for every $i<n$, and $W_i\vee W_j$ is uniformly locally connected for all $i,j<n$, with respect to some/each compatible metric on $X$.
In particular, if $W_i \E W_j$ then $W_i \vee W_j$ is connected.

Given a quasi-partition $\cU$ of a Peano continuum $X$ and a partition $\cP = \set{\cU_0, \dots, \cU_{n-1}}$ of $\cU$ into $\E$-connected subsets, the \emph{amalgam} of $\cU$ with respect to $\cP$ is the quasi-partition $\set{\runion \cU_0, \dots, \runion \cU_{n-1}}$ of $X$.
The amalgam of a brick partition is again a brick partition.

A family of open sets $\cW$ is said to \emph{refine} another one $\cV$, in symbols $\cW \refines \cV$, if every element of $\cW$ is contained in an element of $\cV$.
In the case of brick partitions, a refinement $\cW \refines \cV$ gives rise to a monotone epimorphism from the nerve graph of $\cW$ to that of $\cV$, given by mapping $W \mapsto V$ whenever $W \sub V$.
We denote such map by $\rmap{\cW}{\cV} \colon \cW \to \cV$.
Given two brick partitions $\cW, \cV$ and two walks $w, v$ on $\cW$ and $\cV$ respectively,
we write $(\cW, w) \refines (\cV, v)$ whenever $\cW \refines \cV$ and $\rmap{\cW}{\cV} \cdot w \refex v$.

We will make use of the following fundamental result of Bing from \cites{bing, bing2} implicitly whenever we need to produce arbitrarily fine brick partitions (i.e. with arbitrarily small mesh) of a Peano continuum.
\begin{theorem}[{\cites{bing, bing2}}]\label{thm: brick partitioning theorem}
	Any brick partition of a Peano continuum can be refined by brick partitions of arbitrarily small mesh.
\end{theorem}

Recall that a subset $K \sub X$ is \emph{locally separating} if there is an open connected set $U \subseteq X$ such that $U \setminus K$ is disconnected,
that a space is \emph{planar} if it admits an embedding in $\bbR^2$, and it is \emph{locally non-planar} if none of its nonempty open sets are planar.
We will use the following facts, which easily follow from the definitions.

\begin{fact}
	\label{fact:points-on-boundaries}
	Let $X$ be a Peano continuum and let $\cW$ be a brick partition of $X$ and $W \E W'$ be in $\cW$.

	\begin{enumerate}
		\item \label{itm:bing}
		      There is a point $x \in \Cl(W)\cap \Cl(W') \cap (W \vee W')$;
		\item \label{itm:nls}
		      if moreover $X$ has no locally separating points, then there are two distinct such points;
		\item \label{itm:nls_set} if $K \sub X$ is a non-locally-separating subset with empty interior, then $x$ can be chosen disjoint from $K$.
	\end{enumerate}
\end{fact}

We also need these facts on the existence of certain arcs.

\begin{fact}
	\label{fact:existence-of-arcs}
	Let $X$ be a Peano continuum and $W \sub X$ a connected and uniformly locally connected regular open set.
	\begin{enumerate}
		\item \label{itm:existence-of-arcs} For any $x, y \in \Cl(W)$, there is an arc $\gamma \sub W \cup \set{x,y}$ connecting $x, y$ (\cite{nadler2017continuum}*{8.32(a)}).
		\item \label{itm:existence-of-arcs-nls} If $X$ has no locally separating points, for any $A, B \sub \Cl(W)$, disjoint and containing exactly two points each, there are two disjoint arcs $\gamma_0, \gamma_1 \sub W \cup A \cup B$, connecting $A$ to $B$ (\cite{MR1562166}*{Lemma 1}).
		\item \label{itm:existence-of-arcs-cross} If $X$ has no locally separating points and $W$ is locally non-planar, then for any $x_0, x_1, y_0, y_1 \in \Cl(W)$ all distinct, there are two disjoint arcs $\gamma_0, \gamma_1 \sub W \cup \set{x_0, x_1, y_0, y_1}$ connecting $x_0$ to $y_0$ and $x_1$ to $y_1$, respectively (\cite{MOT}*{Theorem 3.13}).
		\item \label{itm:existence-of-arcs-nls-arc} If $K \sub X$ is a non-locally-separating $F_\sigma$ set with empty interior, and $x, y \in \Cl(W)$, there is an arc $\lambda \sub (W \setminus K) \cup \set{x,y}$ connecting $x, y$.
	\end{enumerate}
\end{fact}
\begin{proof}[Proof of \Cref{fact:existence-of-arcs}.(\ref{itm:existence-of-arcs-nls-arc})]
	By \cite{MOT}*{Proposition 2.5}, $\Bd(W)$ is non-locally-separating in $\Cl(W)$, thus so is $K \cup \Bd(W)$, see \cite{MOT}*{Theorem 2.2}.
	Let $K' = (K \cup \Bd(W)) \setminus \set{x, y }$.
	We claim that $K'$ is still non-locally-separating in $\Cl(W)$.
	Indeed, suppose that $U \sub \Cl(W)$ is a nonempty open connected set such that $U \setminus K'$ is not connected.
	On the other hand $U \setminus K$ is connected, and this is possible only if $K \cup \Bd(W)$ contains a connected component of $U \setminus K'$, which gives a contradiction since both $K$ and $\Bd(W)$ have empty interior.
	In particular, $K'\cap \Cl(W)$ is a non-locally-separating $F_\sigma$ subset of $\Cl(W)$.
	The conclusion then follows by \cite{MR0007095}*{II.5.5}, which states that all connected, locally connected, $G_\delta$ subsets of a complete metric space are arcwise connected.
\end{proof}

\section{Walks on Brick Partitions and their Refinements}
\label{sec:Walks on Brick Partitions and their Refinements}

We now apply the framework developed in \Cref{sec:walks-on-graphs} to walks on nerve graphs of brick partitions of Peano continua. The highlight of this section is \Cref{thm:from-Rosendal-to-walks}, where we show that the existence of a generic chain on $X$ implies an \emph{off-by-one} weak amalgamation principle for walks on (nerve graphs of) brick partitions of $X$. In \Cref{sec:The_proof} we shall derive our non-existence result in \Cref{thm:the-theorem-introduction} from the failure of such principle.

The ingredients of the proof of \Cref{thm:from-Rosendal-to-walks} are a general criterion for the existence of a comeager orbit due to Rosendal, and two discretization results: \Cref{cor:refinement_walk} for open sets of chains on $X$, and \Cref{thm:approximate-homeo-with-epi} for homeomorphism of $X$. Such discretization process revolves around the next definition, which explains how to associate open subsets of $\chains(X)$ to walks on brick partitions.

\begin{definition} \label{def:Ww}
	Let $X$ be a Peano continuum and $\cW$ a brick partition on it. Given a walk $w$ on $\cW$, we define the following open subset of $\chains(X)$:
	\begin{equation*}
		\opwalk{\cW}{w} \coloneqq \left\{ \cC \in \chains(X) : \cC \cap O\left(w(0), \dots, w(i-1)\right) \not = \emptyset \text{ for every } i < \len(w) \right\}.
	\end{equation*}
\end{definition}
Notice that $\opwalk{\cW}{w}$ is nonempty: for each $i< \len(w)$, choose a point $x_i \in w(i)$, and let $\gamma_i$ be an arc connecting $x_i$ to $x_{i+1}$ which is contained in $w(i) \vee w(i+1)$, which exists by \Cref{fact:existence-of-arcs}.\eqref{itm:existence-of-arcs}.
Call $K_0 \coloneqq \set{x_0}$ and $K_i \coloneqq \bigcup_{j < i} \gamma_j$ for each $0 < i < \len(w)$.
Then $K_i$ is a subcontinuum of $X$ such that $K_i \in O\left(w(0), \dots, w(i-1)\right)$, and $K_i \sub K_{i+1}$.
\cite{illanes1999hyperspaces}*{Theorem 14.6} then ensures that there is $\cC \in \chains(X)$ such that $K_i \in \cC$, for each $i< \len(h)$, so $\cC \in \opwalk{\cW}{w}$.

Also notice that for each walk $w$ on $\cW$ there is a reduced walk $\bar w$ such that $\opwalk{\cW}{\bar w} = \opwalk{\cW}{w}$.

It is immediate from the definitions that if $w', w$ are walks on a brick partition $\cW$,
\begin{equation*}
	\text{if }  \quad w' \refex w \quad \text { then } \quad  \opwalk{\cW}{w'} \sub \opwalk{\cW}{w},
\end{equation*}
and, more in general, that whenever $\cW \refines \cV$ and $w, v$ are walks on $\cW, \cV$, respectively
\begin{equation}
	\label{eq:refine-iff-subset}
	\text{if }  \quad \rmap{\cW}{\cV} \cdot w \refex v \quad \text { then } \quad  \opwalk{\cW}{w} \sub \opwalk{\cV}{v}.
\end{equation}

If $\Imm(w) \sub  \Imm(w')$, we have a strong converse to the above:
\begin{equation}
	\label{eq:refine-if-subset}
	\text{if }  \quad  \opwalk{\cW}{w'} \cap \opwalk{\cW}{w} \ne \emptyset \quad \text { then } \quad   w' \refex w.
\end{equation}
Indeed, if $ w' \not \refex w$ but $\Imm(w) \sub  \Imm(w')$, then there are initial segments $\set{w(0), \dots, w(i-1)}, \set{w'(0), \dots, w'(j-1)}$ of $w, w'$, respectively, which are incomparable under inclusion, so no chain can intersect both $O(w(0), \dots, w(i-1))$ and $O(w'(0), \dots, w'(j-1))$.

In particular, if $w, w'$ are two walks with $\Imm(w) = \Imm(w') = \cW$, then $w \refex w'$ if and only if $w'\refex w$.
We introduce the notation
\begin{equation*}
	\chains(\cW) \coloneqq \set{\opwalk{\cW}{w} : \Imm(w)  = \cW}
\end{equation*}
and stress that this is a set of pairwise disjoint open subsets of $\chains(X)$ by the remark above.

The following proposition shows that more is true: $\chains(\cW)$ is a quasi-partition whose mesh is controlled by that of $\cW$.

\begin{proposition}
	\label{prop:induced-quasi-partition}
	Let $X$ be a Peano continuum, and $\cW$ be a brick partition of $X$. Fix a compatible metric $d$ on $X$ and the induced metric $d_\chains$ on $\chains(X)$.
	Then $\chains(\cW)$ is a quasi-partition of $\chains(X)$ such that $\mathrm{mesh}_{d_\chains}(\chains(\cW)) \le \mathrm{mesh}_{d}(\cW)$.
	Moreover, whenever $\cU$ is a brick partition such that $\cU \refines \cW$, then $\chains(\cU) \refines \chains(\cW)$.
\end{proposition}

The proof of the proposition relies on the following lemma.

\begin{lemma}
	\label{lem:criterion-close-chains}
	Let $X$ be a Peano continuum and $\cU$ a brick partition of $X$.
	If $\cC, \cC' \in \chains(X)$ are such that
	\begin{equation}
		\label{eq:inclusion-star}
		\set{\Star(K, \cU) : K \in \cC} \sub \set{\Star(K', \cU) : K' \in \cC'},
	\end{equation}
	then $d_{\chains}(\cC, \cC') \le \mathrm{mesh}_d(\cU)$.
\end{lemma}
\begin{proof}
	Set $\varepsilon \coloneqq \mathrm{mesh}_d(\cU)$.  Recall that in order to show that $d_{\chains}(\cC, \cC') \le \varepsilon$, it suffices to prove that for every $K \in \cC$ there is $K' \in \cC'$ such that $d_H(K,K') \le \varepsilon$, and symmetrically that for every $K' \in \cC'$ there is $K \in \cC$ such that  $d_H(K,K') \le \varepsilon$.

	To this end, fix  $K \in \cC$.
	By hypothesis, there is $K' \in \cC'$ such that $\Star(K, \cU) = \Star(K', \cU)$.
	It follows that $d_H(K, K') \le \varepsilon$.

	Fix next $K' \in \cC'$ and let $K \in \cC$ be the minimum element such that $\Star(K', \cU) \sub \Star(K, \cU)$. If $K$ is the root of $\cC$, then
	\begin{equation*}
		K \sub \bigcap \set{\Cl(U) : U \in \Star(K, \cU)} \sub \bigcap \set{\Cl(U) : U \in\Star(K', \cU)},
	\end{equation*}
	which implies $d_H(K,K') \le \varepsilon$. Otherwise, by maximality and connectedness of $\cC$ (as a subset of $C(X)$, see \cite{illanes1999hyperspaces}*{Lemma 14.7}), see , we have that
	\begin{equation*}
		K = \Cl\left(\bigcup \set{F \in \cC : F \subsetneq K}\right).
	\end{equation*}
	The minimality assumption on $K$ and \eqref{eq:inclusion-star} entail $\Star(F, \cU) \subsetneq \Star(K', \cU)$ for every $F \in \cC$ with $F \subsetneq K$, hence we have that $K \sub  \Cl(\bigcup \Star(K', \cU))$. This, along with $\Star(K', \cU) \sub \Star(K, \cU)$, is sufficient to deduce $d_H(K, K') \le \varepsilon$.
\end{proof}

\begin{proof}[Proof of \Cref{prop:induced-quasi-partition}]
	Let $\varepsilon \coloneqq \mathrm{mesh}_{d}(\cW)$. Given a walk $w$ on $\cW$ such that $\Imm(w) = \cW$, any two chains $\cC, \cC' \in \opwalk{\cW}{w}$ satisfy
	\begin{equation*}
		\set{\Star(K, \cU) : K \in \cC} = \set{\Star(K', \cU) : K' \in \cC'}.
	\end{equation*}
	\Cref{lem:criterion-close-chains} then shows that the $d_{\chains}$-diameter of each $\opwalk{\cW}{w} \in \chains(\cW)$ is at most $\varepsilon$.

	The moreover statement follows from \eqref{eq:refine-iff-subset}. We can thus use it to prove that $\bigcup \chains(\cW)$ is dense in $X$: more precisely, since $\cW$ can be refined by arbitrarily fine brick partitions, it suffices to show that $\bigcup \chains(\cW)$ is $\varepsilon$-dense. Fix some $\cC \in \chains(X)$.
	The set $\set{\Star(K, \cW) : K \in \cC}$
	consists of a finite number of $E$-connected subsets of $\cW$ linearly ordered by inclusion
	\begin{equation*}
		\cW_0 \subset \cdots \subset \cW_{m-1} = \cW.
	\end{equation*}
	Let $w$ be any walk which starts in $\cW_0$, and successively covers each $\cW_i$, taking care of not exiting it before having visited each of its elements.
	By \Cref{lem:criterion-close-chains},  for any $\cC' \in \opwalk{\cW}{w}$, one has $d_\chains(\cC, \cC') \le \varepsilon$.

\end{proof}

\begin{corollary}
	\label{cor:refinement_walk}
	Let $X$ be a Peano continuum, and fix a brick partition $\cU$ and a walk $u$ on $\cU$.
	Let $\bV \subseteq \opwalk{\cU}{u} \sub \chains(X)$ be a nonempty open set.
	There exist a brick partition and a reduced walk $(\cV, v) \refines (\cU, u)$ such that $\opwalk{\cV}{v} \subseteq \bV$.
\end{corollary}
\begin{proof}
	Up to passing to a smaller open set, we can suppose that $\bV$ is the $d_{\chains}$-ball of radius $\varepsilon$ centered around some $\cC \in \opwalk{\cU}{u}$.
	Let $\cV \refines \cU$ be a brick partition whose mesh is smaller than $\varepsilon$.
	By \Cref{prop:induced-quasi-partition}, there is a reduced walk $v$ on $\cV$ such that $\Imm(v)  = \cV$ and $\cC \in \Cl(\opwalk{\cV}{v})$.
	By the same proposition, the diameter of $\opwalk{\cV}{v}$ is smaller than $\varepsilon$, so $\opwalk{\cV}{v} \sub \bV \sub \opwalk{\cU}{u}$.
	Since $\Imm(v)  = \cV$, we can apply \eqref{eq:refine-if-subset} and conclude that $\rmap{\cV}{\cU}\cdot v \refex u$.
\end{proof}

We state a preliminary lemma which will be needed later on.
Recall from \Cref{sec:walks-on-graphs} that $\ball{1}{z}$, for a walk $z$ on $\cZ$, is the set of $ Z' \in \cZ$ such that there is $Z \in \Imm(z)$ with $Z \E Z'$.

\begin{lemma}[Spacing out Paths]
	\label{lem:path-to-spaced}
	Let $X$ be a Peano continuum, and fix a brick partition $\cW$ and a path $w$ on $\cW$.
	There are a brick partition $\cZ \refines \cW$ and a spaced path $z$ on $\cZ$ such that $\rmap{\cZ}{\cW} \cdot z$ monotonically refines $w$ and $\rmap{\cZ}{\cW}(\ball{1}{z}) \sub \Imm(w)$.
	If $w$ is moreover circular, $z$ can also be chosen to be circular.
\end{lemma}
\begin{proof}
	Find $x_0 \in w(0)$ and $x_{\len(w)} \in w(-1)$.
	For each $0< i < \len(w)$, let $x_i \in \Cl(w(i-1)) \cap \Cl(w(i)) \cap (w(i-1) \vee w(i))$, which exists by \Cref{fact:points-on-boundaries}.\eqref{itm:bing}.
	For each $i <\len(w)$, let $\gamma_i$ be an arc in $w(i) \cup \set{x_i, x_{i+1}}$ connecting $x_i, x_{i+1}$, which exist by \Cref{fact:existence-of-arcs}.\eqref{itm:existence-of-arcs}.
	Since $w$ is a path, all the $x_i$'s and all the $\gamma_i$'s are distinct, and the $\gamma_i$'s intersect at most in their endpoints.

	Let $\cZ' \refines \cW$ be a brick partition fine enough that $\ball{1}{\Star(\gamma_i, \cZ')} \sub \cZ'(w(i-1) \vee w(i) \vee w(i+1))$ for each $i<\len(w)$, and let $P_i \coloneqq \Star(\gamma_i, \cZ'(w(i)))$.
	Let $\cZ$ be an amalgam of $\cZ'$ which refines $\cW$ and contains each $\bigvee P_i$ for $i< \len(w)$.
	Then $z \coloneqq \walk{\bigvee P_0, \dots,\bigvee P_{\len(w)-1}}$ is a spaced path satisfying the conclusions of the proposition.

	If $w$ is circular, we can take $x_0 = x_{\len(w) }$ in $\Cl(w(-1)) \cap \Cl(w(0)) \cap (w(-1) \vee w(0))$ and add `$\!\!\mod \len(w)$' where needed in the proof above.
\end{proof}

\subsection{A Criterion for the Absence of Generic Chains}

We recall the following criterion due to Rosendal for verifying the existence of comeager orbits, whose proof appears in \cite{ben2017metrizable}*{Proposition 3.2}.
Recall that a continuous action $H \acts Y$ is \emph{topologically transitive} if $H \cdot \bU \cap \bV \ne \emptyset$ for any nonempty open sets $\bU, \bV \sub Y$.

\begin{theorem}[Rosendal's criterion] \label{fact:rosendal}
	Let $H$ be a Polish group, let $Y$ be a Polish space and let $H \curvearrowright Y$ be a continuous action. The following are equivalent:
	\begin{enumerate}
		\item \label{item1:rosendal} There is a comeager orbit in $H \acts Y$.
		\item \label{item2:rosendal} $H \acts Y$ is topologically transitive and for any open neighborhood $H_0$ of the identity in $H$ and any nonempty open set $\bU \subseteq Y$ there is a nonempty open $\bV \subseteq \bU$ such that for any nonempty open $\bW_0, \bW_1 \subseteq \bV$, $H_0 \cdot \bW_0 \cap \bW_1 \ne \emptyset$.
	\end{enumerate}
\end{theorem}

The goal of this section is to obtain a combinatorial formulation of \eqref{item1:rosendal} $\Rightarrow$ \eqref{item2:rosendal} from \nameref{fact:rosendal} (\Cref{fact:rosendal}) when $Y=\chains(X)$ is the space of maximal connected chains on a Peano continuum $X$ and $H=\Homeo(X)$ is its homeomorphism group. This is done in \Cref{thm:from-Rosendal-to-walks}, after establishing our discretization theorem for homeomorphism of a Peano continuum, \Cref{thm:approximate-homeo-with-epi}.
The following is a generalization of \cite{MOT}*{Corollary 2.11} (where $\cU = \cV$).

\begin{proposition}
	\label{prop:approximate-core-refinement}
	Let $X$ be a Peano continuum and let $\cU, \cV$ be brick partitions of $X$.
	There is an arbitrarily fine brick partition $\cW \refines \cU$ such that:
	\begin{enumerate}
		\item \label{itm:nonempty-connected-core} $\Core(V, \cW)$ is nonempty and $\E$-connected for every $V \in \cV$;
		\item \label{itm:neighbor-core} for every $W \in \cW$ there are $V \in \cV$ and $W' \in \Core(V, \cW)$ such that $W \E W'$.
	\end{enumerate}
\end{proposition}
\begin{proof}
	Up to taking a refinement we can assume that $\cU$ is already as fine as desired.
	For each $V \in \cV$, let $T_V$ be a continuum-theoretical tree in $V$ which meets every $U \in \Star(V,\cU)$.
	Let $\cW_0 \refines\cU$ be a brick partition fine enough so that $\Star(T_V, \cW_0) \cap \Star(\Bd(V), \cW_0) = \emptyset$ for all $V \in \cV$.

	Given $U \in \cU$ and $V \in \cV$, let $\cC^{U,V}$ be the set of elements $W \in \cW_0(U \cap V)$ for which there is a path from $W$ to
	$\Star(T_V \cap U , \cW_0)$ in $\cW_0(U \cap V)$ which avoids $\Star(\Bd(V), \cW_0)$.
	Note that $\cC^{U,V} \subseteq \Core(V, \cW_0)$, and that $\cC^{U,V}$ is nonempty if and only if $U \cap V \ne \emptyset$, by the fineness assumption on $\cW_0$.
	Let $\cC^{U,V}_0, \dots, \cC^{U,V}_{m_{U,V}-1}$ be the $\E$-connected components of $\cC^{U,V}$.

	For each $U \in \cU$, let $\cB^U \coloneqq \cW_0(U) \setminus \bigcup_{V \in \cV} \cC^{U,V}$, and call its $\E$-connected components $\cB^U_0, \dots, \cB^U_{n_U-1}$.
	We claim that if $\cB^U$ is nonempty, then for each $V \in \cV$ and $i<n_U$,
	\begin{equation}
		\label{eq:B_is-intersect-boundaries}
		\cB^U_i \nsubseteq \Core(V, \cW_0).
	\end{equation}
	Indeed, suppose that there is $W \in \cB^U_i \cap \Core(V, \cW_0)$.
	Since $\cW_0(U)$ is $E$-connected and $U \cap V \ne \emptyset$, there is a path $P$ from $W \in \cB^U_i$ to $\Star(T_V, \cW_0)$ in $\cW_0(U)$.
	By assumption $W \not \in \cC^{U,V}$, so there must be a minimal $k<\len(P)$ such that $P(k) \in \Star(\Bd(V), \cW_0)$.
	For each $j < k$, $P(j) \in \cW_0(U \cap V)$, so $P(j) \not \in \cC^{U,V}$, as otherwise there would be a path from $W$ to $\Star(T_V, \cW_0)$ in $\cW_0(U \cap V)$ which avoids $\Star(\Bd(V), \cW_0)$.
	But $P(j) \not \in \cC^{U,V'}$, for $V' \ne V$, because it is contained in $V$, so $P(j) \in \cB^U$.
	Moreover, $P(k) \in \cB^U$, since it belongs to $\Star(\Bd(V), \cW_0)$.
	Since $\cB^U_i$ is a connected component, $P(k) \in \cB^U_i$, so $\cB^U_i \nsubseteq \Core(V, \cW_0)$.

	Notice that
	\begin{equation*}
		\set{\cC^{U,V}_j : U \in \cU, V \in \cV, j < m_{U,V}} \cup \set{\cB^U_i : U \in \cU, i<n_U}
	\end{equation*}
	is a partition of $\cW_0$ in $\E$-connected sets, which refines $\cU$.
	Let $\cW$ be the amalgam of $\cW_0$ relative to this partition.

	Given $W \in \cW$ and $V \in \cV$, by \eqref{eq:B_is-intersect-boundaries} we have that
	\begin{equation} \label{eq:V(W)}
		W \in \Core(V,\cW) \iff W = \bigvee \cC^{U,V}_j \text{ for some } U \in \cU, j < m_{U,V}.
	\end{equation}
	This shows in particular that $\Core(V,\cW) = \Star(T_V, \cW)$, which is nonempty and $E$-connected.

	To verify item \eqref{itm:neighbor-core}, pick $W \in \cW$. If there is $V \in \cV$ such that $W \in \Core(V,\cW)$ then let $W' \coloneqq W$.
	If that is not the case, then by \eqref{eq:V(W)} we know that $W = \bigvee \cB^U_i$ for some $U \in \cU$ and $i<n_U$.
	Since $\cB^U_i$ is a connected component of $\cB^U$ and $\cW_0(U)$ is $E$-connected, there is some $V \in \Star(U,\cV)$ and $j < m_{U,V}$ such that $(\bigvee \cB^U_i) \E (\bigvee \cC^{U,V}_j)$.
	By \eqref{eq:V(W)} we then have $\bigvee \cC^{U,V}_j \in \Core(V,\cW)$, as desired.
\end{proof}

We recall that, under the identification of a partition $\cW$ with its nerve graph, the set $\ball{1}{W}$, for $W \in \cW$, is the set of all $W' \in \cW$ such that $W \E W'$.

\begin{theorem}
	\label{thm:approximate-homeo-with-epi}
	Let $X$ be a Peano continuum, $\cU, \cV$ be brick partitions on $X$, and let $h \in \homeo(X)$.
	There are a brick partition $\cW \refines \cU$ and a monotone epimorphism $\alpha \colon \cW \to \cV$ such that $h[\Cl(W)] \sub \bigvee \ball{1}{\alpha(W)}$ for each $W \in \cW$.

	Moreover, if $u,v$ are walks on $\cU,\cV$, respectively, such that $\opwalk{\cU}{u}\subseteq h^{-1}\cdot\opwalk{\cV}{v}$, then $\cW$ and $\alpha$ can be chosen so that there exists a walk $w$ on $\cW$ such that $\rmap{\cW}{\cU}\cdot w\refex u$ and $\alpha\cdot w\refex v$.
\end{theorem}
\begin{proof}
	Let $\cW$ be given by \Cref{prop:approximate-core-refinement} applied to $\cU$ and $h^{-1} \cdot \cV$.
	We claim that $\cW$ can be picked fine enough so that:
	\begin{enumerate}
		\item \label{itm:clique} if $W \E W'$ then $V \E V'$ for any $V \in \Star(h[\Cl(W)], \cV), V' \in \Star(h[\Cl(W')], \cV)$,
		\item \label{itm:spaced} for each $V \E V'$ there is a path $P$ in $\cW(h^{-1}(V \vee V'))$ connecting $\Core(h^{-1}(V),\cW)$ to $\Core(h^{-1}(V'),\cW)$.
	\end{enumerate}
	It is fairly clear how to satisfy the first condition, so we focus on the second one: as $\cV$ is a brick partition, given $V, V' \in \cV$ with $V \E V'$, by Facts \ref{fact:points-on-boundaries} and \ref{fact:existence-of-arcs} we can find an arc $\gamma \in O(V, V')$.
	By choosing $\cW$ fine enough we can assume that $\Star(h^{-1}(\gamma), \cW) \subseteq \cW(h^{-1}(V \vee V'))$ and that there are
	$W \in \Core(h^{-1}(V),\cW) \cap \Star(h^{-1}(\gamma), \cW)$
	and $W' \in \Core(h^{-1}(V'),\cW) \cap \Star(h^{-1}(\gamma), \cW)$.
	Now let $P$ be any path in $\Star(h^{-1}(\gamma), \cW)$ from $W$ to $W'$.
	Since this can be done for any pair of $E$-adjacent elements in $\cV$, the conclusion follows.

	We define $\alpha$ in two steps.
	First, for each $V \in \cV$ and $W \in \Core(h^{-1}(V),\cW)$, let $\alpha(W) \coloneqq V$.
	Then, by \Cref{prop:approximate-core-refinement}, for every $W \in \cW$ there is at least one $W' \E W$ on which $\alpha$ is defined.
	Pick one such $W'$ and let $\alpha(W) \coloneqq \alpha(W')$.
	Notice that in particular $W \E \cW(h^{-1}(\alpha(W)))$, so $h[\Cl(W)] \cap \Cl(\alpha(W)) \ne \emptyset$, that is:
	\begin{equation}
		\label{eq:alpha_in_star}
		\alpha(W) \in \Star(h[\Cl(W)], \cV).
	\end{equation}
	Moreover, if $W \in \cW(\runion_{V \in \cV_0} h^{-1}(V))$ for some $\cV_0 \sub \cV$, then
	\begin{equation}
		\label{eq:alpha_ball}
		\alpha(W) \in \cV_0,
	\end{equation}
	since if $W' \E W$ and $W' \in \Core(h^{-1}(V),\cW)$, then $V \in \cV_0$.

	Condition \eqref{itm:clique} and \eqref{eq:alpha_in_star} imply that $\alpha$ is a graph homomorphism, while condition \eqref{itm:spaced} and \eqref{eq:alpha_ball} imply that $\alpha$ is surjective on edges.

	For each $V \in \cV$, $\alpha^{-1}(V)$ contains $\Core(h^{-1}(V), \cW)$, which is nonempty and $\E$-connected by \Cref{prop:approximate-core-refinement}.\eqref{itm:nonempty-connected-core}, and by definition of $\alpha$ any other element of $\alpha^{-1}(V)$ is adjacent to $\Core(h^{-1}(V), \cW)$, so $\alpha^{-1}(V)$ is $\E$-connected.

	Notice that condition \eqref{itm:clique} implies that for each $W \in \cW$, $\Star(h[\Cl(W)], \cV)$ is a clique in $\cV$.
	Together with \eqref{eq:alpha_in_star},
	it follows that
	\begin{equation*}
		h[\Cl(W)] \sub \runion \Star(h[\Cl(W)], \cV) \sub \runion \ball{1}{\alpha(W)} \text{ for every } W \in \cW.
	\end{equation*}

	For the moreover, fix $\cC \in \opwalk{\cU}{u} \sub h^{-1} \cdot \opwalk{\cV}{v}$.
	For each $s < \len(u)$, let $C_s^u \in \cC$ be such that $C_s^u \in O(u(0), \dots, u(s))$, and for each $t < \len(v)$, let $C_t^v \in \cC$ be such that $C_t^v \in O(h^{-1}(v(0)), \dots, h^{-1}(v(t)))$.
	We can suppose that $C_0^u = C_0^v$ is the root $\bigcap \cC$ of $\cC$.

	Pick $\cW$ which is moreover fine enough that for each $t < \len(v)$:
	\begin{enumerate}[label=(\alph*)]
		\item \label{itm:t-v} $\Star(C_t^v, \cW) \sub \cW(\runion_{j \le t} h^{-1}(v(j)))$,
		\item \label{itm:W_t} there is $W_{t} \in \Star(C_t^v, \cW) \cap \cW(h^{-1}(v(t)))$.
	\end{enumerate}

	Let $C_0 \subsetneq C_1 \subsetneq \cdots \subsetneq C_{n-1}$ enumerate $\set{C_s^u}_{s < \len(u)} \cup \set{C_t^v}_{t < \len(v)}$ without repetitions,
	and let $w$ be any walk which starts in $W_{0}$ and covers all of $\Star(C_k, \cW)$ before passing to $\Star(C_{k+1}, \cW)$.
	Call $k_t< \len(w)$ the minimum such that $\Star(C_t^v, \cW) = \set{w(0), \dots, w(k_t)}$.

	Since $\cW$ refines $\cU$,
	\begin{equation*}
		O(\Star(C_s^u, \cW)) \sub O(u(0), \dots, u(s)),
	\end{equation*}
	for all $s < \len(u)$, so $\rmap{\cW}{\cU}\cdot w \refex u$, since $\Imm(\rmap{\cW}{\cU}\cdot w) = \rmap{\cW}{\cU}[\Imm(w)] \supseteq  \Imm(u)$.
	Fix $t< \len(v)$. By \ref{itm:t-v} and \eqref{eq:alpha_ball},
	\begin{equation*}
		\alpha[\set{w(0), \dots, w(k_t)}] = \alpha[\Star(C_t^v, \cW)] \sub \set{v(0), \dots,v(t)}.
	\end{equation*}
	On the other hand, for each $j \le t$, $W_j \in \set{w(0), \dots, w(k_t)}$, and by \ref{itm:W_t} and the definition of $\alpha$, $\alpha(W_j) = v(j)$, so
	\begin{equation*}
		\alpha[\set{w(0), \dots, w(k_t)}] = \set{v(0), \dots,v(t)}.
	\end{equation*}
	Since $t$ was arbitrary, $\alpha\cdot w \refex v$.
\end{proof}

Each monotone epimorphism $\alpha \colon \cW \to \cV$ between brick partitions of $X$ gives rise to an open subset of $\homeo(X)$:
\begin{equation*}
	H_{\alpha} \coloneqq \set{h \in \homeo(X) : h[\Cl(W)] \sub \bigvee \ball{1}{\alpha(W)}, \text{ for each } W \in \cW}.
\end{equation*}
While many such $H_\alpha$'s might be empty, \Cref{thm:approximate-homeo-with-epi} implies that their collection forms a basis for the compact-open topology on $\homeo(X)$.
In particular the open sets
\begin{equation}
	\label{eq:def-HV1}
	H_{\cV, 1} \coloneqq H_{\rmap{\cV}{\cV}} = \set{h \in \homeo(X) : h[\Cl(V)] \sub \bigvee \ball{1}{V}, \text{ for each } V \in \cV},
\end{equation}
form a basis of open neighborhood of the identity in $\homeo(X)$, when $\cV$ varies among brick partitions of $X$.

We are finally ready for the main result of this section.

\begin{theorem}
	\label{thm:from-Rosendal-to-walks}
	Let $X$ be a Peano continuum
	with a generic chain.
	Then for each brick partition $\cU$ of $X$ and walk $u$ on $\cU$
	there are a brick partition and a walk $(\cV, v) \refines (\cU, u)$ such that
	for any brick partition $\cW \refines \cV$ and any two walks $w_0, w_1$ on $\cW$ with $\rmap{\cW}{\cV}\cdot w_i \refex v$, for $i=0,1$,
	there are a brick partition $\cZ$, a walk $z$ on $\cZ$ and two monotone epimorphisms $\alpha_0, \alpha_1 \colon \cZ \to \cW$ satisfying:
	\begin{enumerate}
		\item $\alpha_i\cdot z \refex w_i$, for $i=0,1$;
		\item \label{itm:distance1}$d_{\sup}(\rmap{\cW}{\cU}\alpha_0, \rmap{\cW}{\cU}\alpha_1) \le 1$.
	\end{enumerate}
\end{theorem}
\begin{proof}
	Let $\cU$ and a walk $u$ on $\cU$ be given.
	Let $(\cU_0, u_0) \refines (\cU, u)$, with $\cU_0$ a brick partition fine enough that
	\begin{equation}
		\label{eq:distance2to1}
		d_{\cU}(\rmap{\cU_0}{\cU}(U), \rmap{\cU_0}{\cU}(U')) \le 1 \text{ whenever } d_{\cU_0}(U ,U') \le 2.
	\end{equation}
	Find a symmetric open neighborhood of the identity $H_0 \sub H_{\cU_0, 1}$ (see \eqref{eq:def-HV1}).
	Also, set $\bU \coloneqq \opwalk{\cU_0}{u_0}$.
	\nameref{fact:rosendal} (\Cref{fact:rosendal}) gives us a nonempty $\bV \sub \bU$.
	By \Cref{cor:refinement_walk}, there are $(\cV, v) \refines (\cU_0, u_0)\refines (\cU, u)$, with $\cV$ a brick partition, such that $\opwalk{\cV}{v} \subseteq \bV$.

	Suppose now that we are given a brick partition $\cW \refines \cV$ and two walks $w_0, w_1$ on $\cW$ with $\rmap{\cW}{\cV}\cdot w_i \refex v$, for $i=0,1$.
	Set $\bW_i \coloneqq \opwalk{\cW}{ w_i}$; then $\bW_i \sub \opwalk{\cV}{v} \subseteq \bV$, for $i=0,1$.
	Therefore by \nameref{fact:rosendal} there is $h \in H_0$ such that $h^{-1} \cdot \bW_0 \cap \bW_1$ is a nonempty open set. We can thus apply \Cref{cor:refinement_walk} to obtain $(\cZ', z') \refines (\cW, w_1)$ such that $\opwalk{\cZ'}{z'} \sub h^{-1} \cdot \bW_0 \cap \bW_1 \sub \bW_1$.

	By \Cref{thm:approximate-homeo-with-epi}, there are a brick partition $\cZ \refines \cZ'$, a walk $z$ on $\cZ$ and a monotone epimorphism $\alpha_0 \colon \cZ \to \cW$ such that $\rmap{\cZ}{\cZ'}\cdot z \refex z'$, $\alpha_0 \cdot z \refex w_0$ and such that, for each $Z \in \cZ$,
	\begin{equation}
		\label{eq:gamma_1-approx-h^-1}
		h[\Cl(Z)] \sub \runion \ball{1}{\alpha_0(Z)}.
	\end{equation}
	Call $\alpha_1 \coloneqq \rmap{\cZ}{\cW}$; then $\alpha_1\cdot z \refex w_1$.

	It remains to prove item \eqref{itm:distance1}.
	Fix $Z \in \cZ$.
	On the one hand,
	\begin{equation*}
		h[\Cl(Z)] \sub h[\Cl(\rmap{\cZ}{\cU_0}(Z))] \sub \runion \ball{1}{\rmap{\cZ}{\cU_0}(Z)},
	\end{equation*}
	since $h \in H_0 \sub H_{\cU_0, 1}$.
	On the other hand,
	\begin{equation*}
		h[\Cl(Z)] \sub \runion \ball{1}{\alpha_0(Z)} \sub \runion \ball{1}{\rmap{\cW}{\cU_0}\alpha_0(Z)},
	\end{equation*}
	by \eqref{eq:gamma_1-approx-h^-1}.
	Therefore, $\runion \ball{1}{\rmap{\cW}{\cU_0}\alpha_0(Z)} \cap \runion \ball{1}{\rmap{\cZ}{\cU_0}(Z)} \ne \emptyset$, so
	\begin{equation*}
		d_{\cU_0}(\rmap{\cW}{\cU_0}\alpha_0(Z),\rmap{\cZ}{\cU_0}(Z)) \le 2.
	\end{equation*}
	It follows from \eqref{eq:distance2to1} that $d_{\cU}( \rmap{\cW}{\cU}\alpha_0(Z), \rmap{\cZ}{\cU}(Z)) \le 1$.
	Since $Z \in \cZ$ was arbitrary, $d_{\sup}(\rmap{\cW}{\cU}\alpha_0, \rmap{\cW}{\cU}\alpha_1) = d_{\sup}(\rmap{\cW}{\cU}\alpha_0, \rmap{\cZ}{\cU}) \le 1$, as wished.
\end{proof}

\section{Absence of Generic Chains}
\label{sec:The_proof}

In this section we finally put all pieces together and prove \Cref{thm:the-theorem-introduction}.

\subsection{Walks on Peano continua with no locally separating points}

The first step towards \Cref{thm:the-theorem-introduction} is to show that the hypothesis that $X$ has no locally separating points allows us to only deal with walks which are spaced paths. We do so in \Cref{thm:walks-to-paths} (\nameref{thm:walks-to-paths}). This will be helpful when trying to use \Cref{lem:wd-monotonically-refine} and \Cref{lem:wd-refine-initial-segment} to compare winding numbers of walks defined on different partitions.

To do so, we need the following proposition, which is a combinatorial version of Arc Doubling (see \cite{MOT}*{Lemma 3.4}).
\begin{proposition}[Path Doubling]
	\label{prop:path-doubling}
	Let $X$ be a Peano continuum with no locally separating points.
	Let $\cU$ be a brick partition of $X$ and $u$ be a path on $\cU$.
	There are a brick partition $\cV \refines \cU$ and two paths $v_0, v_{1}$ on $\cV$ which have the same endpoints and are disjoint elsewhere and such that $\rmap{\cV}{\cU} \cdot v_i$ monotonically refines $u$ for $i= 0,1$.
\end{proposition}
\begin{proof}
	For each $i< \len(u)-1$, let $A_i$ be a subset of cardinality 2 of $\Cl(u(i))\cap \Cl(u(i+1)) \cap (u(i) \vee u(i+1))$, which exists by \Cref{fact:points-on-boundaries}.\eqref{itm:nls}.
	Given $1 \le i< \len(u)-1$, \Cref{fact:existence-of-arcs}.\eqref{itm:existence-of-arcs-nls} gives two disjoint arcs $\gamma_i, \gamma'_i$ in $u(i) \cup A_{i-1} \cup A_{i}$ connecting $A_{i-1}$ to $A_{i}$.
	Since $u$ is a path, all the $A_i$'s and all the $\gamma_i$'s are distinct, and the $\gamma_i$'s intersect at most in their endpoints. The same holds for the $\gamma'_i$'s.
	Let $\lambda, \lambda'$ be the two disjoint arcs resulting from the concatenations of the $\gamma_i$'s and $\gamma'_i$'s respectively.

	Let $\tilde \cV \refines \cU$ be a brick partition fine enough that $\Star(\lambda, \tilde \cV) \cap \Star(\lambda', \tilde \cV) = \emptyset$.
	Let, for each $1\le i< \len(u)-1$:
	\begin{align*}
		P_i                            & \coloneqq \Star(\lambda, \tilde \cV(u(i))),  \\
		P'_i                           & \coloneqq \Star(\lambda', \tilde \cV(u(i))), \\
		P_0 = P'_0                     & \coloneqq \tilde \cV(u(0)),                  \\
		P_{\len(u)-1} = P'_{\len(u)-1} & \coloneqq \tilde \cV(u(-1)).
	\end{align*}
	Let $\cV$ be an amalgam of $\tilde \cV$ refining $\cU$ and containing each $\bigvee P_i$ and each $\bigvee P'_i$, for $i< \len(u)$.
	Then $v_0 \coloneqq \walk{\bigvee P_0, \dots, \bigvee P_{\len(u)-1}}$ and $v_1 \coloneqq \walk{\bigvee P'_0, \dots, \bigvee P'_{\len(u)-1}}$ are paths on $\cV$ satisfying the conclusions of the proposition.
\end{proof}

\begin{theorem}[Upgrading Walks to Paths]
	\label{thm:walks-to-paths}
	Let $X$ be a Peano continuum with no locally separating points.
	Let $\cV$ be a brick partition of $X$ and $v$ be a reduced walk on $\cV$.
	There is a spaced path $w$ on a brick partition $\cW$ such that $(\cW, w) \refines (\cV, v)$.
	If furthermore $v$ is uncrossed
	then $w$ can be chosen so that $\rmap{\cW}{\cV} \cdot w$ monotonically refines $v$.
\end{theorem}
\begin{proof}
	We proceed by induction on the length $\ell$ of $v$.
	If $\ell=1$, let $\cW \coloneqq \cV$ and $w \coloneqq v$.
	Suppose the result is true for $\ell \ge 1$ and let us prove it for
	$v = \walk{v(0), \dots, v(\ell)}$.
	By the induction hypothesis, there is a brick partition $\cZ \refines \cV$ and a spaced path $z$ on $\cZ$ such that $\rmap{\cZ}{\cV} \cdot z$ refines $\walk{v(0), \dots, v(\ell-1)}$, is confined to it, $\rmap{\cZ}{\cV}(z(-1)) = v(\ell -1)$, and the refinement is monotonic if the latter is uncrossed.

	By \nameref{prop:path-doubling} (\Cref{prop:path-doubling}), there are a brick partition $\cW \refines \cZ$ and spaced paths $w_0, w_1$ on $\cW$, which are disjoint except for their endpoints, and such that $\rmap{\cW}{\cZ} \cdot w_i$ monotonically refines $z$, for $i < 2$.
	Choose any $W \in \cW(v(\ell))$ with $W \E \cW(v(\ell-1))$ and let $\bar w$ be a path from $w_0(-1) = w_1(-1)$ to $W$ in $\cW(v(\ell-1)) \cup \set{W}$.

	\textbf{If $v$ is uncrossed:} Suppose without loss of generality that the \emph{last} time that $\bar w$ intersects the $\E$-connected component of $(w_0 \cup w_1) \cap \cW(v(\ell-1))$ containing $w_0(-1) = w_1(-1)$ is in $w_0(i)$, for some $i<\len(w_0)$.
	If after that step $\bar w$ is disjoint from $(w_0 \cup w_1)$, then
	we can simply take $w$ to be the concatenation of $w_0$ up to $w_0(i)$ and of the final segment of $\bar w$.
	Otherwise, let $w_{\epsilon}(j)$ be the \emph{first} time $\bar w$ intersects $(w_0 \cup w_1)$ after $w_0(i)$.

	If $\epsilon =0$, let $w$ be the walk that follows $w_1$ from start to end, then follows $w_0$ backwards from $w_0(-1)$ to $w_0(i)$, then $\bar w$ until $w_0(j)$, then $w_0$ until it exits $\cW(v(\ell-1))$ and crosses into $\cW(v(\ell))$, which is the case since $v$ is uncrossed and $\rmap{\cW}{\cV} \cdot w_0$ monotonically refines $\walk{v(0), \dots, v(\ell-1)}$.

	If $\epsilon =1$, let $w$ be the walk that follows $w_0$ from start to $w_0(i)$, then $\bar w$ until $w_1(j)$, then $w_1$ until it exits $\cW(v(\ell-1))$ and crosses into $\cW(v(\ell))$, which is the case since $v$ is uncrossed and $\rmap{\cW}{\cV} \cdot w_1$ monotonically refines $\walk{v(0), \dots, v(\ell-1)}$. In either case, $\rmap{\cW}{\cV} \cdot w$ monotonically refines $v$.

	\textbf{Otherwise,} we have to content ourselves to refine non-monotonically.
	Suppose without loss of generality that the \emph{last} time that $\bar w$ intersects $w_0 \cup w_1$ is in $w_0(j)$, for some $j<\len(w_0)$.
	Let $w$ be the walk that follows $w_1$ from start to end, then follows $w_0$ backwards from $w_0(-1)$ to $w_0(j)$ (this part might exit and then re-enter $\cW(v(\ell -1))$) and then $\bar w$ from $w_0(j)$ to $W$.

	By \nameref{lem:path-to-spaced} (\Cref{lem:path-to-spaced}), we can further ensure that $w$ is spaced.
\end{proof}

\subsection{Robust cycles}

The next step is to find a sufficient condition on $X$ under which
the conclusion of \Cref{thm:from-Rosendal-to-walks}, a necessary consequence to the existence of a generic chain on $X$, is violated.

Broadly speaking, our goal is to
isolate a region of $X$ (a circular subpartition to be precise) inside which there is a well-defined notion of winding number, while also making sure that we can refine and extend given paths in such region so as to be able to
reach arbitrarily large positive and negative values of winding number.
Once this is ensured, if $(\cV, v)$ is as in the statement of \Cref{thm:from-Rosendal-to-walks}, we can then find a brick partition
$\cW \refines \cV$ and two paths $w_0, w_1$ on $\cW$ refining $v$ and such that the winding number of $w_0$ is very large, while the one of $w_1$ is negative and very large in absolute value. If $X$ had a generic chain, \Cref{thm:from-Rosendal-to-walks} would permit to amalgamate up to distance one $(\cW, w_0)$ and $(\cW, w_1)$, and this would
lead to a contradiction by \Cref{lem:close-walks}, since $w_0$ and $w_1$ were chosen to have very distant winding numbers. The idea of using the winding number to show the failure of such amalgamation goes back to the proof of \cite{GTZ}*{Theorem 1.2}.

\begin{figure}
	\label{fig:robust}
	\centering{
		\includeinkscape[scale=1]{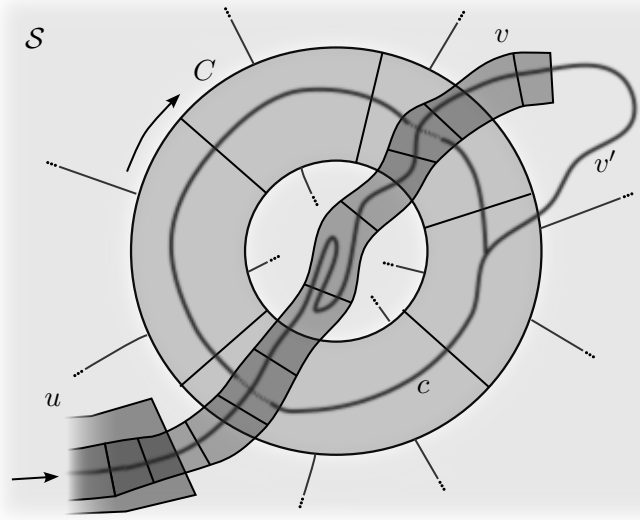}
	}
	\caption{A sketch of a robust cycle: only the relevant elements of the brick partitions appearing in \Cref{def:robust} are shown, and the lasso $v' \conc c$ is depicted as a non-self-intersecting curve.}
\end{figure}

Once this is done, all that is left is to show that the spaces considered in \Cref{thm:the-theorem-introduction} satisfy the necessary property to perform this construction (this is done in Theorems~\ref{thm:local-hp-robust-cycle0}, \ref{thm:local-hp-robust-cycle1} and \ref{thm:circular-covering-robust}). Such property is formalized in \Cref{def:robust} below. The intuitive and simplified idea
is to require the existence of a partition $\cS$ of $X$ with a \emph{robust} cycle $C$: by this we mean a cycle such that, given any brick partition $\cV \refines \cS$ and walk $v$ on $\cV$, even in the situation where $v$ intersects and apparently \emph{cuts} $C$, it is always possible to \emph{zoom in} (that is, find a refinement) to see that $C$ is still intact and in fact $v$ can
be prolonged with a disjoint circular path refining and confined in $C$ in a strong sense.

Such robust cycles do not exist, for instance, in the $2$-sphere $S^2$.
Indeed, any circular path $C$ on a brick partition of $S^2$ can be separated by an arc, and by all paths $v$ on sufficiently fine brick partitions containing such arc.
The resulting cut cannot be mended: any circular path which is confined to and winds around $C$, no matter on how fine a brick partition, must intersect $v$.

\begin{definition} \label{def:robust}
	Let us say that a Peano continuum $X$ has a \emph{robust cycle} if there are a brick partition $\cS$ with a spaced circular path $C$ of circular length $\ell \ge 4$, and a brick partition $\cU \refines \cS$ with a walk $u$ on $\cU$ such that for any brick partition and walk $(\cV, v) \refines (\cU, u)$ there are $(\cV', v'\conc c) \refines (\cV, v)$ where $v'\conc c$ is a lasso such that:
	\begin{equation}
		\label{eq:fattening-in-C}
		\rmap{\cU}{\cS}(\ball{2}{\rmap{\cV'}{\cU} \cdot c}) \sub C,
	\end{equation}
	and
	\begin{equation}
		\label{eq:winds-at-least-once}
		\Card{\wind_C(\rmap{\cV'}{\cS} \cdot c)} \ge \ell.
	\end{equation}
\end{definition}

The definition above is modeled on the statement of \Cref{thm:from-Rosendal-to-walks}, which in turn is structured following the statement of \nameref{fact:rosendal} (\Cref{fact:rosendal}): the pairs $(\cU, u), (\cV, v)$ in the Definition correspond to their counterparts in the Theorem.

\begin{theorem}
	\label{thm:the-theorem}
	Let $X$ be a Peano continuum with no locally separating points.
	If $X$ has a robust cycle, then $X$ has no generic chain.
\end{theorem}

\begin{proof}
	The existence of a robust cycle gives us a brick partition $\cS$, a spaced circular path $C$ of circular length $\ell \ge 4$, a brick partition $\cU \refines \cS$, and and a walk $u$ on $\cU$.
	Up to taking a refinement, by \nameref{lem:lemmetto-monotone} (\Cref{lem:lemmetto-monotone}) we can suppose that
	\begin{equation}
		\label{eq:U-S_2-to-1}
		d_{\cS}(\rmap{\cU}{\cS}(U), \rmap{\cU}{\cS}(U')) \le 1 \text{ whenever } d_{\cU}(U ,U') \le 2.
	\end{equation}

	Suppose we are given a brick partition and a walk $(\cV, v_0) \refines (\cU, u)$. By definition of robust cycle, we can assume that, up to a refinement, $v_0$ is a lasso $v\conc c_+$ such that:
	\begin{equation}
		\label{eq:annuli}
		\rmap{\cU}{\cS}(\ball{2}{\rmap{\cV}{\cU} \cdot c_+}) \sub C,
	\end{equation}
	and
	\begin{equation*}
		\Card{\wind_C(\rmap{\cV}{\cS} \cdot c_+)} \ge \ell.
	\end{equation*}

	Let $c_{-}$ be the circular path going in the opposite direction of $c_+$.
	Set $N \coloneqq \wind_C(\rmap{\cV}{\cS} \cdot c_+)$,
	so that $\wind_C(\rmap{\cV}{\cS} \cdot c_{-}) = -N$.
	Up to switching the role of $c_+, c_{-}$, we might assume that $N\ge \ell>0$.

	Let
	\begin{equation*}
		M_c \coloneqq \max \set{\Card{\wind_C(\rmap{\cV}{\cS} \cdot t)}: t \text{ is an initial segment of } c_+ \text{ or } c_{-}},
	\end{equation*}
	and
	\begin{equation*}
		M_v \coloneqq \max \set{\Card{\wind_C(\rmap{\cV}{\cS} \cdot t)}: t \text{ is an initial or final segment of } v}.
	\end{equation*}

	Since $N \ge \ell >1$, there is $k \in \N$ such that
	\begin{equation*}
		k \cdot N > 3M_v + M_c + 2.
	\end{equation*}
	Let
	\begin{equation*}
		\bar w_0 \coloneqq \underbrace{c_+ \conc c_+ \conc \cdots \conc c_+}_{k\text{ times}}\quad \text{and} \quad \bar w_1 \coloneqq \underbrace{c_- \conc c_- \conc \cdots \conc c_-}_{k\text{ times}}.
	\end{equation*}

	Since $v \conc \bar w_0$ is uncrossed,  by \nameref{thm:walks-to-paths} (\Cref{thm:walks-to-paths}) there is a brick partition $\widetilde \cW \refines \cV$ and a spaced path $\tilde w_0$ on $\widetilde{\cW}$ such that $\rmap{\widetilde \cW}{\cV} \cdot \tilde w_0$ monotonically refines $v \conc \bar w_0$. Since $v \conc \bar w_1$ is uncrossed as well, by \nameref{lem:lemmetto-monotone} it can be monotonically refined by an uncrossed walk $\tilde w_1$ on $\widetilde{\cW}$. Hence, applying \nameref{thm:walks-to-paths} to $\tilde w_1$, we can find a brick partition $\cW \refines \widetilde \cW$ and a spaced path $w_1$ such that $\rmap{\cW}{\cV} \cdot w_1$ monotonically refines $v \conc \bar w_1$. Another application of \nameref{lem:lemmetto-monotone} to $\tilde w_0$ yields finally a spaced path $w_0$ on $\cW$ such that $\rmap{\cW}{\cV} \cdot w_0$ monotonically refines $v \conc \bar w_0$. We stress that $w_0, w_1$ are paths, \emph{not} lassos.

	\begin{claim}
		\label{claim:winding-the-****-around-c}
		\mbox{}
		\begin{enumerate}
			\item[(a0)] If $t$ is an initial segment of $w_0$ then:
			      \begin{equation*}
				      \wind_C(\rmap{\cW}{\cS} \cdot t) \ge - (M_v + M_c),
			      \end{equation*}
			\item[(a1)] If $t$ is an initial segment of $w_1$ then:
			      \begin{equation*}
				      \wind_C(\rmap{\cW}{\cS} \cdot t) \le M_v + M_c,
			      \end{equation*}
			\item[(b0)] If $t$ is a final segment of $w_0$ such that $\rmap{\cW}{\cV} \cdot t$ contains $\bar w_0$ then:
			      \begin{equation*}
				      \wind_C(\rmap{\cW}{\cS} \cdot t) > 2M_v + M_c +2,
			      \end{equation*}
			\item[(b1)] If $t$ is a final segment of $w_1$ such that $\rmap{\cW}{\cV} \cdot t$ contains $\bar w_1$ then:
			      \begin{equation*}
				      \wind_C(\rmap{\cW}{\cS} \cdot t) < -(2M_v + M_c +2).
			      \end{equation*}
		\end{enumerate}
	\end{claim}
	\begin{proof}[Sketch of Proof]
		We only show that (b0) holds, the other cases being analogous.
		Let such a $t$ be given.
		By definition, $\wind_C(\rmap{\cV}{\cS} \cdot \bar w_0 )> 3M_v + M_c + 2$, and for any final segment $s$ of $v$, $\wind_C(\rmap{\cV}{\cS} \cdot s) \ge -M_v$.
		Since $\rmap{\cW}{\cV} \cdot t$ monotonically refines a final segment of $v \conc \bar w_0$ containing $\bar w_0$, we conclude by \Cref{lem:wd-monotonically-refine}.
	\end{proof}

	We now apply \Cref{thm:from-Rosendal-to-walks}; hence towards a contradiction suppose there are a brick partition $\cZ$, two monotone epimorphisms $\alpha_0, \alpha_1 \colon \cZ \to \cW$, and a reduced walk $z$ on $\cZ$ such that $\alpha_i \cdot z \refex w_i$, for $i=0,1$, and
	\begin{equation}
		\label{eq:close-functions}
		d_{\sup}(\rmap{\cW}{\cU} \alpha_0, \rmap{\cW}{\cU} \alpha_1) \le 1.
	\end{equation}

	Let $\bar z$ be the shortest initial segment of $z$ such that there is $i \in \set {0,1}$ for which $\alpha_i \cdot \bar z \refex w_i$.
	Necessarily, $\alpha_i \cdot \bar z$ is confined in $w_i$.
	Let us suppose $i= 0$, the case $i=1$ being symmetrical.

	We will find a contradiction by examining $\wind_C(\rmap{\cW}{\cS} \alpha_1 \cdot \bar z)$.
	On one hand $\alpha_1 \cdot \bar z$ refines and is confined to an initial segment of the spaced path $w_1$, so by \Cref{claim:winding-the-****-around-c}.(a1) and \Cref{lem:wd-refine-initial-segment}
	\begin{equation}
		\label{eq:wd-2-z}
		\wind_C(\rmap{\cW}{\cS} \alpha_1 \cdot \bar z) \le M_v + M_c.
	\end{equation}

	On the other hand,
	write $\bar z$ as $\bar z_v\conc \bar z_0$, where $\bar z_0$ is the longest final segment such that
	\begin{equation}
		\label{eq:def-of-bar-z_1}
		\rmap{\cU}{\cS} (\ball{1}{\rmap{\cW}{\cU} \alpha_0 \cdot \bar z_0}) \sub C.
	\end{equation}
	Notice that $\bar z_v$ might be the empty walk (this is the case if $C = \cS$).
	In particular, $\alpha_0 \cdot \bar z_0$ refines and is confined to a final segment of $w_0$ whose image under $\rmap{\cW}{\cV}$ contains $\bar w_0$, and moreover $\alpha_0 \cdot \bar z_0(-1) = w_0(-1)$, so by \Cref{claim:winding-the-****-around-c}.(b0) and \Cref{lem:wd-refine-initial-segment}
	\begin{equation}
		\label{eq:wd-bar-z_1}
		\wind_C(\rmap{\cW}{\cS} \alpha_0 \cdot \bar z_0) > 2M_v + M_c +2.
	\end{equation}

	By \eqref{eq:close-functions} and \eqref{eq:def-of-bar-z_1} we know that $\rmap{\cW}{\cS} \alpha_1 \cdot \bar z_0$ is confined to $C$. We moreover have that
	\begin{equation*}
		d_{\cU}(\rmap{\cW}{\cU} \alpha_0 \cdot \bar z_0 (i), \rmap{\cW}{\cU} \alpha_1 \cdot \bar z_0(i)) \le 1 \text{ for each } i<\len(\bar z_0),
	\end{equation*}
	and
	\begin{multline*}
		d_{\cU}(\rmap{\cW}{\cU} \alpha_0 \cdot \bar z_0 (j), \rmap{\cW}{\cU} \alpha_1 \cdot \bar z_0(j+1)) \le d_{\cU}(\rmap{\cW}{\cU} \alpha_0 \cdot \bar z_0 (j), \rmap{\cW}{\cU} \alpha_0 \cdot \bar z_0(j+1))+ \\ + d_{\cU}(\rmap{\cW}{\cU} \alpha_0 \cdot \bar z_0 (j+1), \rmap{\cW}{\cU} \alpha_1 \cdot \bar z_0(j+1))\le 2,
	\end{multline*}
	for each $j<\len(\bar z_0) -1$.
	Using \eqref{eq:U-S_2-to-1} we get $d_{\cS}(\rmap{\cW}{\cS} \alpha_0 \cdot \bar z_0 (j), \rmap{\cW}{\cS} \alpha_1 \cdot \bar z_0(j+1)) \le 1$, so we can apply \Cref{lem:close-walks} and \eqref{eq:wd-bar-z_1} to conclude
	\begin{equation}
		\label{eq:wd-2-z_1}
		\wind_C(\rmap{\cW}{\cS} \alpha_1 \cdot \bar z_0) > 2M_v + M_c.
	\end{equation}

	If $\bar z_v$ is the empty walk then $\wind_C(\rmap{\cW}{\cS} \alpha_1 \cdot \bar z) = \wind_C(\rmap{\cW}{\cS} \alpha_1 \cdot \bar z_0)$, so we obtain a contradiction from \eqref{eq:wd-2-z} and \eqref{eq:wd-2-z_1}.
	Otherwise $\alpha_1 \cdot \bar z_v$ refines an initial segment of $w_1$ and $\rmap{\cW}{\cV}\alpha_1(\bar z_v(-1))$ is contained in $v$, and not in $\bar w_1$: indeed
	by definition of $\bar z_0$ the vertex $\rho^\cW_\cU(\alpha_0(\bar z_v(-1)))$ must be $E$-adjacent to one in $\cU \setminus (\rmap{\cU}{\cS})^{-1}(C)$, and this in combination with
	\eqref{eq:close-functions} yields
	\begin{equation*}
		d_{\cU}(\rmap{\cW}{\cU}\alpha_1(\bar z_v(-1)), \cU \setminus (\rmap{\cU}{\cS})^{-1}(C)) \le 2,
	\end{equation*}
	whereas $\rmap{\cU}{\cS}( \ball{2}{\rmap{\cV}{\cU} \cdot \bar w_1}) = \rmap{\cU}{\cS}( \ball{2}{\rmap{\cV}{\cU} \cdot c_+}) \sub C$, by \eqref{eq:annuli}.
	\Cref{lem:wd-refine-initial-segment} and the definition of $M_v$ then give
	\begin{equation}
		\label{eq:wd-2-z_v}
		\Card{\wind_C(\rmap{\cW}{\cS} \alpha_1 \cdot \bar z_v)} \le M_v.
	\end{equation}

	Combining \eqref{eq:wd-2-z_1} and \eqref{eq:wd-2-z_v} we conclude $\wind_C(\rmap{\cW}{\cS} \alpha_1 \cdot \bar z) > M_v + M_c$, again a contradiction with \eqref{eq:wd-2-z}.
\end{proof}

In the next three subsections we verify that each of the three conditions considered in \Cref{thm:the-theorem-introduction} is enough to guarantee the existence of a robust cycle, and therefore is covered by \Cref{thm:the-theorem}.

\subsection{Locally non-planar open subset}

Useful examples of Peano continuum with no locally separating points which have locally non-planar open subsets to keep in mind throughout this subsection are the Menger curve, or any $n$-dimensional closed manifold, for $n \ge 3$.

The next proposition is a combinatorial translation of the Disjoint Arcs Property (see \cite{MR1373283}*{Theorem 1}).
\begin{proposition}[Disjoint Paths Property]
	\label{prop:disjoint-path-property}
	Let $X$ be a Peano continuum with no locally separating points, and let $O \sub X$ be a locally non-planar open set.
	Let $\cU$ be a brick partition of $X$ and $u, u'$ be paths on $\cU$, such that $u$ is confined to $\cU(O)$.
	There are a brick partition $\cV \refines \cU$ and paths $v, v'$ on $\cV$ which are disjoint and such that $\rmap{\cV}{\cU} \cdot v, \rmap{\cV}{\cU} \cdot v'$ monotonically refine $u, u'$, respectively.
	If $u$ is circular, $v$ can be chosen to be circular.
\end{proposition}
\begin{proof}
	For each $1 \le i< \len(u)$, let $x_i$ be a point in $\Cl(u(i-1))\cap \Cl(u(i)) \cap (u(i-1) \vee u(i))$, and, for each $1 \le j< \len(u')$, let $x'_j$ be a point in $\Cl(u'(j-1))\cap \Cl(u'(j)) \cap (u'(j-1) \vee u'(j))$.
	We can assume that all these points are distinct by \Cref{fact:points-on-boundaries}.\eqref{itm:nls}, since $u$ and $u'$ are paths, so for each $i < \len(u)$ there is at most one $j<\len(u')$ such that $u(i) = u'(j)$, and vice versa.

	If $u(i) \not = u'(j)$ for every $j < \len(u')$, let $\gamma_i$ be any arc in $u(i) \cup \set{x_i, x_{i+1}}$ connecting $x_i$ to $x_{i+1}$.
	Similarly, if for $j<\len(u')$ there is no $i$ such that $u'(j) =u(i)$, let $\gamma'_j$ be any arc in $u'(j) \cup \set{x'_j, x'_{j+1}}$ connecting $x'_j$ to $x'_{j+1}$.

	Otherwise, for each $U = u(i) = u'(j)$, apply \Cref{fact:existence-of-arcs}.\eqref{itm:existence-of-arcs-cross} (note that in this case $U \subseteq O$) to find two disjoint arcs $\gamma_i, \gamma'_j$ in $U \cup \set{x_i, x_{i+1}, x'_j, x'_{j+1}}$ connecting $x_i$ to $x_{i+1}$ and $x'_j$ to $x'_{j+1}$, respectively.

	Let $\lambda, \lambda'$ be the two disjoint arcs resulting from the concatenations of the $\gamma_i$'s and $\gamma'_j$'s, respectively.
	Let $\tilde \cV \refines \cU$ be a brick partition fine enough that $\Star(\lambda, \tilde \cV) \cap \Star(\lambda', \tilde \cV) = \emptyset$.
	Set $P_i \coloneqq \Star(\lambda, \tilde \cV(u(i))), P'_j \coloneqq \Star(\lambda', \tilde \cV(u'(j)))$, for each $i< \len(u), j<\len(u')$, and let $\cV$ be the amalgam of $\tilde \cV$ refining $\cU$ and containing $\bigvee P_i$ and $\bigvee P'_j$, for all $i< \len(u), j<\len(u')$.
	Then $v \coloneqq \walk{\bigvee P_0, \dots,\bigvee P_{\len(u)-1}}$ and $v' \coloneqq \walk{\bigvee P'_0, \dots, \bigvee P'_{\len(u')-1}}$ are paths on $\cV$ satisfying the conclusions of the proposition.

	If $u$ is circular, we can take $x_0 = x_{\len(u)}$ in $\Cl(w(-1)) \cap \Cl(w(0)) \cap (w(-1) \vee w(0))$, and arcs $\gamma_0, \gamma_{\len(u)}$ connecting $x_0$ to $x_1, x_{\len(u)-1}$, respectively, and add `$\!\!\mod \len(u)$' where needed.
\end{proof}

The \nameref{prop:disjoint-path-property} allows us to find a robust cycle in any locally non-planar open set.

\begin{theorem}
	\label{thm:local-hp-robust-cycle0}
	Let $X$ be a Peano continuum with no locally separating points. If $X$ has a locally non-planar open subset then it has a robust cycle.
\end{theorem}
\begin{proof}
	Suppose that $O\subseteq X$ is open and locally non-planar.
	Fix any brick partition $\cW$ of $X$ which is fine enough that there are $W, W'$ which are in the same $E$-connected component of $\cW(O)$ but are not adjacent.
	Let $w$ be a path from $W$ to $W'$ in $\cW(O)$.
	By \nameref{prop:path-doubling} (\Cref{prop:path-doubling}), there are a brick partition $\cS \refines \cW$ and two paths $s_{0},s_{1}$ on $\cS$ which are disjoint except for their endpoints and such that $\rmap{\cS}{\cW} \cdot s_{i}$ monotonically refines $w$, for $i < 2$.
	Let $C$ be the circular path which follows $s_0$ forward and then $s_1$ backwards.
	Up to a refinement, we can assume $C$ is spaced, by \nameref{lem:path-to-spaced} (\Cref{lem:path-to-spaced}).
	Since $s_{0}, s_{0}$ have length at least $3$, $C$ has circular length at least $4$.

	By \nameref{lem:path-to-spaced} twice, we find a brick partition $\cU \refines \cS$ and a spaced circular path $D$ on $\cU$ such that $\rmap{\cU}{\cS} \cdot D$ monotonically refines $C$ and
	\begin{equation}
		\label{eq:small-D-lol}
		\rmap{\cU}{\cS}(\ball{2}{D}) \sub C.
	\end{equation}

	Let $u$ be the empty walk and suppose that a brick partition and a walk $(\cV, v) \refines (\cU, u)$ are given. We claim that it suffices to find $(\cV',v')\refines(\cV,v)$ with $v'$ a path, and a circular path $c$ disjoint from $v'$ satisfying \eqref{eq:fattening-in-C} and \eqref{eq:winds-at-least-once} in \Cref{def:robust}. 	Indeed, in this case, by \nameref{prop:path-doubling} we can assume that, up to a refinement, there are two paths $v'_0, v'_1$
	which are disjoint except in their endpoints, disjoint from $c$, and such that $\rmap{\cV'}{\cV} \cdot v'_i \refex v$.
	Let $P$ be a path from $v'_0 \cup v'_1$ to $c$, and say without loss of generality that it starts in $v'_0$.
	Then the lasso that follows $v'_1$ from beginning to end, then $v'_0$ backwards until meeting $P$, then follows $P$, then does a complete lap following $c$, is as required.

	We now proceed to find $v'$ and $c$ as above. Up to a refinement, we can suppose that $v$ is a path, by \nameref{thm:walks-to-paths} (\Cref{thm:walks-to-paths}).
	Use \nameref{lem:lemmetto-monotone} (\Cref{lem:lemmetto-monotone}) to find a spaced circular path $D'$ on $\cV$ such that $\rmap{\cV}{\cU} \cdot D'$ monotonically refines $D$.
	Since $D'$ is confined in $\cV(O)$, by \nameref{prop:disjoint-path-property} (\Cref{prop:disjoint-path-property}) there are a brick partition $\cV' \refines \cV$ and two disjoint paths $v', c$ on $\cV'$ such that $\rmap{\cV'}{\cV} \cdot v', \rmap{\cV'}{\cV} \cdot c$ monotonically refine $v$ and $D'$, respectively, and $c$ is circular.
	Then $c$ satisfies \eqref{eq:fattening-in-C} by \eqref{eq:small-D-lol}, and
	it follows by \Cref{lem:wd-monotonically-refine} that it satisfies \eqref{eq:winds-at-least-once}.
\end{proof}

\subsection{Planar open set with non-locally-separating simple closed curve}

The next two results capture the essential aspects behind the assumption of the existence of a non-locally-separating simple closed curve in \Cref{thm:the-theorem-introduction}.
Good examples to keep in mind are the closed disk and the Sierpi\'{n}ski carpet.

We start with an auxiliary definition.
\begin{definition}
	\label{def:gamma-compatible}
	Let $X$ be a Peano continuum, $\gamma$ be a simple closed curve in $X$, and $\cS$ be a brick partition of $X$.
	A spaced circular path $c = \langle c(0), \dots, c(\ell-1), c(0) \rangle$ on $\cS$ is \emph{$\gamma$-compatible} if
	\begin{enumerate}
		\item \label{itm:contains-gamma} $\set{c(i) : i < \ell} = \Star(\gamma, \cS)$;
		\item \label{itm:intersection-1} $\set{c(i) \cap \gamma : i < \ell}$ is a brick partition of $\gamma$.
	\end{enumerate}
\end{definition}

Throughout this subsection we make repeated use of \Cref{fact:points-on-boundaries}.\eqref{itm:nls_set} and \Cref{fact:existence-of-arcs}.\eqref{itm:existence-of-arcs-nls-arc} applied to a non-locally-separating simple closed curve $\gamma$ in a Peano continuum with no locally separating points.
Indeed, this is possible the since simple closed curves in Peano continua with no locally separating points necessarily have empty interior.

\begin{lemma}
	\label{lemma:fine-gamma-compatible}
	Let $X$ be a Peano continuum with no locally separating points and let $\gamma$ be a non-locally-separating simple closed curve in an open planar subset $O \sub X$.
	For each brick partition $\cU$ of $X$, there are a brick partition $\cV \refines \cU$ and a $\gamma$-compatible spaced circular path $c$ on $\cV$.
\end{lemma}
\begin{proof}
	Up to a refinement we can assume that if $U \in \Star(\gamma, \cU)$, then $\Cl(U) \sub O$.
	By Jordan's theorem, we can embed $O$ in the unit disk in the plane in such a way that $\gamma$ is mapped to the unit circle, since $\gamma$ is non-separating in $O$.
	Fix the usual $\bbR^2$ metric.
	\begin{claim}
		For each $U \in \Star(\gamma, \cU)$, it holds that $\gamma \cap \Cl(U)$ has finitely many connected components, all of which are non-trivial.
		Moreover $\gamma \cap \Bd(U)$ is totally disconnected.
	\end{claim}
	\begin{proof}[Proof of the Claim]
		Let $\varepsilon$ be less than half of $\min\{\diam(U) : U\in\cU\}$, fix once and for all some $U \in \cU$, and let $\delta$ be the uniform local connectedness constant of $U$ relative to $\varepsilon$.
		We prove that if two points $y_0, y_1 \in \gamma \cap \Cl(U)$ have distance less than $\delta$, then they belong to the same connected component.
		Indeed, in this case there is an arc $\lambda$ from $y_0$ to $y_1$ in $(U \setminus \gamma) \cup \set{y_0, y_1}$ of diameter less or equal than $\varepsilon$.
		The arc $\lambda$ separates $O$ in at least two regions, one of which meets $U$, has diameter less that $2\varepsilon$, and contains in its interior an open interval $I \sub \gamma$ with endpoints $y_0, y_1$.
		But then such region is too small to contain any $U' \in \cU$.
		Moreover, it cannot even meet any $U' \ne U$, since $\lambda$ would separate $U'$ and $\lambda$ is contained in $\Cl(U)$.
		Therefore it has to be contained in $\Cl(U)$.
		But then $I \sub U$, and thus $y_0, y_1$ are in the same connected component of $\gamma \cap \Cl(U)$.
		This shows that there are only finitely many components.
		It also shows that $\gamma \cap \Bd(U)$ is totally disconnected, since the open interval between any two points in the same connected component of $\gamma \cap \Cl(U)$ is contained in $U$.

		Suppose towards a contradiction that $\set{x}$ is a trivial connected component of $\Cl(U) \cap \gamma$, so in particular $x \in \Bd(U)$.
		Let $\delta_0$ be the minimum between $\delta$ and the uniform local connectedness constant of $\runion \Star(x, \cU \setminus \set{U})$ relative to $\varepsilon$.
		Take two points $y_0, y_1 \in \gamma \cap \Star(x,\cU \setminus \{U\})$ whose distance is smaller than $\delta_0$ and such that $x \in \gamma(y_0, y_1)$.
		Then there is an arc $\lambda$ of diameter less than $\varepsilon$ connecting $y_0, y_1$ in $\runion \Star(x, \cU \setminus \set{U})$.
		In particular $\lambda \sub O$, so it separates $O$ in at least two regions, one of which contains $x$ in its interior and has diameter less that $2\varepsilon$.
		Arguing as in the first part of the proof we conclude that $U$ cannot be contained in such region, so $x \not \in \Cl(U)$, a contradiction.
	\end{proof}

	Let $\gamma_0, \dots, \gamma_{n-1}$ be the collection of connected components of $\gamma$ in $\Cl(U)$, for $U$ ranging in $\cU$.
	Any three of them have empty intersection, since otherwise $\gamma$ would contain a triod.
	Moreover, if $i \ne j$, then $\Card{\gamma_i \cap \gamma_j} \le 1$.
	Indeed, since they are arcs in the same circle, and $n\ge 3$, their intersection can only by empty, a point, or a nontrivial arc.
	But such an intersection is contained in some $\Bd(U)$, and thus is totally disconnected.
	It follows that their interior in $\gamma$ form a brick partition of $\gamma$, and we can assume that they are ordered so that $\gamma_0 \E \gamma_1 \E \cdots \E \gamma_{n-1} \E \gamma_0$.

	Let $\cV' \refines \cU$ be a brick partition fine enough that $\Star(\gamma_i, \cV')$ and $\Star(\gamma_j, \cV')$ have distance at least 2 in $\cV'$, for any $\Card{i-j} > 1 \mod n$.
	Let $\cV$ be an amalgam of $\cV'$ refining $\cS$ and containing $P_i \coloneqq \bigvee \Star(\gamma_i, \cV'(U))$, where $U\in\cU$ is such that $\gamma_i\subseteq\Cl(U)$ and $i < n$.
	Then $c \coloneqq \walk{P_0, \dots, P_{n-1}, P_0}$ is a $\gamma$-compatible spaced circular walk.
\end{proof}

It is now not hard to see that we can zoom-in and disjoin any path from a $\gamma$-compatible spaced circular path.

\begin{proposition}
	\label{prop:disjoin-gamma-compatible}
	Let $X$ be a Peano continuum with no locally separating points and let $\gamma$ be a non-locally-separating simple closed curve contained in an open planar subset $O \sub X$.
	For each brick partition $\cV$ of $X$ and path $v$ on $\cV$, there are a brick partition and a path $(\cV', v') \refines (\cV, v)$, and a $\gamma$-compatible spaced circular path $c$ on $\cV'$ disjoint from $v'$.
\end{proposition}
\begin{proof}
	For each $1 \le i< \len(v)-1$, let $x_i$ be a point in $\Cl(v(i-1))\cap \Cl(v(i)) \cap (v(i-1) \vee v(i))$ which does not belong to $\gamma$.
	Let $\lambda_i$ be an arc in $v(i) \cup \set{x_{i-1}, x_{i}}$ connecting $x_{i-1}$ to $x_{i}$ and disjoint from $\gamma$.

	Let $\lambda$ be the concatenation of the $\lambda_i$'s and
	let $\widetilde{\cV} \refines \cV$ be a brick partition fine enough that
	\begin{equation} \label{eq:gamma_lambda}
		\Star(\lambda, \widetilde{\cV}) \cap \Star(\gamma, \widetilde{\cV}) = \emptyset,
	\end{equation}
	and such that there is a $\gamma$-compatible spaced circular path $c$ on $\widetilde{\cV}$, which is possible by \Cref{lemma:fine-gamma-compatible}.
	Notice that $c$ is confined to $\Star(\gamma, \widetilde{\cV})$.

	Set $P_j \coloneqq \bigvee \Star(\lambda, \widetilde \cV(v(j)))$ for each $j< \len(v)$, and let $\cV'$ be any amalgam of $\widetilde \cV$ refining $\cV$ which contains all $P_j$ and all elements of $c$.
	Then $v' \coloneqq \walk{P_0, \dots, P_{\len(v)-1}}$ is a path that is disjoint from $c$ by \eqref{eq:gamma_lambda}.
\end{proof}

We are ready to prove that $\gamma$-compatible spaced circular paths are robust.

\begin{theorem}
	\label{thm:local-hp-robust-cycle1}
	Let $X$ be a Peano continuum with no locally separating points. If $X$ has a planar open set containing a simple closed curve which is not locally separating
	then it has a robust cycle.
\end{theorem}
\begin{proof}
	Let $O \subseteq X$ be a planar open set and $\gamma \subseteq O$ a simple closed curve which is not locally separating.
	Fix any brick partition $\cW$ of $X$ which is fine enough that $\Card{\Star(\gamma, \cW)} \ge 4$.
	Let $\cS \refines \cW$ be a brick partition and let $C$ be $\gamma$-compatible spaced circular path $C$ on $\cS$, using \Cref{lemma:fine-gamma-compatible}.
	By the assumption on $\cW$, the circular length $\ell$ of $C$ is at least $4$.

	Let $\cU_0 \refines \cS$ be a brick partition fine enough that
	\begin{equation}
		\label{eq:thin-star}
		d_{\cU_0}(\Star(\gamma, \cU_0), \cU_0(\cS \setminus C)) > 2.
	\end{equation}
	By \Cref{lemma:fine-gamma-compatible}, there are a brick partition $\cU \refines \cU_0$ and a $\gamma$-compatible spaced circular path $D$ on $\cU$.
	Let $u$ be the empty walk on $\cU$.

	Suppose a brick partition and a walk $(\cV, v) \refines (\cU, u)$ are given. As in \Cref{thm:local-hp-robust-cycle0}, it is sufficient to find a brick partition and a path
	$(\cV', v') \refines (\cV, v)$ and a circular path $c$ disjoint from $v'$ satisfying the conditions in \Cref{def:robust}.
	Up to a refinement, we can suppose that $v$ is a path, by \nameref{thm:walks-to-paths} (\Cref{thm:walks-to-paths}).
	By \Cref{prop:disjoin-gamma-compatible}, there are a brick partition and a path $(\cV', v') \refines (\cV, v)$ and a $\gamma$-compatible spaced circular path $c$ on $\cV'$ which is disjoint from $v'$.

	Now $\rmap{\cV'}{\cU} \cdot c$ is confined in $D$, by point \eqref{itm:contains-gamma} of \Cref{def:gamma-compatible}, so \eqref{eq:fattening-in-C} follows from \eqref{eq:thin-star}.
	Since changing the starting vertex of $C$ does not affect the winding number, and flipping the direction of $C$ just flips the sign of $\wind_C(\rmap{\cV'}{\cS} \cdot c)$, we can assume that $c(0) \in \cV'(C(0))$ and that the first time $c$ exits $\cV'(C(0))$ it enters $\cV'(C(1))$.
	Item \eqref{itm:intersection-1} in \Cref{def:gamma-compatible} then implies that both $c$ and $C$ induce brick partitions of $\gamma$, with the former refining the latter.
	It follows that $\rmap{\cV'}{\cS} \cdot c$ monotonically refines $C$, thus by \Cref{lem:wd-monotonically-refine} we can conclude that $\Card{\wind_C(\rmap{\cV'}{\cS} \cdot c)} = \ell$.
\end{proof}

\subsection{Circular covering}

We recall from the introduction that a Peano continuum $X$ has a circular covering if it admits a covering $O_0, \dots, O_{\ell-1}$ in $\ell\ge 4$ connected open subsets such that $\Cl(O_i) \cap \Cl(O_j) \neq \emptyset$ if and only if $\Card{i-j} \le 1 \mod \ell$, and $O_i \setminus \bigcup_{j\ne i} O_j$ is connected for each $i<\ell$.
Examples of Peano continua with no locally separating points admitting a circular covering are the torus, the Menger curve or the Sierpi\'{n}ski carpet.

\begin{proposition}
	\label{prop:circular-bp-equiv-circular-pc}
	A Peano continuum has a circular covering if and only if it admits a brick partition whose nerve graph is a cycle of length at least $4$.
\end{proposition}
\begin{proof}
	Suppose that $X$ has a brick partition $\cS$ whose nerve graph is a cycle of length $\ell \ge 4$.
	Say the elements of $\cS$ are $S_0 \E S_1 \E \cdots \E S_{\ell-1} \E S_0$.
	Let $\cW \refines \cS$ be given by \Cref{prop:approximate-core-refinement} applied to $\cU = \cV = \cS$.
	We can moreover pick $\cW$ fine enough so that no element of $\Star(\Cl(S_i), \cW)$ is $E$-adjacent to one of $\Star(\Cl(S_{i+2 \mod \ell}), \cW)$ for every $i < \ell$.
	Note furthermore that every $\Star(\Cl(S_i), \cW)$ is $E$-connected, so setting
	\begin{equation*}
		O_i \coloneqq \runion \Star(\Cl(S_i), \cW),
	\end{equation*}
	we get a covering $O_0, \dots, O_{\ell-1}$ whose pieces are connected and intersect as required.
	Moreover, $O_i \setminus \bigcup_{j\ne i} O_j = \runion \Core(S_i, \cW)$, which is connected by \Cref{prop:approximate-core-refinement}.

	Conversely, suppose $X$ has a circular covering $O_0, \dots, O_{\ell-1}$.
	Define
	\begin{equation*}
		C_i \coloneqq O_i \setminus \bigcup_{j\ne i} O_j.
	\end{equation*}
	These sets are closed and pairwise disjoint, since $x \in C_j \cap C_k$ with $j \not = k$ if and only if $x \not \in O_i$ for every $i < \ell$, which is not possible.

	Let $\cU$ be a brick partition refining $\set{O_0, \dots, O_{\ell-1}}$, which is fine enough that
	\begin{equation} \label{eq:Ci_Cj}
		\Star(C_i, \cU) \cap \Star(C_j, \cU) = \emptyset, \text{ for each } i\ne j,
	\end{equation}
	and
	\begin{equation} \label{eq:Oi_Oj}
		\begin{gathered}
			d_{\cU}\left(\Star\left(O_{i} \cap O_{i+1 \mod \ell}, \mathcal{U}\right), \Star\left(O_{j} \cap O_{j+1 \mod \ell}, \mathcal{U}\right)\right) \geq 2, \text{ for each } i\ne j, \\
			d_{\cU}\left(\Star\left(O_{i} \cap O_{i+1 \bmod \ell}, \cU\right), \Star(C_j, \cU) \right) \geq 2, \text{ for all } j \notin \set{i, i+1 \mod \ell}.
		\end{gathered}
	\end{equation}

	Since $C_i$ is connected, $\cU_i \coloneqq \Star(C_i, \cU)$ is $E$-connected, and all $\cU_i$'s are pairwise disjoint by \eqref{eq:Ci_Cj}.
	Given $i < \ell$, divide $\Star(O_i \cap O_{i+1 \mod \ell}, \cU )$ into two sets $\cP_i, \cQ_{i+1 \mod \ell}$ based on whether $U$ is such that
	\begin{equation*}
		d_\cU(U, \cU_i) \le d_\cU(U, \cU_{i+1 \mod \ell}) \text{ or } d_\cU(U, \cU_{i+1 \mod \ell}) < d_\cU(U, \cU_i).
	\end{equation*}
	Notice that, by \eqref{eq:Oi_Oj}, if $U \in \cP_i$ then all the elements of any path of minimal length, call it $P$, connecting $U$ to $\cU_i$ are either in $\cU_i$, or in $\Star(O_i \cap O_{i+1 \mod \ell}, \cU )$, and therefore in $\cP_i$. Indeed, if $U'$ is the first step in $P$ that does not belong to $\Star(O_i \cap O_{i+1 \mod \ell}, \cU )$, then by \eqref{eq:Oi_Oj} we have $U' \cap O_j = \emptyset$ for all $j \ne i, i+1 \mod \ell$. It thus follows that $U'$ has to belong to either $\cU_i$ or $\cU_{i+1 \mod \ell}$ (but not to both, by \eqref{eq:Ci_Cj}), and the latter is not possible as otherwise $U \in \cQ_i$.

	A similar observation holds for $\cQ_i$. We can thus conclude that the sets $\cQ_i \cup \cU_i \cup \cP_i$ are $E$-connected and form a partition of $\cU$.
	The amalgam $\cS$ of $\cU$ obtained from such partition is the desired cycle of $\ell$ elements.
\end{proof}

We now show that the circular brick partition found in \Cref{prop:circular-bp-equiv-circular-pc} is robust.

\begin{theorem}
	\label{thm:circular-covering-robust}
	Let $X$ be a Peano continuum with no locally separating points.
	If $X$ has a circular covering then it has a robust cycle.
\end{theorem}
\begin{proof}
	By \Cref{prop:circular-bp-equiv-circular-pc}, there is a brick partition $\cS$ of $X$ whose nerve graph is a cycle of length $\ell \ge 4$.
	Let $C$ be a circular path around $\cS$ (in particular, the circular length of $C$ is $\ell$).
	By \nameref{thm:walks-to-paths} (\Cref{thm:walks-to-paths}), there are a brick partition $\cU \refines \cS$ and a spaced path $u$ such that $\rmap{\cU}{\cS} \cdot u$ {monotonically} refines $C\conc C$, which is possible since $C \conc C$ is uncrossed.
	In particular, $\wind_C(\rmap{\cU}{\cS} \cdot u) = 2\ell$, by \Cref{lem:wd-monotonically-refine}.

	Suppose a brick partition and a walk $(\cV, v) \refines (\cU, u)$ are given.
	Suppose without loss of generality that $v(0) \in \cV(C(0))$. Note that for any initial segment $t$ of $v$ such that $\rmap{\cV}{\cS}(t(-1)) \in C(k)$ we get
	\begin{equation} \label{eq:modulo}
		\wind_C(\rmap{\cV}{\cS} \cdot t) \equiv k \mod \ell.
	\end{equation}

	Let $t_0$ be the shortest initial segment of $v$ such that
	$\Card{\wind_C(\rmap{\cV}{\cS} \cdot t_0)} > \ell$.
	Such initial segment exists by \Cref{lem:wd-refine-initial-segment} applied to $v$ and $u$, and let
	$k_0< \ell$ be such that $\rmap{\cV}{\cS}(t_0(-1)) \in C(k_0)$;
	then \eqref{eq:modulo} implies $k_0>0$.

	Using \nameref{prop:path-doubling} (\Cref{prop:path-doubling}) we find a brick partition $\cV' \refines \cV$ and two paths $v_0,v_1$ which are disjoint except at their ends and such that $\rmap{\cV'}{\cV}\cdot v_i$ monotonically refines $v$, for $i < 2$.
	Then \Cref{lem:wd-monotonically-refine} gives
	\begin{equation}
		\label{eq:v_i-same-wd}
		\wind_C(\rmap{\cV'}{\cS}\cdot v_i) = \wind_C(\rmap{\cV}{\cS} \cdot v), \text{ for } i =0,1.
	\end{equation}

	Let $s_i$ be an initial segment of $v_i$ such that $\rmap{\cV'}{\cV}\cdot s_i$ monotonically refines $t_0$.
	Then, again by \Cref{lem:wd-monotonically-refine},
	\begin{equation*}
		\wind_C(\rmap{\cV'}{\cS}\cdot s_i) = \wind_C(\rmap{\cV}{\cS} \cdot t_0) > \ell.
	\end{equation*}
	To lighten notation, let us refer to $\wind_C(\rmap{\cV'}{\cS}\cdot r)$ as $\omega_r$, for any walk $r$ on $\cV'$.

	Let us partition the set of the $V$'s in $\cV'(C(k_0))$ which belong to $v_0 \cup v_1$ into two sets $A, B$, depending on whether $|\omega_r | > \ell$ or $|\omega_r | < \ell$, where $r$ is the initial segment of either $v_0$ or $v_1$ ending in $V$ (by \eqref{eq:modulo} the value $\ell$ is never achieved).
	Both of these sets are nonempty, as witnessed by the $s_i$'s.
	Let $z$ be a path from $A$ to $B$ in $\cV'(C(k_0))$.
	There are two cases.

	\textbf{Case 1:} There is $i < 2$ such that both endpoints of $z$ are contained in $v_i$, and we might as well assume $i = 1$.
	Let $v' \conc c$ be the lasso path in $\cV'$ which follows $v_0$ from beginning to end, then $v_1$ backwards until it hits $z$ for the first time (this is $v'$), then follows $z$ (possibly backwards), then $v_1$ forwards until it hits $z$ again (this is $c$).
	Since $z$ is included in $\cV'(C(k_0))$ it does not contribute to the winding number of $c$, which is therefore equal to the difference $\omega_{r_0} - \omega_{r_1}$ where $r_0, r_1$ are the initial segments of $v_1$ at either end of $z$.  By \eqref{eq:modulo} we have $\omega_{r_0} \equiv \omega_{r_1} \mod \ell$.
	Furthermore, we have that either $r_0 \in  A$ and $r_1 \in B$ or viceversa, hence $|\omega_c| =| \omega_{r_0} - \omega_{r_1} | \ge\ell$, and \eqref{eq:winds-at-least-once}, is satisfied.

	\textbf{Case 2:} Otherwise, let $v' \conc c$ be the lasso path which follows $v_0$ from the beginning until it hits $z$ (this is $v'$), then continues following $v_0$ until the end, then $v_1$ backwards until it hits $z$, then follows $z$, possibly backwards (this is $c$).
	As above, $z$ does not contribute to the winding number of $c$, which is therefore equal to the difference $\omega_{r_0} - \omega_{r_1}$, where $r_0, r_1$ are the final segments of $v_0$ and $v_1$ at either end of $z$.
	But since $v_0, v_1$ have the same winding number by \eqref{eq:v_i-same-wd}, $\omega_{r_0} - \omega_{r_1}$ is equal to the difference of the winding numbers of the respective initial segments.
	We therefore conclude as in Case 1.

	In both cases, $v_0$ is an initial segment of $v' \conc c$, so $(\cV',v' \conc c ) \refines (\cV, v)$.
	Moreover, \eqref{eq:fattening-in-C} is trivially true, since $C$ covers all of $\cS$.
\end{proof}

We close this section showing that the sphere and the real projective plane are the only two closed surfaces without a circular covering.

A continuum is \emph{unicoherent} if it cannot be written as the union of two subcontinua whose intersection is disconnected.
If $X$ has a circular covering $O_0, \dots, O_{\ell-1}$, it is not unicoherent, as witnessed by $\Cl(O_0)$ and $\bigcup_{i \ne 0} \Cl(O_i)$.

\begin{proposition}
	\label{prop:vaguely-circular-surfaces}
	The only two closed surfaces without a circular covering are the sphere and the real projective plane.
\end{proposition}
\begin{proof}

	By \cite{MR0007095}*{Theorem 7.4}, the sphere and the real projective plane are unicoherent, therefore they do not have a circular covering.

	On the other hand, it is clear that the torus and the Klein bottle have circular coverings.
	Recall that, by the classification theorem for compact surfaces \cite{MR2766102}*{Theorem 6.15} every compact surface is either the sphere, a connected sum of copies of the torus, or a connected sum of copies of the real projective plane $\mathbb RP^2$. Since $\mathbb RP^2\#\mathbb RP^2$ is homeomorphic to the Klein bottle,
	we conclude by showing that the connected sum of a surface $X$ which has a circular covering, with any other surface $X'$, has a circular covering: suppose $O_0, \dots, O_{\ell-1}$ is a circular covering of $X$;
	take an open disk $D$ in $O_0 \setminus \bigcup_{j\ne 0} O_j$, and any open disk $D'$ in $X'$, and glue $X, X'$ along $\Bd(D), \Bd(D')$.
	Then $(O_0 \setminus D) \cup (X' \setminus D'), O_1, \dots, O_{\ell-1}$ is a circular covering of $X \# X'$.
\end{proof}

\section{Classifiability and Dynamics}
\label{sec:Minimal_spaces_of_chains_and_dynamical_consequences}

In this section we elaborate on the dynamical consequences of \Cref{thm:the-theorem-introduction} and we derive from it, under sufficient homogeneity conditions,
\Cref{thm:no_gpp}, \Cref{corollary:no_gpp} and \Cref{corollary:no_classification}
(see Corollaries \ref{thm:dynamical-consequences} and \ref{cor:no_ccs}).

Recall that a Polish group has the generic point property if every minimal flow has a comeager orbit.
Given a Peano continuum $X$ and a subgroup $G \le \Homeo(X)$, we can obtain the failure of the generic point property for $G$ directly from \Cref{thm:the-theorem-introduction},
if we can prove that the action $G \acts \chains(X)$ is minimal.
The first thing to notice is that a necessary prerequisite for this is the minimality the action $G \acts X$, since we recall that the map from $\chains(X)$ onto $X$ associating to each chain $\cC$ its root $\bigcap \cC$ is a continuous $\Homeo(X)$-equivariant map.
Therefore, for any compact manifold with boundary, the action of the homeomorphism group on the space of chains is not minimal.
Minimality of $G \acts X$ is not sufficient, though: while $\Homeo(S^1)$ acts minimally, even transitively on the circle $S^1$, $\Homeo(S^1)\acts \chains(X)$ has a proper closed invariant subset:
take the set of those chains $\cC$ such that for each $C \in \cC \setminus \set{S^1}$ the root of $\cC$ belongs to the boundary of $C$.

The following gives a sufficient condition for minimality of $G \acts \chains(X)$, when $X$ is a Peano continuum with no locally separating points.
\begin{proposition}\label{prop:sufficient-minimality}
	Let $X$ be a Peano continuum with no locally separating points.
	Let $G \le \Homeo(X)$ be such that $G \acts X$ is minimal.
	Suppose that for each $\cC \in \chains(X)$ with root $x_0$, and any connected nonempty open sets $U_0, U_1, V$ with $U_0, U_1 \sub V$, there are $C \in \cC$ and $g \in G$ such that:
	\begin{enumerate}
		\item $g(x_0) \in U_0$,
		\item $g[C] \cap U_1 \ne \emptyset$,
		\item $g[C] \sub V$.
	\end{enumerate}
	Then the action $G \acts \chains(X)$ is minimal.
\end{proposition}

The proof relies on the following straightforward lemma.

\begin{lemma}
	\label{lem:first-last}
	Let $X$ be a Peano continuum, let $\cU$ be a brick partition of $X$ and let $u$ be a spaced path on $\cU$.
	Let $\cC \in \chains(X)$ with root $x_0$ be such that $x_0 \in u(0)$ and suppose there is $C \in \cC$ such that $C \cap u(-1) \ne \emptyset$ and $C \sub \runion_{i<\len(u)} u(i)$.
	Then $\cC \in \opwalk{\cU}{u}$.
\end{lemma}
\begin{proof}
	Since $u$ is spaced, $x_0 \in C$, and $C \in O(u(0), \dots, u(-1))$, it holds that any subcontinuum $C'$ of $C$ which contains $x_0$ is such that
	\begin{equation}
		\label{eq:interseca-quelli-prima}
		\text{ if } C' \cap u(i) \ne \emptyset \text{ then } C' \cap u(j) \ne \emptyset, \text{ for each } j \le i.
	\end{equation}
	By maximality and connectedness of $\cC$, for each $i < \len(u)-1$, there is $x_0 \in C_i \sub C$ in $\cC$ such that $C_i \cap u(i) \ne \emptyset$ but $C_i \cap u(i+1) = \emptyset$.
	By \eqref{eq:interseca-quelli-prima}, $C_i \in O(u(0), \dots, u(i))$, and we are done.
\end{proof}

\begin{proof}[Proof of \Cref{prop:sufficient-minimality}]
	Let $\cC \in \chains(X)$ with root $x_0$ and a nonempty open $\bU \sub \chains(X)$ be given.
	By \Cref{cor:refinement_walk} and \nameref{thm:walks-to-paths} (\Cref{thm:walks-to-paths}), there are a brick partition $\cU$ of $X$ and a spaced path $u$ on $\cU$ such that $\opwalk{\cU}{u} \sub \bU$.

	Since $G$ acts minimally on $X$, there is $g_0 \in G$ such that $g_0(x_0) \in u(0)$.
	By hypothesis there are $g_0[C] \in g_0 \cdot \cC$ and $g_1 \in G$ such that $g_1 g_0( x_0) \in u(0)$, $g_1g_0[C] \cap u(-1) \ne \emptyset$ and $g_1 g_0[C] \sub \runion_{i< \len(u)}u(i)$.
	We conclude that $g_1 g_0 \cdot \cC \in \opwalk{\cU}{u}$, by \Cref{lem:first-last}.
\end{proof}

We briefly recall what it means for an equivalence relation on a Polish space to be classifiable by countable structures and refer the reader to \cite{MR1967835} or \cite{hjorth2000classification} for more details.
Given a countable relational language $\cL=\{R_i\}_{i\in I}$, where each $R_i$ has arity $n_i$, we topologize the collection $\mathrm{Mod}(\cL)$ of all $\cL$-structures with domain $\N$, by identifying it with the product space
\begin{equation*}
	\prod_{i\in I}2^{\N^{n(i)}}
\end{equation*}
where every $R_i\subseteq\N^{n(i)}$ is identified with its indicator function.
Consider now another Polish space $Y$ and an equivalence relation $\equiv$ on $Y$.
We say that $(Y,\equiv)$ is \emph{classifiable by countable structures} if there exists a Borel function $f\colon Y \to\mathrm{Mod}(\cL)$, for some countable relational language $\cL$, such that for all $y,y'\in Y$,
\begin{equation}
	y \equiv y'\iff f(y)\cong f(y').
\end{equation}

In order to prove \Cref{corollary:no_classification}, we rely on the notion of turbulence and on a classic theorem due to Hjorth \cite{hjorth2000classification}.
The following definitions and statements can be found in \cite{MR1967835}*{Sections 7 and 8}.

\begin{definition}
	Let $G \acts Y$ be a continuous action on a topological space.
	Given a point $y \in Y$, an open neighborhood $\bU$ of $y$, and an open neighborhood $G_0$ of the identity in $G$,
	the \emph{$(\bU, G_0)$-local orbit of $y$} is the set of all $y' \in \bU$ such that there is a finite sequence $y = y_0, \dots, y_n = y'$ of elements of $\bU$ verifying $y_{i+1} \in G_0 \cdot y_{i}$ for all $i< n$.
	We shall say that a point $y \in Y$ is \emph{turbulent} if for any $\bU, G_0$ as above, the closure of the $(\bU, G_0)$-local orbit of $y$ has nonempty interior.

	An action is \emph{generically turbulent} if it is minimal, has no comeager orbit and comeagerly many points are turbulent, see \cite{MR1967835}*{Proposition 8.7}.
\end{definition}

Recall that the \emph{orbit equivalence relation} $\equiv_G$ on $Y$ associated to $G \acts Y$ is defined as $y \equiv_G y'$ if and only if $y' \in G\cdot y$.

\begin{theorem}[\cite{hjorth2000classification}]
	\label{fact:Hjorth}
	Let $G$ be a Polish group and $G \acts Y$ be a continuous action on a Polish space.
	The action is generically turbulent if and only if the orbit equivalence relation $\equiv_G$ is not classifiable by countable structures.
\end{theorem}

\subsection{Strongly Locally Homogeneous Spaces}
In \Cref{thm:local-transitive} below we show that the action $G \acts \chains(X)$ is minimal and comeagerly many point are turbulent, as long as $X$ satisfies a strong form
of homogeneity and $G \le \Homeo(X)$ is large enough. We begin by making these last two conditions precise.

For $G \le \Homeo(X)$, a continuous action $G \acts X$ is called \emph{locally transitive}\footnote{Notice that in \cite{gutman_minimal} local transitivity refers to a slightly stronger property. Nevertheless, the two properties are equivalent when $X$ is a Peano continuum with no locally separating points.} if for any $x \in X$ and any neighborhood $V$ of $x$, $G_V \cdot x$ contains an open neighborhood of $x$, where
\begin{equation} \label{eq:rigid}
	G_V \coloneqq \set{g \in G : g(y) = y \text{ for all } y \not \in V}
\end{equation}
is the \emph{rigid stabilizer} of $V$ in $G$.

Spaces $X$ such that $\Homeo(X) \acts X$ is locally transitive are called \emph{strongly locally homogeneous} (see for instance \cite{MR2674033}).
All closed manifolds, as well as the Hilbert cube, are strongly locally homogeneous.
Furthermore, $\Homeo_0(X)$ acts locally transitively on $X$ whenever $X$ is a closed manifold, as do $\Homeo_+(X)$ and $\mathrm{Diffeo}(X)$ when $X$ is moreover orientable or smooth, respectively \cite{gutman_minimal}*{Examples 3.1}.
Other examples of strongly locally homogeneous Peano continua with no locally separating points are the Menger curve $\mu^1$ and the universal $k$-dimensional compactum $\mu^k$, for $k>1$ \cite{MR0920964}*{Theorem 3.2.2}.

\begin{proposition}[\cite{MR2674033}*{Proposition 7}]
	\label{fact:n-local-transitivity}
	If $X$ is a Peano continuum with no locally separating points and $G\le \Homeo(X)$ is such that $G \acts X$ is locally transitive, then for any connected open set $V \sub X$ and any bijection $f \colon A \to B$ between finite subsets of $V$, there is $g \in G_V$ extending $f$.
	In particular, $G$ acts on $X$ transitively.
\end{proposition}

As a consequence, strongly locally homogeneous Peano continuum with no locally separating points are homogeneous. Any Peano continuum which is homogeneous and is not the circle or a closed surface, is locally non-planar \cite{MR2674033}*{Remark 8}.
It then follows from \Cref{thm:the-theorem-introduction} and \Cref{prop:vaguely-circular-surfaces} that:
\begin{corollary}
	\label{cor:homo-chains-later}
	No homogeneous Peano continuum except for the circle and, possibly, the sphere and real projective plane, has a generic chain.
\end{corollary}

We now turn to the main theorem of this section.
We remark that minimality of $G \acts \chains(X)$, for $X$ a closed manifold of dimension at least 2 and $G \le \Homeo(X)$ a subgroup which acts locally transitively on $X$, was already proved in \cite{gutman_minimal}*{Theorem 6.5}\footnote{Gutman's theorem applies to Peano continua which are \emph{strongly arcwise inseparable}, a property that implies that points are not locally separating, but that fails for any space containing a separating arc, as is the case for the Menger curve \cite{MR0096180}.
}.
The following result applies to all the examples of strongly locally homogeneous Peano continua which were described at the beginning of the subsection. In particular, to the Menger curve and its group of homeomorphisms.

\begin{theorem}
	\label{thm:local-transitive}
	Let $X$ be a Peano continuum with no locally separating points.
	Let $G \le \Homeo(X)$ be such that $G \acts X$ is locally transitive.
	Then the action $G \acts \chains(X)$ is minimal and comeagerly many points are turbulent.
\end{theorem}
\begin{proof}
	By \Cref{fact:n-local-transitivity}, $G \acts X$ is minimal, so to prove minimality of $G \acts \chains(X)$ it is enough to check that the conditions of \Cref{prop:sufficient-minimality} are satisfied.
	To this end, let $\cC \in \chains(X)$ with root $x_0$ be given, together with nonempty open connected sets $U_0, U_1, V$ such that $U_0, U_1 \sub V$.

	By minimality of $G \acts X$, up to a translation, we can assume that $x_0 \in U_0$.
	Let $C \in \cC$ be non-trivial and such that $C \sub V$, and pick $x_1 \in C \setminus \set{x_0}$.
	Fix moreover $y_1 \in U_1$ distinct from $x_0$.
	By \Cref{fact:n-local-transitivity}, there is $g \in G_V$ such that $g(x_0) = x_0 \in U_0$ and $g(x_1) = y_1 \in U_1 \cap g[C]$.
	As $g$ is the identity outside $V$, we get $g[C] \sub V$, hence the action $G \acts \chains(X)$ is minimal by \Cref{prop:sufficient-minimality}.

	Towards proving that comeagerly many points of $\chains(X)$ are turbulent, fix
	a sequence $\cU_n$ of brick partitions of $X$ such that $\cU_{n+1} \refines \cU_n$ and the mesh of $\cU_n$ tends to zero.
	By \Cref{prop:induced-quasi-partition}, the set $Z \coloneqq \bigcap_{n \in \N} \bigcup \chains(\cU_n)$ is comeager in $\chains(X)$.
	Let $\cC \in Z$ with root $x_0$, an open neighborhood $\bU$ of $\cC$, and an open neighborhood $G_0$ of the identity in $G$ be given.
	Since $\cC \in Z$, there are $n \in \N$ and a walk $u$ on $\cU_n$ such that $\Imm(u) = \cU_n$ and $\cC \in \opwalk{\cU_n}{u} \sub \bU$.
	To lighten notation, let us rename $\cU_n$ to $\cU$.

	We prove that the closure of the $(\bU, G_0)$-local orbit of $\cC$ contains $\opwalk{\cU}{u}$.
	To this end, let $\bV \sub \opwalk{\cU}{u}$ be open and nonempty (notice that $\cC$ need not belong to $\bV$).
	By \Cref{cor:refinement_walk}, there are a brick partition and a walk $(\cV, v) \refines (\cU, u)$ such that $\opwalk{\cV}{v} \sub \bV$.
	Up to a refinement, by \nameref{thm:walks-to-paths}, we can suppose that $v$ is a spaced path, and that $\cV$ is fine enough that
	\begin{equation*}
		G_{\cV,1} \sub G_0,
	\end{equation*}
	where $G_{\cV,1}$ is defined in \eqref{eq:def-HV1}.
	Crucially, for any $V \E V'$ in $\cV$, the rigid stabilizer $G_{V \vee V'}$ (see \eqref{eq:rigid}) is a subset of $G_{\cV,1}$, and hence of $G_0$.
	We use this fact repeatedly from here onwards.

	By picking $\cV$ fine enough we can moreover assume there is $V_0 \in \cV(u(0))$ such that $x_0 \in \Cl(V_0)$ and such that $\ball{1}{V_0} \subseteq \cV(u(0))$. Fix a path
	$P$ from $V_0$ to $v(0)$ in $\cV(u(0))$.
	We will move $x_0$ along $P$, step by step.
	As a preliminary move, let $\cC_{-1} \coloneqq \cC$ and, using \Cref{fact:n-local-transitivity}, find $g_{-1} \in G_{\ball{1}{V_0}} \sub G_{\cV,1}$ mapping $x_0$ to some $y_0 \in V_0 = P(0)$
	and set $\cC_0 \coloneqq g_{-1} \cdot \cC$. Note that $\cC_0 \in \opwalk{\cU}{u}$ since $\cC \in \opwalk{\cU}{u}$ and $g_{-1}$ fixes every element of $u$ setwise.
	Suppose next we are given $\cC_i \in \opwalk{\cU}{u}$ with root $y_i \in P(i)$, for some $i< \len(P)-1$.
	By \Cref{fact:n-local-transitivity}, there is $g_{i} \in G_{P(i)\vee P(i+1)}$ which maps $y_i$ to some point $y_{i+1} \in P(i+1)$.
	Again, since $g_{i}$ fixes every element of $u$ setwise, $\cC_{i+1} \coloneqq g_{i} \cdot \cC_i$ belongs to $\opwalk{\cU}{u}$.

	To ease the notation, let $\cD_0 \coloneqq \cC_{\len(P)-1}$ and $z_0 \coloneqq x_{\len(P)-1}$.
	Then $z_0$ is the root of $\cD_0$ and it belongs to $v(0)$.
	We now proceed by induction.

	Suppose that $\cD_j \in \opwalk{\cU}{u}$ with root $z_0$ is given,
	and $\cD_j \in \opwalk{\cV}{\walk{v(0), \dots, v(j)}}$
	for some $j< \len(v)-1$.
	Let $D_j \in \cD_j$ be such that $D_j \in O(v(0), v(1), \dots, v(j))$, and let $t_j \in v(j) \cap D_j$. In case $j=0$, make sure that $t_0 \not = z_0$.
	By \Cref{fact:n-local-transitivity}, there is $h_{j} \in G_{v(j)\vee v(j+1)}$ which maps $t_j$ to some point $t_{j+1} \in v(j+1)$. Again, if $j=0$, pick $h_j$ so that it also
	satisfies $h_0(z_0) = z_0$.
	Call $\cD_{j+1} \coloneqq h_j \cdot \cD_j$.
	Notice that the root of $\cD_{j+1}$ is $h_j(z_0) = z_0$, since for $j > 0$ the set $v(0)$ is pointwise fixed by $h_j$.
	Moreover, $h_j[D_j] \sub \runion_{k \le j+1} v(k)$, as the set on the right is setwise fixed by $h_j$.
	\Cref{lem:first-last} then implies
	\begin{equation}
		\label{eq:D_in_initial_segment_v}
		\cD_{j+1} \in \opwalk{\cV}{\walk{v(0), \dots, v(j+1)}}.
	\end{equation}

	It is left to prove that $\cD_{j+1} \in \opwalk{\cU}{u}$.
	If $\rmap{\cV}{\cU} \cdot \walk{v(0), \dots, v(j+1)} \refex u$, then
	$\opwalk{\cV}{\walk{v(0), \dots, v(j+1)}} \sub \opwalk{\cU}{u}$, in which case we conclude by \eqref{eq:D_in_initial_segment_v}.
	Otherwise, it holds that $v(j) \vee v(j+1) \sub \runion_{k < \len(u) -1} u(k)$.
	By inductive assumption $\cD_j \in \opwalk{\cU}{u}$, so there is $A \in \cD_j$ with $A \in O(u(0), \dots, u(-1))$.
	Therefore $A \cap u(-1) \ne \emptyset$, and since $h_j$ fixes $u(-1)$, we can infer $h_j[A] \cap u(-1) \ne \emptyset$.
	Moreover, $h_j[A] \sub \runion_{k < \len(u)} u(k)$ since the set on the right is setwise fixed by $h_j$.
	We conclude by \Cref{lem:first-last}.

	Since $g_i, h_j \in G_{\cV,1} \sub G_0$ and $\cC_i, \cD_j \in \opwalk{\cU}{u} \sub \bU$ for each $i<\len(P), j<\len(v)$, it follows that $\cD_{\len(v)-1}$ is in the $(\bU, G_0)$-local orbit of $\cC$.
	But $\cD_{\len(v)-1} \in \opwalk{\cV}{v} \sub \bV$ by \eqref{eq:D_in_initial_segment_v}, so we are done.
\end{proof}

We finally obtain:
\begin{corollary}
	\label{thm:dynamical-consequences}
	Let $X$ be a strongly locally homogeneous Peano continuum which is not the circle, the sphere, or the real projective plane.
	Then $\Homeo(X)$ does not have the generic point property.
	In particular, its universal minimal flow is not metrizable and has no comeager orbit.

	The same is true for any closed $G \le \Homeo(X)$ which acts locally transitively on $X$.
\end{corollary}
\begin{proof}
	The action $G \acts \chains(X)$ is minimal, by \Cref{thm:local-transitive}.
	But $\Homeo(X) \acts \chains(X)$ has no comeager orbit by \Cref{cor:homo-chains}, so neither does $G \acts \chains(X)$.
\end{proof}

\begin{corollary} \label{cor:no_ccs}
	Let $X$ be a strongly locally homogeneous Peano continuum which is not the circle, the sphere, or the real projective plane.
	Then chains on $X$ are not classifiable by countable structures.
\end{corollary}
\begin{proof}
	This follows by combining \Cref{fact:Hjorth}, \Cref{thm:local-transitive}, and \Cref{cor:homo-chains}.
\end{proof}

\begin{question}
	What is the exact complexity of the orbit equivalence relation of $\Homeo(X) \acts \chains(X)$, for $X$ the torus or the Menger curve?
\end{question}

\begin{remark}
	\Cref{thm:dynamical-consequences} was proved in the case of closed manifolds of dimension at least 3, and the Hilbert cube, in \cite{GTZ}.
	Moreover, one can more easily show that $\Homeo(Y)$ and the subgroup $\Homeo_+(Y)$ of orientation preserving homeomorphism, for an orientable closed surface of positive genus $Y$, do not have the generic point property by remarking that such property is preserved under quotients, and that the mapping class group $\Homeo_+(Y)/\Homeo_0(Y)$ and extended mapping class group $\Homeo(Y)/\Homeo_0(Y)$ of such surfaces are countably infinite and discrete, and thus do not have the generic point property.
	This line of reasoning does not yield any information on the dynamics of $\Homeo_0(Y)$, unlike our results.
\end{remark}

\appendix

\section{Chains on \texorpdfstring{$1$}{1}-dimensional manifolds}
\label{sec:appendix-cicle-arc}

In this appendix we prove that, in contrast with the results established in Corollaries \ref{cor:homo-chains-later} and \ref{cor:no_ccs}, chains on $1$-dimensional manifolds are classifiable by countable structures and have a generic element.
We prove these results for $[0,1]$ and then reduce the case for $S^1$ to them.
The structure of the proof is inspired by that of \cite{MR2156712}*{Section 6}.

\subsection{Chains on \texorpdfstring{$[0,1]$}{[0,1]}}

We begin by giving a more concrete description of $\chains([0,1])$ which we will then use for the whole section.
Recall from \Cref{ss.chains}, that we denote by $C([0,1])$ the space of subcontinua of $[0,1]$.
The elements of $C([0,1])$ are closed intervals $[x,y]$ with $0\leq x\leq y\leq 1$ and the natural map $\varphi\colon C([0,1])\to\mathbb{R}^2$ given by $\varphi([x,y])=(x,y)$ is a homeomorphism onto its image, which is the closed triangle $T$ with vertices $(0,0)$, $(0,1)$ and $(1,1)$ \cite{illanes1999hyperspaces}*{Example 5.1}.
Elements of $C([0,1])$ are partially ordered by inclusion.
For pairs of points in $T$, this translates to the partial order $(x,y) \refex (x',y')$ iff $x\ge x'$ and $y \le y'$.

Recall that every chain in $\chains([0,1])$ is the image of an injective continuous order-preserving function $c\colon([0,1], \le) \to (C([0,1]), \sub)$ such that $c(0)$ is a singleton and $c(1)=[0,1]$.
Under the identification of $C([0,1])$ with $T$ described above, every chain $\cC\in\chains([0,1])$ gives rise to an arc $\varphi \cdot \cC \sub T$ containing the vertex $(0,1)$ and a point $(x_0,x_0)$, for some $x_0\in[0,1]$, on the diagonal, and such that for all $(x,y), (x', y') \in \varphi \cdot \cC$, either $(x, y) \refex (x', y')$ or $(x', y') \refex (x, y)$.
In the following we will identify $\chains([0,1])$ with the subspace of $C(T)$ of such arcs, dropping any mention of $\varphi$.

$\Homeo_+([0,1])$ acts on $T$ diagonally, that is, $h \cdot(x, y) = (h(x), h(y))$.
The induced action on $\chains([0,1]) \sub C(T)$ coincides with the natural action $\Homeo_+([0,1]) \acts \chains([0,1])$, discussed in \Cref{S.pc}.

\begin{definition}
	Given a chain $\cC\subseteq T$, a closed subset $\cI \subseteq\cC$ is called a \emph{subchain (of $\cC$)} if it is connected (in $T$).
	A non-trivial subchain $\cI$ is
	\begin{enumerate}
		\item \emph{horizontal} if for all $(x,y),(x',y')\in \cI$, we have $y=y'$,
		\item \emph{vertical} if for all $(x,y),(x',y')\in \cI$, we have $x=x'$,
		\item \emph{strictly monotone} in $\cC$ if for all $(x,y),(x',y')\in \cI$, we have $x = x'$ iff $y = y'$.
	\end{enumerate}
	A subchain $\cI \subseteq\cC$ is \emph{maximal horizontal} if it is horizontal and it is not properly contained in any horizontal subchain of $\cC$.
	Analogously, we define \emph{maximal vertical subchains} and \emph{maximal strictly monotone subchains}.
	By Zorn's Lemma, each horizontal/vertical/strictly monotone subchain, is contained in a maximal subchain of the same type.
	A subchain is \emph{maximal} if it is either maximal horizontal, maximal vertical, or maximal strictly monotone.

	Each subchain $\cI \subseteq\cC$ has a maximum and a minimum element with respect to $\refex$.
	If $(x_1, y_1), (x_2, y_2)$ are the maximum and minimum of $\cI$, we call the intervals $[x_1, x_2]$ and $[y_2, y_1]$ in $[0,1]$ the \emph{domain} and \emph{codomain} of $\cI$, respectively, and denote them by $\dom(\cI)$ and $\codom(\cI)$, see \Cref{fig:domain}.
	Notice that if $\cI$ is maximal, then at least one among the domain and codomain is non-trivial.
\end{definition}

\begin{figure}
	\centering{
		\includeinkscape[scale=1]{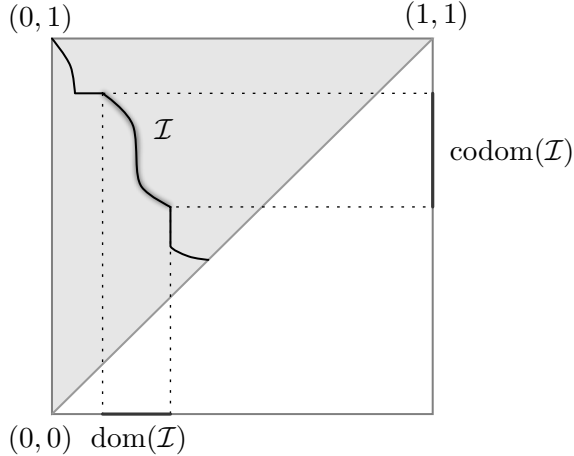}
	}
	\caption{Domain and codomain of a subchain $\cI$.} \label{fig:domain}
\end{figure}

Note that if $\cC\subseteq T$ is a chain,
and $\cI, \cI'$ are distinct maximal subchains, then their domains meet in at most one point, and the same is true for their codomains.
Since the domains of maximal horizontal and maximal strictly monotone subchains are non-trivial, it follows that there are only countably many such subchains.
By reasoning about codomains, we also see that there are countably many maximal vertical subchains, so that $\cC$ contains countably many maximal subchains.

We claim that if $(x_0, x_0)$ is the root of $\cC$, then the union of the domains of maximal subchains of $\cC$ is dense in $[0, x_0]$, while the union of their codomains is dense in $[x_0,1]$.
The projection of $\cC$ on the $x$-axis is $[0, x_0]$ and on the $y$-axis it is $[x_0, 1]$.
Suppose $x_0 \ne 0,1$, as otherwise the statement is trivial.
Let us show the result for the domains, the result for the codomains follows analogously.
Fix $0 \le x_1 \le x_2 \le x_0$, and let $(x_1, y_1), (x_2, y_2)$ be points in $\cC$.
If the subchain $\cI$ with maximum $(x_1, y_1)$ and minimum $(x_2, y_2)$ is not strictly monotone, there is a subchain included in it which is either vertical or horizontal.
In all three cases, $\cI$ intersects a maximal subchain whose domain contains a point between $x_1$ and $x_2$, and we are done.

Before proceeding with the proof of the classifiability result we are after, we need to check that a few sets that will arise in its proof are Borel, which we do in the following two lemmas.

\begin{lemma}\label{lemma:listing endpoints is borel}
	The following subsets of $\chains([0,1])\times C([0,1])$ (identified with a subspace of $C(T)\times T$ as above) are Borel:
	\begin{align*}
		\begin{split}
			\cE^H  \coloneqq \set{(\cC,(x,y)) : & (x,y)\text{ is the maximum of a maximal} \\  &\text{horizontal subchain in }\cC},
		\end{split} \\
		\begin{split}
			\cE^V \coloneqq \set{(\cC,(x,y)) : & (x,y)\text{ is the maximum of a maximal} \\  &\text{vertical subchain in }\cC},
		\end{split}        \\
		\begin{split}
			\cE^S  \coloneqq \set{(\cC,(x,y)) : & (x,y)\text{ is the maximum of a maximal} \\ &\text{strictly monotone subchain in }\cC}.
		\end{split}
	\end{align*}
\end{lemma}
\begin{proof}
	$(x,y)\in\cC$ is a closed condition.
	Then $(x,y)$ is the maximum of a horizontal subchain in $\cC$ if and only if $(x+1/n,y)\in\cC$ for some $n\in\N$, which is a $\mathbf{\Sigma}^0_2$ condition.
	Finally, it is the maximum of a \emph{maximal} horizontal subchain in $\cC$ if and only if moreover $(x-1/n,y)\not\in\cC$ for all $n\in\N$, which is a $\mathbf{\Pi}^0_2$ condition, showing that $\cE^H$ is Borel.
	The argument for $\cE^V$ is exactly the same, so it remains to show that $\cE^S$ is Borel.
	Indeed, $(x,y)$ is the maximum of a maximal strictly monotone subchain in $\cC$ if and only if the following conditions are all satisfied:
	\begin{enumerate}[label=(\alph*)]
		\item $(x,y)\in\cC$, which is a closed condition,
		\item there exists $n\in\N$ such that for all $m\in\N$ there are no $(x_1,y_1)$ and $(x_2,y_2)\in\cC$ such that for $i\in\{1,2\}$, $x\leq x_i\leq x+1/n$, $y\geq y_i\geq y-1/n$ and either $x_1=x_2$ and $y_1=y_2+1/m$ or $y_1=y_2$ and $x_1=x_2+1/m$, which is a $\mathbf{\Sigma}^0_3$ condition,
		\item for all $n\in\N$ there exists $m\in\N$ and there exist $(x_1,y_1)$ and $(x_2,y_2)\in\cC$ such that for $i\in\{1,2\}$, $x\geq x_i\geq x-1/n$, $y\leq y_i\leq y+1/n$ and either $x_1=x_2$ and $y_1=y_2+1/m$ or $y_1=y_2$ and $x_1=x_2+1/m$, which is a $\mathbf{\Pi}^0_3$ condition,
	\end{enumerate}
	showing that $\cE^S$ is Borel.

\end{proof}

\begin{lemma}\label{lemma:chains with infinitely many subchains are Borel}
	The subspace $\chains^\infty([0,1])\subseteq\chains([0,1])$ consisting of all chains containing infinitely many maximal subchains is Borel.
	The subspace $\chains^0([0,1])$ of chains with no horizontal or vertical subchains is $\mathbf{\Pi}^0_2$.
\end{lemma}
\begin{proof}
	Note that a chain contains finitely many maximal subchains if and only if it contains finitely many horizontal subchains and finitely many vertical subchains,
	since between any two distinct maximal strictly monotone subchains there must be a maximal horizontal or vertical subchain.
	Therefore, to prove that $\chains^\infty([0,1])$ is Borel, it suffices to show that, for each $n\in\N$, the sets $\chains^{\geq n}_H([0,1])$ and $\chains^{\geq n}_V([0,1])$ consisting respectively of chains containing at least $n$ horizontal subchains and chains containing at least $n$ vertical subchains are Borel.
	Note that, for all $n,m\in\N$ the set $\chains^{\geq n,m}_H([0,1])$ consisting of all chains that contain at least $n$ horizontal subchains, which are each of length at least $1/m$ and are pairwise at distance at least $1/m$ is closed, and that every element of $\chains^{\geq n}_H([0,1])$ is contained in $\chains^{\geq n,m}_H([0,1])$ for some $n,m\in\N$.
	This shows that
	\begin{equation*}
		\chains^{\geq n}_H([0,1])=\bigcup_{m\in\N}\chains^{\geq n,m}_H([0,1])
	\end{equation*}
	is $\mathbf{\Sigma}^0_2$.
	By an analogous argument $\chains^{\geq n}_V([0,1])$ is also $\mathbf{\Sigma}^0_2$, showing that $\chains^\infty([0,1])$ is $\mathbf{\Pi}^0_3$.

	Moreover, $\chains^0([0,1])$ is the complement of the union of $\chains^{\geq 1}_H([0,1])$ and $\chains^{\geq 1}_H([0,1])$, and thus is $\mathbf{\Pi}^0_2$.
\end{proof}

We are now able to show that there is a generic chain on the interval.

\begin{proposition}
	\label{prop:arc-has-generic-chains}
	The chain $\cC_0 \coloneqq \set{[1/2-\alpha, 1/2+\alpha] : 0\le \alpha \le 1/2}$ is a generic chain on $[0,1]$.
\end{proposition}
\begin{proof}
	We prove that $\Homeo([0,1]) \cdot \cC_0 = \chains^0([0,1])$ and that $\chains^0([0,1])$ is dense in $\chains([0,1])$, and conclude by \Cref{lemma:chains with infinitely many subchains are Borel}.

	The chain $\cC_0$ is the graph of the function $f_0(x) \coloneqq 1-x$ from $[0,1/2]\to[1/2,1]$.
	Clearly $\chains^0([0,1])$ is invariant under the action of $\Homeo([0,1])$, so fix $\cC \in \chains^0([0,1])$ with the objective of finding $h \in \Homeo([0,1])$ such that $h \cdot \cC = \cC_0$.
	Up to precomposing with a $g \in \Homeo_+([0,1])$ which maps the minimum of $\cC$ to $(1/2, 1/2)$, we can suppose that also $\cC$ has minimum $(1/2, 1/2)$.
	Then $\cC$ is the graph of a continuous strictly decreasing function $f \colon [0,1/2] \to [1/2, 1]$ such that $f(0)=1, f(1/2)=1/2$.
	Let $h \in \Homeo([0,1])$ be the identity on $[0,1/2]$ and $h(y) \coloneqq 1 - f^{-1}(y)$ for $1/2 \le y \le 1$.
	Then
	\begin{equation*}
		hf(x) = 1- x = f_0(x),
	\end{equation*}
	that is, $h \cdot \cC = \cC_0$.

	To check for density, notice that given any finite sequence $(x_0,y_0) \refex \cdots \refex (x_{n-1}, y_{n-1})$ of points in the interior of $T$, and any $\varepsilon >0$, there is a continuous strictly decreasing function $f \colon [0,1] \to [0,1]$ such that $f(0) =1$ and $\Card{f(x_i) - y_i} < \varepsilon$, for all $i<n$, and that the intersection of the graph of $f$ with $T$ is a chain in $\chains^0([0,1])$.
	We conclude by noticing that the above condition determines a basic open set in $\chains([0,1])$.
\end{proof}

\begin{theorem}\label{theorem:classifiability}
	Chains on the interval $[0,1]$ are classifiable by countable structures.
\end{theorem}
\begin{proof}
	Since $\Homeo_+([0,1])$ is a clopen subgroup of $\Homeo([0,1])$, it is enough to prove that $\Homeo_+([0,1])\curvearrowright\chains([0,1])$ is classifiable by countable structures, and apply \cite{hjorth2000classification}*{Lemma 4.7}.

	We concentrate on $\chains^\infty([0,1])$; the case for chains with finitely many maximal subchains can be shown to be classifiable by countable structures with an analogous, but easier, argument, which we leave to the reader.
	Since $\chains^\infty([0,1])$ is Borel by Lemma \ref{lemma:chains with infinitely many subchains are Borel}, there is no harm in treating the two cases separately.
	Let $\cE^\infty \coloneqq(\chains^\infty([0,1])\times T)\cap(\cE^H\cup\cE^V\cup\cE^S)$, which is Borel by Lemma \ref{lemma:listing endpoints is borel}.
	Since, for every $\cC\in\chains^\infty([0,1])$, the section $\set{(x,y)\in T : (\cC,(x,y))\in\cE^\infty }$ is countably infinite, it follows from the Lusin-Novikov uniformization theorem \cite{KecBook}*{Exercise 18.12} that there exists a sequence $e_n\colon\chains^\infty([0,1])\to T$ of Borel functions such that for every $\cC\in\chains^\infty([0,1])$, $\{e_n(\cC) : n\in\N\}$ is an injective enumeration of $\{(x,y)\in T : (\cC,(x,y))\in\cE^\infty\}$.
	In addition to the quantity and type of maximal subchains occurring in $\cC$ we only need to code their relative position in a countable structure in order to determine $\cC$ up to homeomorphism.
	Let $\cL \coloneqq \{P_H,P_V,P_S,<\}$ be the language containing the three unary predicates $P_H,P_V$ and $P_S$ together with a binary relation $<$.
	Given a chain $\cC$, consider the structure $\cM(\cC)$ on $\N$ defined as follows:
	\begin{enumerate}
		\item $P_H^{\cM(\cC)}(n)$ holds if and only if $(\cC,e_n(\cC))\in\cE^H$,
		\item $P_V^{\cM(\cC)}(n)$ holds if and only if $(\cC,e_n(\cC))\in\cE^V$,
		\item $P_S^{\cM(\cC)}(n)$ holds if and only if $(\cC,e_n(\cC))\in\cE^S$,
		\item $n <^{\cM(\cC)} m$ holds if and only if $e_n(\cC) \refex e_m(\cC)$.
	\end{enumerate}
	Since the $\cE^H$, $\cE^V$, $\cE^S$, all of the $e_n$'s, and the partial order $\refex$ (as a subset of $T^2$) are Borel, the assignment $\cC\mapsto\cM(\cC)$ is Borel.

	It remains to verify that two chains $\cC,\cC'\in\chains^\infty([0,1])$ are in the same $\Homeo_+([0,1])$-orbit if and only if $\cM(\cC)\cong\cM(\cC')$.

	Suppose first that $\cC'=h\cdot\cC$ for some $h\in\Homeo_+([0,1])$.
	Then the function $\psi\colon\cM(\cC)\to\cM(\cC')$, defined by letting $\psi(n)$ be the unique $m$ such that $e_m(\cC')=h(e_n(\cC))$, is the desired isomorphism.

	Conversely, suppose that $\cC$ and $\cC'$ are such that the associated structures are isomorphic and fix an isomorphism $\psi\colon\cM(\cC)\to\cM(\cC')$. Let $x_0$ and $x_0'$ be the roots of $\cC$ and $\cC'$ respectively. We enumerate the maximal subchains of $\cC$ as $\{\cI_n\}_{n\in\N}$, where $\cI_n$ is such that its maximum is $e_n(\cC)$. Analogously we enumerate the maximal subchains of $\cC'$ as $\{\cI'_n\}_{n\in\N}$ so that $e_n(\cC')$ is the maximum of $\cI'_n$.
	For all $n$ we define a pair of order-preserving homeomorphisms $h_n^x\colon \dom(\cI_n) \to \dom(\cI'_{\psi(n)})$ and $h_n^y\colon\codom(\cI_n) \to \codom(\cI'_{\psi(n)})$.
	We will then verify that $h^x=\bigcup_n h_n^x$ is a densely defined, order-preserving partial homeomorphism $[0,x_0]\to[0,x_0']$ and that $h^y$ defined analogously has the same properties as a function $[x_0,1]\to[x_0',1]$ so that $h^x\cup h^y$ can be extended to the desired homeomorphism of $[0,1]$.

	Given $n\in\cM(\cC)$, suppose first that $P^{\cM(\cC)}_H(n)$ holds, so that $\cI_n$ is a maximum horizontal subchain of $\cC$, with maximum $e_n(\cC)=(x_1,y_1)$ and minimum $(x_2,y_1)$.
	Similarly, $\cI_{\psi(n)}'$ is a maximal horizontal subchain of $\cC$ with maximum $e_{\psi(n)}(\cC')=(x_1',y_1')$ and minimum $(x_2',y_1')$.
	Notice that if $y_1 \in [x_1,x_2]$, then $y_1=x_2$, since $(x_2,y_1) \in T$.
	But, in this case, $n$ is the minimum element of $<^{\cM(\cC)}$, and thus also $\psi(n)$ is the minimum element of $<^{\cM(\cC')}$, so $y_1'=x_2'$.
	In any case, define $h_n^x\colon \dom(\cI_n) \to \dom(\cI'_{\psi(n)})$ to be any order preserving homeomorphism and define $h_n^y(y_1) \coloneqq y_1'$.
	Note that
	\begin{equation}
		\label{eq:h-horizontal}
		\cI'_{\psi(n)} = \set{(h_n^x(x), h_n^y(y)) : (x,y) \in \cI_n}.
	\end{equation}

	\smallskip

	The case in which $P_V^{\cM(\cC)}(n)$ holds is analogous, so it remains to deal with the case in which $P_S^{\cM(\cC)}(n)$ holds.
	By definition this means that $\cI_n$ is a maximal strictly monotone subchain in $\cC$, with maximum $e_n(\cC)=(x_1,y_1)$ and minimum $(x_2,y_2)$.
	Similarly, $\cI'_{\psi(n)}$ is a maximal strictly monotone subchain of $\cC'$ with maximum $e_{\psi(n)}(\cC')=(x'_1,y'_1)$ and minimum $(x'_2,y'_2)$.
	Now $\cI_n$ is the graph of a continuous, strictly decreasing function $f_1\colon\dom(\cI_n) \to \codom(\cI_n)$, and similarly we obtain from $\cI'_{\psi(n)}$ a continuous strictly decreasing function $f_2\colon\dom(\cI'_{\psi(n)}) \to \codom(\cI'_{\psi(n)})$.
	If $\dom(\cI_n) \cap \codom(\cI_n) \ne \emptyset$, then they meet only in $x_2 = y_2$, since $(x_2, y_2) \in T$.
	As above, in this case, $n$ is the minimum element of $<^{\cM(\cC)}$, and thus also $\psi(n)$ is the minimum element of $<^{\cM(\cC')}$, so $x_2'=y_2'$.

	In any case, let $h_n^x\colon \dom(\cI_n) \to \dom(\cI'_{\psi(n)})$ be any orientation preserving homeomorphism, and $h_n^y\colon \codom(\cI_n) \to \codom(\cI'_{\psi(n)})$ be $h_n^y \coloneqq f_2h_n^xf_1^{-1}$.
	Notice that also in this case:
	\begin{align}
		\label{eq:h-monotone}
		\cI'_{\psi(n)} & = \set{(x', f_2(x')) : x' \in \dom(\cI'_{\psi(n)})} \\ \nonumber &= \set{(h_n^x(x), h_n^y(f_1(x))) : x \in \dom(\cI_n)} \\  \nonumber &= \set{(h_n^x(x), h_n^y(y)) : (x,y) \in \cI_n}.
	\end{align}

	\smallskip

	Finally, let $h^x=\bigcup_{n\in\N}h_n^x$ and $h^y=\bigcup_{n\in\N}h_n^y$. Note that if $\cI_n$ and $\cI_m$ are two maximal subchains of $\cC$ whose domains $[x_1^n,x_2^n]$ and $[x_1^m,x_2^m]$ have nonempty intersection, then they meet in exactly one point, and we must have either $x_2^n=x_1^m$ or $x_1^n=x_2^m$. Suppose without loss of generality that the former holds, which implies that $m$ is the $<^{\cM(\cC)}$-immediate successor of $n$. We must have that $h_n^x(x_2^n)$ is the minimum of the domain of $\cI'_{\psi(n)}$, while $h_m^x(x_1^m)$ must be the maximum of the domain of $\cI'_{\psi(m)}$, but since $\psi$ is an isomorphism, $\psi(m)$ is the $<^{\cM(\cC')}$-immediate successor of $\psi(n)$, showing that
	\begin{equation*}
		h_n^x(x_2^n)=h_m^x(x_1^m)=h_m^x(x_2^n),
	\end{equation*}
	that is, $h^x$ is well-defined. Since $h^x$ is the union of a family of injective functions that cohere on their common domains, $h^x$ is injective as well. With a similar argument $h^y$ is shown to be well-defined and injective too.

	\begin{claim}
		$h^x$ is a densely defined, order-preserving, partial homeomorphism $[0,x_0]\to[0,x_0']$ with dense image. Similarly, $h^y$ is a densely defined, order-preserving, partial homeomorphism $[x_0,1]\to[x_0',1]$ with dense image.
	\end{claim}
	\begin{proof}[Proof of the Claim.] We only verify the claim for $h^x$, since the case of $h^y$ is analogous. The domain of $h^x$ is dense in $[0,x_0]$ because $h^x$ is defined on the union of the domains of all maximal subchains of $\cC$, while its image is dense in $[0,x_0']$, because $h^x$ has as image the union of the domains of the maximal subchains of $\cC'$.

		In order to show that $h^x$ is order-preserving fix $z<z'$ in its domain and let $n,m\in\N$ be such that $z$ belongs to the domain of $\cI_n$, while $z'$ belongs to the domain of $\cI_m$.
		If $n=m$, $h^x(z) = h^x_n(z) < h^x_n(z') =h^x(z')$, since $h^x_n$ is order preserving.
		Otherwise, by definition of $<^{\cM(\cC)}$, we have $m<^{\cM(\cC)}n$, so that $\psi(m)<^{\cM(\cC')}\psi(n)$. This implies that $h^x(z)\leq h^x(z')$, since the former is in the domain of $\cI'_{\psi(n)}$ while the latter is in the domain of $\cI'_{\psi(m)}$.
		Since $h^x$ is injective, we have $h^x(z)<h^x(z')$, as desired.

		Since each $h_n^x$ is a homeomorphism from its domain onto its image,
		$h^x$ is a homeomorphism, e.g., by the pasting lemma \cite{MR3728284}*{Exercise 2.9(c)}.
	\end{proof}

	By the previous claim $h^x\cup h^y$ extends to an order-preserving homeomorphism $h\colon [0,1]\to[0,1]$, and it only remains to verify that $h\cdot\cC=\cC'$.
	But $h \cdot \cI_n = \cI'_{\psi(n)}$ by \eqref{eq:h-horizontal} and \eqref{eq:h-monotone}, and $\bigcup_n\cI_n, \bigcup_n\cI'_{\psi(n)}$ are dense in $\cC, \cC'$, respectively, so we conclude by continuity of $h$.
\end{proof}

\subsection{Chains on \texorpdfstring{$ S^1 $}{S1}}

In this section we reduce to the results in the previous section to prove analogous statements for the circle.

\begin{theorem} \label{thm:genericS1}
	The orbit equivalence relation of $\Homeo(S^1)\curvearrowright\chains(S^1)$ is continuously reducible to that of $\Homeo([0,1])\curvearrowright\chains([0,1])$.
	Therefore, chains on the circle $S^1$ are classifiable by countable structures.
	Moreover, there is a generic chain in $S^1$.
\end{theorem}
\begin{proof}[Proof sketch.]
	Let $S^1 = [0,1] / \mod 1$, and denote by $\pi$ be quotient map.
	Since the closure of the complement of a subcontinuum of $S^1$ is itself a subcontinuum of $S^1$, to every chain $\cC\in\chains(S^1)$ we can associate not only its root, $\bigcap\cC$ but also its \emph{endpoint}:
	\begin{equation*}
		e(\cC) \coloneqq \bigcap_{C\in\cC}\Cl(S^1\setminus C),
	\end{equation*}
	which is a point, since the diameters of the complements go to zero.

	For each $z \in S^1$, let $\chains_z(S^1)$ be the closed subset of chains whose endpoint is $z$,
	and fix the clockwise $\rho_z$ rotation taking $z$ to $\pi(0)$.
	Let $\cR \colon \chains(S^1) \to \chains_{\pi(0)}(S^1)$ be the continuous map:
	\begin{equation*}
		\cC \mapsto \rho_{e(\cC)} \cdot \cC.
	\end{equation*}
	Then $h \cdot \cC = \cC'$ if and only if $\rho_{e(\cC')} h \rho_{e(\cC)}^{-1} \cdot \cR(\cC) = \cR(\cC')$, showing that $\cR$ is a continuous reduction from the orbit equivalence relation of $\Homeo(S^1)\curvearrowright\chains(S^1)$ to that of $\Homeo_{\pi(0)}(S^1)\curvearrowright\chains_{\pi(0)}(S^1)$, where $\Homeo_{\pi(0)}(S^1)$ is the stabilizer of $\pi(0)$.

	But $\pi$ induces an isomorphism $\pi^*$ of flows between $\Homeo([0,1]) \acts \chains([0,1])$ and $\Homeo_{\pi(0)}(S^1) \acts \chains_{\pi(0)}(S^1)$, since $\pi$ is injective on proper non-trivial subcontinua of $[0,1]$.
	We conclude by \Cref{theorem:classifiability}.

	For the moreover, notice that $\cR^{-1}(\pi^*[\chains^0([0,1])])$ is comeager and an orbit.
\end{proof} \section{The Sierpi\'{n}ski carpet}
\label{sec:appendix-sierpinski}

This appendix is devoted to proving that the action $\Homeo(S) \acts \chains(S)$ is minimal, for the Sierpi\'{n}ski carpet $S$, and thus concluding that $\Homeo(S)$ does not have the generic point property and its universal minimal flow is not metrizable.
Since, as we will see, the Sierpi\'{n}ski carpet is not homogeneous, and thus not strongly locally homogeneous, the minimality of $\Homeo(S) \acts \chains(S)$ does not follow from \Cref{thm:local-transitive}.
We do not settle whether this action is generically turbulent.

An \emph{$S$-curve} is a one-dimensional planar Peano continuum such that its complement in the plane consists of countably many connected open sets whose boundaries are mutually disjoint simple closed curves. \cite{whyburn}*{Theorem 4} states there exists a unique one-dimensional planar Peano continuum $S$ with no locally separating points, and that all $S$-curves are homeomorphic to such $S$. The space $S$ is called the \emph{Sierpi\'{n}ski carpet}.

It is immediate that $S$ satisfies condition \eqref{itm:nls-curve} of \Cref{thm:the-theorem-introduction}, and it is not hard to see that it also satisfies condition \eqref{itm:circular}.
\begin{corollary}
	There is no generic chain on the Sierpi\'{n}ski carpet.
\end{corollary}

The Sierpi\'{n}ski carpet has two types of points: \emph{rational} and \emph{irrational}.
The collection of rational points is given by the union of all the boundaries of the maximal connected components of the complements of $S$ in the plane or, equivalently and more abstractly,
it is the union of all non-trivial non-locally-separating proper subcontinua of $S$ (see \cite{MR0200910}*{\S8}).
There are countably many maximal such subcontinua, called \emph{rational boundary curves}, they are all simple closed curves, and their union is dense and co-dense in the space.
Since it is the union of countably many non-locally-separating subcontinua, the set of rational points is a non-locally-separating $F_\sigma$ set, by \cite{MOT}*{Theorem 2.2}.
The set of irrational points is the complement.
A subcontinuum of $S$ is said to be rational/irrational if all of its points are rational/irrational.

The action $\Homeo(S) \acts S$ has two orbits, the rational and irrational points \cite{krasinkiewicz}, thus it is minimal, but not locally transitive.

We need a couple of preliminary statements before showing that the action $\Homeo(S) \acts \chains(S)$ is minimal.
The following proposition mirrors Facts \ref{fact:points-on-boundaries} and \ref{fact:existence-of-arcs}.

\begin{proposition}
	\label{prop:facts-about-sierpinski}
	Let $\cW$ be a brick partition of $S$ and $W, W' \in \cW$.
	\begin{enumerate}
		\item \label{itm:irrational-points} If $W \E W'$ then there is an irrational point $x \in \Cl(W)\cap \Cl(W') \cap (W \vee W')$.
		\item \label{itm:irrational-arcs} For any $x, y \in \Cl(W)$, there is an arc $\gamma \sub W \cup \set{x,y}$ connecting $x, y$, which is irrational except possibly at $x,y$.
		\item \label{itm:existence-of-arcs-sierpinski} For any $A, B \sub \Cl(W)$, disjoint and containing exactly two points each, there are two disjoint arcs $\gamma_0, \gamma_1 \sub W \cup A \cup B$, connecting $A$ to $B$ which are irrational except possibly at their endpoints.
	\end{enumerate}
\end{proposition}
\begin{proof}
	(\ref{itm:irrational-points}) and (\ref{itm:irrational-arcs}) follow directly from \Cref{fact:existence-of-arcs}.\eqref{itm:existence-of-arcs-nls-arc} and \Cref{fact:points-on-boundaries}.\eqref{itm:nls_set}, respectively.

	(\ref{itm:existence-of-arcs-sierpinski}).
	By \Cref{fact:existence-of-arcs}.\eqref{itm:existence-of-arcs-nls}, there are two disjoint arcs $\lambda_0, \lambda_1$ connecting $A$ to $B$ in $W \cup A \cup B$.
	Let $\cV \refines \cW$ be fine enough that $\Star(\lambda_0, \cV)$ and $\Star(\lambda_1, \cV)$ are disjoint, and call $P_i \coloneqq \runion \Star(\lambda_1, \cV)$, $i=0,1$.
	By item \eqref{itm:irrational-arcs} of the current proposition, there are $\gamma_i$ in $P_i \cup A \cup B$, with the same endpoints as $\lambda_i$, and which are irrational except possibly at such endpoints.
	Since the $P_i$'s are disjoint, so are the $\gamma_i$'s.
\end{proof}

We recall that $C(S)$ is the space of all closed connected subsets of $S$, with the Vietoris topology (see \Cref{ss.chains}).
\begin{proposition}
	\label{prop:sierpinski-irrational-closed-curves}
	Irrational simple closed curves are dense in $C(S)$.
\end{proposition}
\begin{proof}
	It is easy to see, arguing as in \Cref{sec:Walks on Brick Partitions and their Refinements}, that $C(\cW) \coloneqq \set{O(\cW_0) : \cW_0 \sub \cW \text{ nonempty and $E$-connected}}$ is a quasi partition of $C(X)$ of $d_H$ mesh controlled by the $d$-mesh of $\cW$.
	We can thus suppose we are given a nonempty open set of $C(S)$ of the form $O(\cW_0)$, for an $E$-connected subset $\cW_0 \sub \cW$ of some brick partition $\cW$ on $S$.
	Let $w$ be a reduced walk on $\cW$, which is confined to and visits all vertices of $\cW_0$.

	By \nameref{thm:walks-to-paths} (\Cref{thm:walks-to-paths}) and \nameref{prop:path-doubling} (\Cref{prop:path-doubling}), there are a brick partition $\cV \refines \cW$ and two paths $v_{0},v_{1}$ on $\cV$ which are disjoint except for their endpoints and such that $\rmap{\cV}{\cW} \cdot v_{i} \refex w$, for $i < 2$.
	Let $v$ be the circular path which follows $v_0$ forward and then $v_1$ backwards.
	For each $i < \ell$, let $x_i$ be an irrational point in $\Cl(v(i))\cap \Cl(v(i+1 \mod \ell)) \cap (v(i) \vee v(i+1 \mod \ell))$, which exists by \Cref{prop:facts-about-sierpinski}.\eqref{itm:irrational-points}.
	By point \eqref{itm:irrational-arcs} of the same proposition, there are irrational arcs $\gamma_i$ in $v(i) \cup \set{x_i, x_{i+1 \mod \ell}}$ connecting $x_i$ to $x_{i+1 \mod \ell}$.
	Let $\gamma$ be the union of such arcs.
	It is an irrational simple closed curve which meets each $v(i)$, for $i<\len(v)$, and is contained in $\runion_{i<\len(v)} v(i)$.
	Since $\rmap{\cV}{\cW} \cdot v \refex w$, it follows that $\gamma$ meets each $W \in \cW_0$ and is contained in $\runion \cW_0$, that is, $\gamma \in O(\cW_0)$.
\end{proof}

The proof of \Cref{thm:minimality-sierpinski} builds upon the following fundamental homogeneity result for the Sierpi\'{n}ski carpet from \cite{whyburn} (see also \cite{MR0200910}*{(7.3)}).
\begin{theorem}[{\cite{whyburn}}]
	\label{fact:extension-sierpinski}
	Let $\gamma$ be a rational boundary curve of $S$ and let $h_0$ be a homeomorphism of $\gamma$ onto itself.
	Then there is $h \in \Homeo(S)$ which extends $h_0$.
\end{theorem}

Let $\gamma$ be a simple closed curve in $S$, and let $W$ be a connected component of $S \setminus \gamma$.
As noted in \cite{whyburn}*{\S2}, if every rational maximal subcontinuum of $S$ that is contained in $W \cup \gamma$ is either disjoint from $\gamma$ or coincides with $\gamma$, then $W \cup \gamma$ is an $S$-curve, and hence it is homeomorphic to $S$.
This is the case for instance if $\gamma$ is irrational and $W$ is any one of the two components of $S \setminus \gamma$.
A fact that we will often use implicitly is that a point in $W$ is irrational in $W \cup \gamma$ if and only if it is irrational in $S$.

The last ingredient is the following technical lemma.

\begin{lemma}
	\label{lem:separate-stuff-sierpinski}
	Let $\gamma$ be a rational boundary curve of $S$ and let $U_0,U_1 \sub S$ be disjoint nonempty open connected subsets such that $U_0 \cap \gamma \ne \emptyset \ne U_1 \cap \gamma$.
	Then there are open sets $W_0 \sub W'_0$, with $\Cl(W'_0) \sub U_0 \setminus \gamma$, such that for any $x \in S \setminus \Cl(W'_0)$ there are three disjoint arcs $\theta, \zeta, \eta$ whose endpoints belong to $\gamma$ and which are irrational everywhere else, and which are such that:
	\begin{enumerate}
		\item $\theta, \zeta \sub U_0$, one of the connected components of $S \setminus \theta$ is contained in $U_0$, and the same is true of one of the connected components of $S \setminus \zeta$,
		\item $\eta \sub U_1$, and one of the connected components of $S \setminus \eta$ is contained in $U_1$,
		\item \label{item2:lemmaS} $\zeta$ separates $\theta$ from $x$ and $\eta$,
		\item \label{item3:lemmaS} $\theta$ separates $W_0$ from $\zeta$.
	\end{enumerate}
\end{lemma}
\begin{proof}
	Apply \Cref{prop:sierpinski-irrational-closed-curves} to find an irrational simple closed curve $\gamma'_0$ which is entirely contained in $U_0$.
	It separates $S$ into two nonempty components, one of which is contained in $U_0$.
	Call this latter component $W'_0$.
	In particular, $\Cl(W'_0) \sub U_0 \setminus \gamma$.

	Repeat this operation to find an irrational simple closed curve $\gamma_0$ contained in $W'_0$.
	It separates $S$ in two components, one of which, call it $W_0$, is contained in $W'_0$.

	Now, let $x \in S \setminus \Cl(W'_0)$ be given.
	Notice that $\Sigma \coloneqq S \setminus W'_0$ is an $S$-curve and that $\gamma, \gamma'_0$ are rational boundary curves in it.
	By \Cref{lemma:fine-gamma-compatible}, there are a brick partition $\cV'$ of $\Sigma$ and two disjoint spaced circular paths $c, c'_0$ on $\cV'$ which are $\gamma$-compatible and $\gamma'_0$-compatible, respectively.
	We can moreover suppose that $\cV'$ is fine enough so that:
	\begin{enumerate}[label=(\roman*)]
		\item there is $i < \len(c)$ such that $c(i) \in \cV'(U_1)$, and
		\item there is a path $P'$ from $c$ to $c'_0$ in $\cV'(U_0)$ which is disjoint from $\Star(x, \cV')$,
	\end{enumerate}
	which is possible by \nameref{prop:path-doubling} (\Cref{prop:path-doubling}).
	Since $c$ is $\gamma$-compatible, $\Card{\gamma \cap c(i)} \ge 2$, for each $i < \len(c)$, and the same is true for $c'_0$ and $\gamma'_0$.

	Using \Cref{prop:facts-about-sierpinski}.\eqref{itm:irrational-arcs}, we find an arc $\eta \sub c(i)$ whose endpoints belong to $\gamma$ and whose other points are irrational.
	In particular, $\eta \sub U_1$, and since $\gamma \cap c(i)$ is connected, one the connected components of $S \setminus (\gamma \cup \eta)$ is contained in $U_1$.

	By \Cref{prop:facts-about-sierpinski}.\eqref{itm:existence-of-arcs-sierpinski}, there are two disjoint arcs $\zeta_0, \zeta_1$ in $\runion_{i \in \len(P')} P'(i)$ which are irrational (in $\Sigma$) except at their endpoint and connect $\gamma'_0$ to $\gamma$.
	Now, $\gamma'_0$ is split into two arcs by the endpoints of $\zeta_0, \zeta_1$.
	One of these arcs, call it $\zeta_2$, is such that $\zeta \coloneqq \zeta_0 \cup \zeta_2 \cup \zeta_1$ separates $S$ into two components, one of which contains $W'_0$ and is contained in $W'_0 \cup \runion_{i \in \len(P')} P'(i) \sub U_0$.
	Call this component $V'_0$.
	Then $W_0 \sub V'_0$, while $x$ and $U_1$ are in $S \setminus V'_0$.

	By the same reasoning as above, there are two disjoint arcs $\theta_0, \theta_1$ in $V'_0 \cup \gamma$ which are irrational (in $S \setminus W_0$) except at their endpoint and connect $\gamma_0$ to $\gamma$.
	Now, $\gamma_0$ is split into two arcs by the endpoints of $\theta_0, \theta_1$.
	One of these arcs, call it $\theta_2$, is such that $\theta \coloneqq \theta_0 \cup \theta_2 \cup \theta_1$ separates $S$ into two components, one of which contains $W_0$ and is contained in $V'_0$.
	Call this component $V_0$.

	Then $\theta \sub V'_0$, while $x$ and $U_1$ are in $S \setminus V'_0$, so point \eqref{item2:lemmaS} is satisfied.
	On the other hand, $W_0 \sub V_0$ and $\zeta \sub S \setminus V_0$, so also point \eqref{item3:lemmaS} is taken care of.
\end{proof}

\begin{theorem}
	\label{thm:minimality-sierpinski}
	The action $\Homeo(S) \acts \chains(S)$ is minimal, for the Sierpi\'{n}ski carpet $S$.
\end{theorem}
\begin{proof}
	Since $\Homeo(S) \acts S$ is minimal, it suffices to verify the conditions of \Cref{prop:sufficient-minimality}.
	Let $\cC \in \chains(X)$ with root $x_0$ be given, together with nonempty open connected sets $U_0, U_1, V$ such that $U_0, U_1 \sub V$. In case $U_0 \cap U_1 \not = \emptyset$,
	by minimality of $\Homeo(S) \acts S$ we can assume that $x_0 \in U_0 \cap U_1$, so the conclusion follows. We henceforth consider the case
	$U_0 \cap U_1 = \emptyset$.

	\sloppy
	By \Cref{prop:sierpinski-irrational-closed-curves}, there is a simple closed irrational curve $\gamma \in O(U_0, U_1, V)$.
	Such curve $\gamma$ separates $S$ into two nonempty components, one of which is contained in $V$ and meets $U_0$ and $U_1$.
	Call this component $Z$, and its closure $\Sigma \coloneqq Z \cup \gamma$.
	Call the other component $Z'$, with closure $\Sigma' \coloneqq Z' \cup \gamma$. Up to shrinking $U_0$ and $U_1$, we can assume that $U_0 \cap \Sigma$ and $U_1 \cap \Sigma$ are connected.

	Let $W_0 \sub W'_0 \sub U_0 \setminus \gamma$ be given by \Cref{lem:separate-stuff-sierpinski} applied to $\gamma$ and $\Sigma$.
	By minimality of $\Homeo(S) \acts S$, we can assume that $x_0 \in W_0$, hence by maximality of $\cC$, there is $C \in \cC$ such that $C \sub Z$ and $C \nsubseteq \Cl(W'_0)$.
	Let $x_1$ be a point of $C \setminus \Cl(W'_0)$.

	By \Cref{lem:separate-stuff-sierpinski}, there are three disjoint arcs $\theta, \zeta, \eta$ whose endpoints belong to $\gamma$ and which are irrational everywhere else, and which are such that:

	\begin{enumerate}
		\item \label{item1:sierp} $\theta, \zeta \sub U_0$, one of the connected components of $S \setminus \theta$ is contained in $U_0$, and the same is true of one of the connected components of $S \setminus \zeta$,
		\item \label{item1andahalf:sierp}$\eta \sub U_1$, and one of the connected components of $S \setminus \eta$ is contained in $U_1$,
		\item \label{item2:sierp} $\zeta$ separates $\theta$ from $x_1$ and $\eta$,
		\item \label{item3:sierp} $\theta$ separates $W_0$ from $\zeta$.
	\end{enumerate}

	\begin{figure}[h!]
		\label{fig:sierpi}
		\centering{
			\includeinkscape[scale=1]{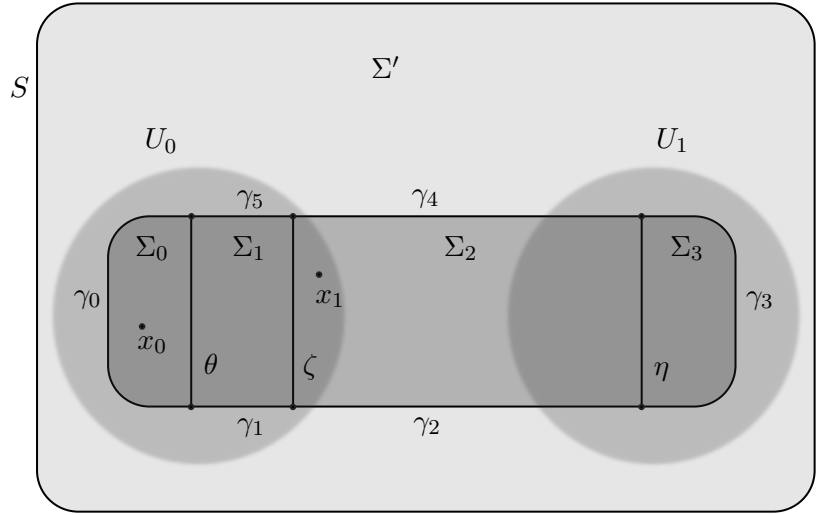}
		}
		\caption{The construction in the proof of \Cref{thm:minimality-sierpinski}.}
	\end{figure}
	
	The endpoints of $\theta, \zeta, \eta$ separate $\gamma$ into $6$ arcs $\gamma_0, \dots, \gamma_5$, such that
	\begin{align*}
		\sigma_0 & \coloneqq \theta \cup \gamma_0,                                    \\
		\sigma_1 & \coloneqq \theta \cup \gamma_1 \cup \zeta \cup \gamma_5,           \\
		\sigma_2 & \coloneqq \zeta \cup \gamma_2 \cup \eta \cup \gamma_4, \text{ and} \\
		\sigma_3 & \coloneqq \eta \cup \gamma_3
	\end{align*}
	are all simple closed curves.
	Also, let $\sigma_{<0,1>} \coloneqq \gamma_0 \cup \gamma_1 \cup \zeta \cup \gamma_5$, and $\sigma_{<2,3>} \coloneqq \zeta \cup \gamma_2 \cup \gamma_3 \cup \gamma_4$; these too are simple closed curves.
	Moreover, all of them are irrational curves in $S$, and they each separate $S$ into two components, one of which is contained in $\Sigma$.
	Call these components $Z_i$, and their closures $\Sigma_i = Z_i \cup \sigma_i$, for $i \in \set{0,1,2,3, <0,1>, <2,3>}$.
	Then all the $\Sigma_i$'s are $S$-curves, and $\Sigma_{<0,1>} = \Sigma_0 \cup \Sigma_1$ and $\Sigma_{<2,3>} = \Sigma_2 \cup \Sigma_3$.

	By condition \eqref{item3:sierp}, $x_0 \in \Sigma_0$;
	by condition \eqref{item1:sierp}, $\Sigma_{<0,1>} \sub U_0$;
	by condition \eqref{item2:sierp}, $x_1 \in \Sigma_2$;
	and by condition \eqref{item1andahalf:sierp}, $\Sigma_3 \sub U_1$.

	Let $\tilde h_{2,3} \colon \sigma_{<2,3>} \to \sigma_3$ be any orientation preserving homeomorphism which maps $\zeta$ onto $\eta$.
	By \Cref{fact:extension-sierpinski}, $\tilde h_{2,3}$ extends to a homeomorphism $h_{2,3} \colon
		\Sigma_{<2,3>} \to \Sigma_3$.

	Let $\tilde h_{1} \colon \sigma_1 \to \sigma_2$ be any orientation preserving homeomorphism which maps $\theta$ onto $\zeta$ and agrees with $\tilde h_{2,3}$ on $\zeta$.
	By \Cref{fact:extension-sierpinski}, $\tilde h_{1}$ extends to a homeomorphism $h_{1} \colon
		\Sigma_1 \to \Sigma_2$.

	Finally, let $\tilde h_{0} \colon \sigma_0 \to \sigma_{<0,1>}$ be any orientation preserving homeomorphism which agrees with $\tilde h_{1}$ on $\theta$.
	By \Cref{fact:extension-sierpinski}, $\tilde h_{0}$ extends to a homeomorphism $h_{0} \colon
		\Sigma_0 \to \Sigma_{<0,1>}$.

	The homeomorphisms $h_0, h_1, h_{2,3}$ all agree on their common domains, so they give rise to a homeomorphism $h \colon \Sigma \to \Sigma$ that fixes set-wise $\gamma$.
	We can use one last time \Cref{fact:extension-sierpinski} on $\tilde h \coloneqq h \restr{\gamma}$ to obtain a homeomorphism
	$h' \colon \Sigma' \to \Sigma'$ which agrees with $h$ on $\gamma$. Their union $g$ is a homeomorphism of $S$ such that $g(x_0) \in U_0$, $g(x_1) \in U_1$ and
	$g[Z] = Z \sub V$, so in particular $g[C] \subseteq V$, as desired.
\end{proof}

Therefore, we conclude that:

\begin{corollary}
	$\Homeo(S)$, for the Sierpi\'{n}ski carpet $S$, does not have the generic point property. In particular, its universal minimal flow is not metrizable and has no comeager orbit.
\end{corollary}

\bibliography{chains.bib}

\end{document}